\documentclass[11pt,a4paper]{article}

\usepackage[a4paper,margin=1in]{geometry}
\usepackage{amsmath,amssymb,amsthm,mathtools}
\usepackage{enumitem}
\usepackage{aliascnt}
\usepackage{microtype}
\usepackage{xcolor}
\usepackage[T1]{fontenc}
\usepackage{lmodern}

\usepackage[nameinlink,noabbrev]{cleveref}

\numberwithin{equation}{section}

\newtheoremstyle{hzplain}%
  {6pt}{6pt}{\itshape}{}{\bfseries}{}{0.5em}{}
\theoremstyle{hzplain}
\newtheorem{theorem}{Theorem}[section]

\newaliascnt{proposition}{theorem}
\newtheorem{proposition}[proposition]{Proposition}
\aliascntresetthe{proposition}

\newaliascnt{definition}{theorem}
\newtheorem{definition}[definition]{Definition}
\aliascntresetthe{definition}

\newaliascnt{lemma}{theorem}
\newtheorem{lemma}[lemma]{Lemma}
\aliascntresetthe{lemma}

\newaliascnt{corollary}{theorem}
\newtheorem{corollary}[corollary]{Corollary}
\aliascntresetthe{corollary}

\newtheoremstyle{hzremark}%
  {6pt}{6pt}{\normalfont}{}{\bfseries}{}{0.5em}{}
\theoremstyle{hzremark}
\newaliascnt{remark}{theorem}
\newtheorem{remark}[remark]{Remark}
\aliascntresetthe{remark}
\newtheorem*{remark*}{Remark}

\crefname{theorem}{theorem}{theorems}
\Crefname{theorem}{Theorem}{Theorems}
\crefname{proposition}{proposition}{propositions}
\Crefname{proposition}{Proposition}{Propositions}
\crefname{lemma}{lemma}{lemmas}
\Crefname{lemma}{Lemma}{Lemmas}
\crefname{corollary}{corollary}{corollaries}
\Crefname{corollary}{Corollary}{Corollaries}
\crefname{remark}{remark}{remarks}
\Crefname{remark}{Remark}{Remarks}

\newcommand{\R}{\mathbb R}
\newcommand{\T}{\mathbb T}
\newcommand{\D}{\mathcal D}
\newcommand{\eps}{\varepsilon}

\definecolor{revisiongreen}{RGB}{0,88,52}

\title{\bfseries  Bistable traveling fronts in strong shear flows: speed and profile asymptotics }
\author{
Weiwei Ding\thanks{School of Mathematical Sciences, South China Normal University, Guangzhou 510631, China (\texttt{dingweiwei@m.scnu.edu.cn}). W. Ding is partially supported by NSFC (12471197) and Guangdong
	Basic and Applied Basic Research Foundation (2023B1515020034).}
\and
Mingmin Zhang\thanks{School of Mathematical Sciences, University of
Science and Technology of China, Hefei, Anhui 230026, China
(\texttt{mingmin.zhang.math@gmail.com}). M. Zhang is supported by the
French ANR ReaCh (ANR-23-CE40-0023) project.}
\and
Zhaoyun Zhang\thanks{School of Mathematics and Big Data, Anhui University
of Science and Technology, Huainan, Anhui 232001, China
(\texttt{zhyzhang@aust.edu.cn}).}
}

\date{}

\begin{document}

\maketitle

\begin{abstract}

We study the strong-shear limit of bistable traveling fronts in infinite cylinders with periodic transverse boundary conditions. We prove that, for every sufficiently regular periodic shear profile, the front speed normalized by the flow amplitude converges as the amplitude tends to infinity. After suitable longitudinal rescaling and translations, the corresponding profiles converge along subsequences to full fronts of the limiting degenerate equation. Furthermore, under a H\"ormander-type non-degeneracy condition on the shear, the limiting front is proved to be regular and unique up to translation, and the whole normalized family converges uniformly. A main difficulty is that the sign-changing bistable reaction allows the limiting transition to split through intermediate transverse equilibria. We rule out this possibility by exploiting the instability of such equilibria together with suitable regularization and comparison arguments.
Finally, in contrast with the combustion case, where every nonconstant mean-zero shear yields a positive limiting speed, we construct an example showing that, for a fixed bistable reaction, smooth mean-zero shears can produce negative, zero, or positive limiting speeds. In particular, the example shows that a shear which accelerates propagation at small amplitudes may reverse the propagation direction when its amplitude becomes large.


\end{abstract}

\section{Introduction and main results}

Traveling fronts are fundamental objects in the study of propagation
phenomena for reaction-diffusion equations. For bistable reactions,
they describe the invasion of one stable state by another, and their
existence, stability and qualitative properties have been extensively
studied in cylinders and spatially periodic media; see
\cite{Berestycki2003,Xin2000} for general accounts and
\cite{BerestyckiHamel2000,DingGiletti2021,DingHamelLiang2025,DingHamelZhao2017,Ducrot2016,DucrotGilettiMatano2014,Roquejoffre1997,Xin1991,Xin1993,XinZhu1995} and references therein
for related results. In the presence of a flow, advection interacts with
diffusion and reaction and may affect both the speed and the shape of
a front. This leads naturally to the question of how propagation
behaves as the strength of the flow increases.

In this paper, we study the effect of strong periodic shear flows on
bistable traveling fronts. More precisely, we consider
\begin{equation}
	v_t+A\alpha(y)v_x=\Delta v+f(v),
	\qquad t\in\R,\quad (x,y)\in\R\times\R^{N-1},
	\label{model}
\end{equation}
where $N\geq2$, $A\geq0$ is the flow amplitude, and $\alpha$ is a
$(1,\ldots,1)$-periodic function of class $C^{1,\delta}$, for some
$\delta\in(0,1)$. Thus, the advection field is the incompressible shear
flow $Aq$, with
\begin{equation*}
	q(x,y)=(\alpha(y),0,\ldots,0).
\end{equation*}
We identify periodic functions of $y$ with functions on
$\T^{N-1}=\R^{N-1}/\mathbb Z^{N-1}$, and denote the mean of the shear by
\begin{equation*}
	\bar\alpha:=\int_{\T^{N-1}}\alpha(y)\,dy.
\end{equation*}
Throughout the paper, the reaction $f\in C^{1,\delta}([0,1])$ is assumed
to be of bistable type: there exists $\theta\in(0,1)$ such that
\begin{equation}
	\begin{cases}
		f(0)=f(\theta)=f(1)=0,~~f<0~\hbox{in }(0,\theta),~~
		f>0~\hbox{in }(\theta,1),\vspace{5pt}\\
		f'(0)<0,~~f'(\theta)>0,~~f'(1)<0,~~
		\displaystyle\int_0^1f(s)\,ds>0.
	\end{cases}
	\label{eq:bistable}
\end{equation}

Equation \eqref{model} can be interpreted as a population model with
a strong Allee effect in a prescribed shear flow; see, for instance,
\cite{rghk,HeinzePapanicolaouStevens2001}. Here, $v$ denotes the
population density normalized by the carrying capacity, while diffusion
and advection describe random dispersal and transport by the flow,
respectively. The stable equilibria $0$ and $1$ represent extinction and
the carrying capacity, and the unstable equilibrium $\theta$ is the
Allee threshold separating population decline from growth in the absence
of spatial transport.

We are interested in traveling fronts of the form
$v(t,x,y)=u_A(x-c_At,y)$,
where the profile $u_A$ is periodic in $y$ and connects $1$ to $0$.
Berestycki and Nirenberg \cite{BerestyckiNirenberg1992}
studied the existence and uniqueness of traveling fronts
in infinite cylinders with shear flows and homogeneous
Neumann boundary conditions. Their results
(Theorems~1.7-1.8), adapted to periodic boundary conditions
in $y$, yield, for each fixed $A\geq0$, a unique speed
$c_A=c_A(\alpha,f)$ and a unique (up to translations in $x$)
profile $u_A:\R\times\T^{N-1}\to(0,1)$ satisfying
\begin{equation}
	\left\{
	\begin{aligned}
		&\Delta u_A+
		\bigl(c_A-A\alpha(y)\bigr)u_{A,x}+f(u_A)=0
		&&\hbox{in }\R\times\T^{N-1},\\
		&0\equiv u_A(+\infty,\cdot)<u_A
		<u_A(-\infty,\cdot)\equiv1
		&&\hbox{in }\R\times\T^{N-1},
	\end{aligned}
	\right.
	\label{eqn_TW_A}
\end{equation}
where the two limits $0$ and $1$ are uniform in $y\in\T^{N-1}$.
 Moreover, $u_{A,x}<0$ on $\R\times\T^{N-1}$. 

In the absence of advection ($A=0$), the front is planar and reduces to
the classical one-dimensional bistable wave studied by Aronson and
	Weinberger \cite{AW78} and by Fife
and McLeod~\cite{FM77}. More precisely, $u_0(x,y)=q(x)$,
where the unique speed $c_0$ and the profile $q$, unique up to
translations, satisfy
\begin{equation*}
	q''+c_0q'+f(q)=0\quad\hbox{in }\R,
	\qquad q(-\infty)=1,\quad q(+\infty)=0,
\end{equation*}
with $q'<0$ in $\R$. Moreover, the positive-integral assumption in \eqref{eq:bistable} ensures
that $c_0>0$, so that the state $1$ invades the state $0$ in the positive
$x$-direction. The case of a negative integral can be reduced to this
one by reversing the longitudinal variable and interchanging the two
stable states.

In the presence of a shear flow, a natural question is how the speed
and profile of the front depend on the amplitude $A$. The influence of
flows on propagation has been investigated through numerical studies
of autocatalytic fronts \cite{AllenBrindleyMerkinPilling1996}, asymptotic
analyses of rapidly varying flows \cite{AudolyBerestyckiPomeau2000},
and estimates for strong-advection enhancement
\cite{KiselevRyzhik2001,Zlatos2011}. Variational characterizations of
multidimensional propagation speeds provide another approach to this
question \cite{Hamel1999,HeinzePapanicolaouStevens2001}. 
For bistable reactions in layered media and
shear flows, Papanicolaou and Xin \cite{PapanicolaouXin1991} derived
perturbative speed formulas, including a two-dimensional bistable shear
model. Related small-amplitude asymptotics, large-amplitude estimates,
numerical results and probabilistic bounds were obtained for space-time
periodic and stationary Gaussian shears in
\cite{NolenXin2003,NolenXin2004}.

For steady shear flows, Heinze, Papanicolaou and Stevens
\cite{HeinzePapanicolaouStevens2001} derived min-max characterizations
of propagation speeds under appropriate existence and stability
assumptions and applied them to bistable and combustion fronts in
cylinders. Their analysis provides quantitative speed bounds and a
rigorous second-order expansion for small shear amplitudes.
More precisely, for a smooth nonconstant mean-zero shear $\alpha$ and a smooth
bistable reaction satisfying \eqref{eq:bistable}, the corresponding
expansion in the present periodic setting takes the form
\begin{equation*}
	c_A(\alpha,f)=c_0+\kappa(\alpha,f)A^2+o(A^2)
	\qquad\text{as }\,\, A\to0,
\end{equation*}
where $\kappa(\alpha,f)>0$. Thus, every such shear increases the
propagation speed when its amplitude is sufficiently small. This
perturbative result, however, does not determine the behavior of the
front in the strong-shear regime $A\to+\infty$, or whether the
small-amplitude speed-up persists as the flow strength increases.

A first observation in the strong-flow regime is that the speed grows
at most linearly with $A$. Indeed, for every $k\in\R$, replacing $\alpha$ by $\alpha+k$ and $c_A$
by $c_A+Ak$ leaves the profile equation unchanged, and hence, by uniqueness, 
\begin{equation}
	c_A(\alpha+k,f)=c_A(\alpha,f)+Ak
	\qquad\text{for all }\,\,k\in\R.
	\label{eq:constant-shear-shift}
\end{equation}
Furthermore, as shown in Lemma~\ref{lem:order-speed} below, if
$\alpha\leq\beta$ in $\T^{N-1}$, then
$c_A(\alpha,f)\leq c_A(\beta,f)$ for every $A\geq0$.
Combining this ordering with \eqref{eq:constant-shear-shift}
gives
\begin{equation}
	\big|c_A(\alpha,f)-c_A(\beta,f)\big|
	\leq A\|\alpha-\beta\|_{L^\infty(\T^{N-1})}.
	\label{Lip_speed}
\end{equation}
In particular, comparison with the constant shears
$\min_{\T^{N-1}}\alpha$ and $\max_{\T^{N-1}}\alpha$ yields
\begin{equation}
	A\min_{\T^{N-1}}\alpha+c_0
	\leq c_A(\alpha,f)
	\leq A\max_{\T^{N-1}}\alpha+c_0.
	\label{speed-bounds}
\end{equation}
These bounds ensure that $c_A/A$ remains bounded as $A\to+\infty$,
but they do not imply its convergence.

To study this normalized speed together with the corresponding
profiles, we introduce the longitudinal rescaling
\begin{equation*}
	U_A(x,y):=u_A(Ax,y).
\end{equation*}
Then,
\begin{equation*}
	A^{-2}U_{A,xx}+\Delta_yU_A
	+\left(\frac{c_A}{A}-\alpha(y)\right)U_{A,x}+f(U_A)=0
	\quad\hbox{in }\R\times\T^{N-1}.
\end{equation*}
As $A\to+\infty$, the longitudinal diffusion coefficient tends to
zero. Thus, any limiting equation associated with a subsequential
limit of $c_A/A$ is degenerate in the propagation direction. The
strong-shear problem therefore concerns not only the asymptotic speed,
but also the persistence of the transition from $1$ to $0$ under this
loss of ellipticity.

For nonnegative combustion reactions, Hamel and Zlato\v{s}
\cite{HamelZlatos2013} proved convergence of the normalized speeds and,
for general $C^{1,\delta}$ periodic shears, subsequential
almost-everywhere convergence of suitably translated rescaled profiles
to full distributional fronts of the limiting equation. Under an
additional H\"ormander-type condition on the shear (see \eqref{cdn_alpha} below), they also obtained
regularity and uniqueness up to translation of the limiting
speed-profile pair. Moreover, for a nonconstant shear, the limiting
speed is strictly greater than the mean flow, giving an asymptotically
linear speed-up relative to that mean.

In contrast, the bistable reaction changes sign, so its spatial
integral allows cancellation and no longer directly controls the
transition region. Local compactness of the rescaled profiles is
therefore insufficient to recover a front connecting the two stable
states. In particular, a limiting transition could a priori split into
two fronts through an intermediate transverse equilibrium. The sign
change of the reaction also raises a question about the asymptotic
speed: the speed-up produced by a weak mean-zero shear need not persist
when the same shear becomes strong. Our purpose is to determine the
strong-shear limit of bistable fronts and to clarify these differences
from the combustion case.

We first introduce the notion of a full front for the limiting equation.
\begin{definition}
	A full degenerate front $U$ with speed $\gamma\in\R$ is a bounded
	distributional solution of
	\begin{equation}
		\left\{
		\begin{aligned}
			&\Delta_yU+(\gamma-\alpha(y))U_x+f(U)=0
			&&\text{in }\D'(\R\times\T^{N-1}),\\
			&0\leq U\leq1
			&&\text{a.e. in }\R\times\T^{N-1},\\
			&U(-\infty,\cdot)=1,\qquad U(+\infty,\cdot)=0
			&&\hbox{uniformly in }\T^{N-1}.
		\end{aligned}
		\right.
		\label{eq:main-degenerate-front}
	\end{equation}
\end{definition}

To fix longitudinal translations, for a bounded function $V$ on
$\R\times\T^{N-1}$, we set
\begin{equation*}
	\mathcal N(V):=\int_0^1\int_{\T^{N-1}}V(x,y)\,dy\,dx.
\end{equation*}
For a multi-index
$\zeta=(\zeta_1,\ldots,\zeta_{N-1})\in\mathbb N_0^{N-1}$, where
$\mathbb N_0:=\{0,1,2,\ldots\}$, we write
$|\zeta|:=\zeta_1+\cdots+\zeta_{N-1}$ and
$D^\zeta:=\partial_{y_1}^{\zeta_1}\cdots
\partial_{y_{N-1}}^{\zeta_{N-1}}$.
Throughout the paper, $|E|$ denotes the Lebesgue measure of a measurable
set $E\subset\T^{N-1}$. Our first main result is the following theorem.

\begin{theorem}\label{thm:main}
	Assume that $\alpha\in C^{1,\delta}(\T^{N-1})$. Then, the following
	statements hold.
	
	\smallskip\noindent
	\textup{(i)} There exists
	$\gamma^*(\alpha,f)\in[\min_{\T^{N-1}}\alpha,
	\max_{\T^{N-1}}\alpha]$ such that
	\begin{equation}
		\lim_{A\to+\infty}\frac{c_A}{A}=\gamma^*(\alpha,f).
		\label{eq:speed-limit}
	\end{equation}
	The lower bound is strict if
	$|\{y\in\T^{N-1}:\alpha(y)=\min_{\T^{N-1}}\alpha\}|=0$,
	and the upper bound is strict if
	$|\{y\in\T^{N-1}:\alpha(y)=\max_{\T^{N-1}}\alpha\}|=0$.
	
	Furthermore, setting $V_A(x,y)=U_A(x+\tau_A,y)$, where $\tau_A$ is the unique real
	number satisfying $\mathcal N(V_A)=1/2$, we have
	\begin{equation}
		\lim_{A\to+\infty}A^{-2}
		\int_{\R\times\T^{N-1}}|V_{A,x}|^2\,dx\,dy=0.
		\label{eq:main-longitudinal-energy}
	\end{equation}
	Every sequence $A_n\to+\infty$ has a subsequence $($still denoted
	by $A_n$$)$ along which $V_{A_n}$ converges a.e. in
	$\R\times\T^{N-1}$ to a full degenerate front $U$ with speed
	$\gamma=\gamma^*(\alpha,f)$, satisfying
	\begin{equation*}
		U_x\leq0\quad\hbox{in }\D'(\R\times\T^{N-1}),
		\qquad \nabla_yU\in L^2(\R\times\T^{N-1}),
		\qquad \mathcal N(U)=1/2,
	\end{equation*}
	and for every $1\leq p<+\infty$,
	\begin{equation}
		\|V_{A_n}-U\|_{L^p(\R\times\T^{N-1})}
		+\|\nabla_yV_{A_n}-\nabla_yU\|_{L^2(\R\times\T^{N-1})}
		\to0.
		\label{eq:general-profile-convergence}
	\end{equation}
	Moreover, the function $f(U)$ belongs to $L^1(\R\times\T^{N-1})$, and
	\begin{equation}
		\begin{aligned}
			\gamma^*(\alpha,f)-\bar\alpha
			&=\int_{\R\times\T^{N-1}}f(U)\,dx\,dy,\\
			\int_{\R\times\T^{N-1}}|\nabla_yU|^2\,dx\,dy
			&=\int_{\R\times\T^{N-1}}f(U)(U-\tfrac12)\,dx\,dy.
		\end{aligned}
		\label{eq:general-identities}
	\end{equation}
	
	\smallskip\noindent
	\textup{(ii)} Suppose in addition that, for some positive integer $r$,
	\begin{equation}
		\alpha\in C^\infty(\T^{N-1}),\qquad
		\sum_{1\leq|\zeta|\leq r}|D^\zeta\alpha(y)|>0
		\quad\text{for every }y\in\T^{N-1}.
		\label{cdn_alpha}
	\end{equation}
	Then, the full degenerate front $(\gamma^*(\alpha,f),U)$ is unique up to
	translation of $U$ in $x$, satisfying
	\begin{equation*}
		\min_{\T^{N-1}}\alpha<\gamma^*(\alpha,f)<\max_{\T^{N-1}}\alpha.
	\end{equation*}
	Furthermore,
	$U\in C^{1,\beta}_{\mathrm{loc}}(\R\times\T^{N-1})$ for some
	$\beta\in(0,\delta)$, with $0<U<1$, $U_x<0$, and
	$\nabla_yU\in L^2(\R\times\T^{N-1})\cap
	L^\infty(\R\times\T^{N-1})$.
	As $A\to+\infty$, the whole normalized family $V_A$ converges
	uniformly on $\R\times\T^{N-1}$ to the unique front satisfying
	$\mathcal N(U)=1/2$, and, for every $1\leq p\leq+\infty$,
	\begin{equation}
		\|V_A-U\|_{L^p(\R\times\T^{N-1})}
		+\|\nabla_yV_A-\nabla_yU\|_{L^2(\R\times\T^{N-1})}
		\to0.
		\label{eq:regular-profile-convergence}
	\end{equation}
\end{theorem}

\begin{remark}
	Condition \eqref{cdn_alpha} is the same non-degeneracy assumption
	used in \cite{HamelZlatos2013}. It ensures that the vector fields
	associated with
	\begin{equation*}
		L_\gamma:=\Delta_y+(\gamma-\alpha(y))\partial_x
	\end{equation*}
	and their commutators up to order $r+1$ span all spatial directions,
	sand hence, $L_\gamma$ satisfies H\"ormander's hypoellipticity condition
	\cite{Hormander1967}.
	
	Without this assumption, the limiting fronts may be discontinuous and
	uniqueness up to translation may fail. For instance, for a constant shear
	$\alpha\equiv\alpha_0$, the normalized rescaled profiles satisfies 
	$\|V_A-\mathbf1_{\{x<1/2\}} \|_{L^p(\R\times\T^{N-1})}\to 0$ as $A\to+\infty$ for every $1\leq p<\infty$. 
	At the same speed $\gamma=\alpha_0$, the limiting equation
	also admits profiles with arbitrary plateaus at $\theta$, which
	are not all translates of one another. Thus, convergence of the
	rescaled profiles does not require uniqueness in the larger class
	of distributional full fronts. This constant-shear case is discussed
	in Section~\ref{subsec:constant}.
\end{remark}

\begin{remark}
	The limiting speed inherits the comparison and scaling properties
	of the finite-amplitude problem. More precisely, for
	$\alpha,\beta\in C^{1,\delta}(\T^{N-1})$ and $k\in\R$, we have
	\begin{equation}
		\begin{aligned}
			\left|\gamma^*(\alpha,f)-\gamma^*(\beta,f)\right|
			&\leq\|\alpha-\beta\|_{L^\infty(\T^{N-1})},\\
			\gamma^*(\alpha+k,f)&=\gamma^*(\alpha,f)+k.
		\end{aligned}
		\label{eq:speed-Lipschitz}
	\end{equation}
	More generally, for every $t>0$, there holds
	\begin{equation*}
		\gamma^*(t\alpha+k,f)=t\gamma^*(\alpha,f)+k.
	\end{equation*}
	Indeed, this follows by dividing
	$c_A(t\alpha+k,f)=c_{tA}(\alpha,f)+Ak$ by $A$ and passing to the limit.
\end{remark}

We next examine the sign of the limiting speed and its
consequences for propagation. By \eqref{eq:speed-limit},
\begin{equation*}
	c_A(\alpha,f)-A\bar\alpha
	=A\bigl(\gamma^*(\alpha,f)-\bar\alpha\bigr)+o(A)
	\qquad\text{as }\,\,A\to+\infty.
\end{equation*}
Moreover, \eqref{eq:speed-Lipschitz} gives
$\gamma^*(\alpha,f)-\bar\alpha
=\gamma^*(\alpha-\bar\alpha,f)$.
Thus, to study the limiting speed relative to the mean flow,
it suffices to consider mean-zero shears.

For nonconstant mean-zero shears, Hamel and Zlato\v{s}
\cite{HamelZlatos2013} proved that
$\gamma^*(\alpha,f)>0$ for combustion reactions.
For bistable reactions, it is natural to ask whether the
positive-integral assumption in \eqref{eq:bistable},
which guarantees $c_0>0$, also implies
$\gamma^*(\alpha,f)>0$.
For mean-zero shears, the first identity in
\eqref{eq:general-identities} reads
$\gamma^*(\alpha,f)
	=\int_{\R\times\T^{N-1}}f(U)\,dx\,dy$.
Unlike $\int_0^1f(u)\,du$, which depends only on $f$,
this spatial integral also depends on the limiting
profile $U$. 
The shear remains present in the limiting equation and
can influence the spatial distribution of $U$, thereby
affecting the relative magnitudes of the positive reaction
contribution from $\{\theta<U<1\}$ and the negative
contribution from $\{0<U<\theta\}$.
Thus, $\int_0^1f(u)\,du>0$
 no longer determines the sign of this spatial integral.

The following example shows that, for one fixed bistable
reaction satisfying \eqref{eq:bistable}, the limiting
speed can be negative, zero, or positive as the mean-zero
shear profile varies.
Let $N=2$, and, for $0<\eps<1$ and $s\in[-1,1]$, set
$\theta_\eps=(1-\eps)/2$ and
\begin{equation}
	f_\eps(u)=2\eps u(1-u)(u-\theta_\eps),
	\qquad
	\alpha_s(y)=\cos(2\pi y)+s\cos(4\pi y).
	\label{eq:example-data}
\end{equation}
It is straightforward to check that $f_\eps$ satisfies
\eqref{eq:bistable}, with
$\int_0^1f_\eps(u)\,du=\eps^2/6>0$, and that each
$\alpha_s$ is smooth, nonconstant, has zero mean, and
satisfies \eqref{cdn_alpha}. Moreover, for each
fixed $\eps$, varying $s$  changes only the shear profile,
while leaving the reaction unchanged.

\begin{proposition}\label{prop:opposite-signs}
	For the family defined in \eqref{eq:example-data}, there exist
	$C>0$ and $\eps_0\in(0,1)$ such that
	\begin{equation}
		\left|
		\gamma^*(\alpha_s,f_\eps)
		+\frac{3s}{20\pi^2(4+s^2)}\eps
		\right|
		\leq C\eps^{3/2}
		\label{eq:example-speed-expansion}
	\end{equation}
	for all $s\in[-1,1]$ and $0<\eps<\eps_0$. Furthermore, for each fixed $\eps\in(0,\eps_0)$, the map
	$s\mapsto\gamma^*(\alpha_s,f_\eps)$ is continuous, and there
	exists $s_\eps\in(-1,1)$ such that
	\begin{equation*}
		\gamma^*(\alpha_1,f_\eps)
		<\gamma^*(\alpha_{s_\eps},f_\eps)=0
		<\gamma^*(\alpha_{-1},f_\eps).
	\end{equation*}
\end{proposition}

 The proof of Proposition~\ref{prop:opposite-signs} relies on an explicit
construction of monotone comparison profiles. We supplement a slowly
varying one-dimensional transition with periodic transverse correctors,
which cancel the
leading terms in the profile equation, leaving a residual that is small
relative to the longitudinal derivative of the comparison profile. A
sliding argument at finite $A$ then converts this residual estimate
into a two-sided bound for the normalized speed $c_A/A$. Passing to
the limit $A\to+\infty$, with $\eps$ fixed, yields
\eqref{eq:example-speed-expansion}, with an error uniform in $s$.

Combining Proposition~\ref{prop:opposite-signs} with the small-amplitude
expansion recalled above (see also \cite{HeinzePapanicolaouStevens2001}) yields the following consequences for the
dependence of the propagation speed on the shear amplitude.

\begin{corollary}\label{cor:flow-effects}
	Fix $0<\eps<\eps_0$ as in Proposition~{\rm \ref{prop:opposite-signs}},
	and denote by $c_0>0$ the speed without advection.
	\begin{enumerate}
		\item [{\rm (i)}]
		For the shear $\alpha_1$, the speed $c_A=c_A(\alpha_1,f_\eps)$
		satisfies $c_A>c_0$ for all sufficiently small $A>0$, whereas
		$c_A\to-\infty$ as $A\to+\infty$. Consequently, the map
		$A\mapsto c_A$ is not monotone, and there exist $A_d,A_*>0$
		such that
		\begin{equation*}
			0<c_{A_d}<c_0,\qquad c_{A_*}=0.
		\end{equation*}
		
		\item [{\rm (ii)}]
		There exists a smooth nonconstant mean-zero shear $\tilde{\alpha}$
		satisfying \eqref{cdn_alpha} such that
		\begin{equation*}
			\gamma^*(\tilde{\alpha},f_\eps)=0,\qquad
			c_A(\tilde{\alpha},f_\eps)=o(A)
			\quad\text{as }A\to+\infty.
		\end{equation*}
		Moreover, the corresponding normalized rescaled profiles converge
		uniformly on $\R\times\T$ to a full front of
		\begin{equation*}
			U_{yy}-\tilde{\alpha}(y)U_x+f_\eps(U)=0
			\quad\hbox{in }\R\times\T.
		\end{equation*}
	\end{enumerate}
\end{corollary}

Corollary~\ref{cor:flow-effects} shows that a fixed mean-zero shear
can accelerate a bistable front at small amplitudes, slow it down at
intermediate amplitudes, and reverse its direction at large amplitudes.
This behavior differs from the asymptotically linear speed-up of
combustion fronts and shows that the small-amplitude enhancement of
bistable propagation does not extend to arbitrary flow strengths.

We finally turn to describe the main difficulties and the ideas underlying the
proof of Theorem~\ref{thm:main}. The first difficulty is to recover the
full connection from $1$ to $0$ in the degenerate limit. Local estimates
uniform under longitudinal translations provide compactness of the
rescaled profiles, but do not identify their end states. To this end,
we consider the set of transverse equilibria
\begin{equation*}
	\mathcal E_f:=\big\{w\in C^2(\T^{N-1}):0\leq w\leq1,
	\ \Delta_yw+f(w)=0\ \hbox{in }\T^{N-1}\big\}.
\end{equation*}
Using longitudinal convolution and transverse elliptic estimates, we
show that the monotone weak limits obtained from the rescaled profiles
approach elements of $\mathcal E_f$ uniformly as $x\to\pm\infty$.
The order properties of $\mathcal E_f$ then reduce the possible limiting
configurations to either a full front or a two-front terrace connecting
$1$ to some $w\in\mathcal E_f\setminus\{0,1\}$ and $w$ to $0$.
Both components of such a terrace inherit the same limiting speed,
since they arise from different translations of the same sequence.

For nonconstant shears, we exclude the terrace alternative by exploiting
the instability of every intermediate equilibrium. More precisely, for
$w\in\mathcal E_f$, let $\lambda_1(w)$ denote the principal eigenvalue
of the self-adjoint operator $\mathcal A_w:=-\Delta_y-f'(w)$ on
$\T^{N-1}$. Then, $\lambda_1(w)<0$ for every
$w\in\mathcal E_f\setminus\{0,1\}$; see Proposition~\ref{prop_Ef}.
The limiting profiles need not be continuous or strictly separated
from $w$, so the argument must apply directly to weak nonnegative
differences. Exponential convolution in the longitudinal variable
produces positive regularizations for which transverse elliptic Harnack
estimates are available. A weighted logarithmic identity then gives
incompatible conditions on the common speed, ruling out terraces through
both constant and nonconstant intermediate equilibria. Constant shears
are treated separately by using the explicit form of their fronts.

Once the full transition has been recovered, a second difficulty is to
pass from local to global convergence. Longitudinal smoothing and
monotonicity first provide control of the approximating profiles on
fixed transverse sections. The stability of $0$ and $1$ then yields
exponential tail estimates uniform along the selected subsequence,
giving convergence of the profile differences in global
$L^p$ norm for every $1\leq p<+\infty$. Exact energy identities, justified for the weak limit by
longitudinal convolution, further yield strong convergence of the
transverse gradients and the vanishing of the longitudinal diffusion
energy. These arguments do not require continuity of the limiting front.

Finally, convergence of the normalized speeds is first proved for
shears satisfying \eqref{cdn_alpha}, using regularity and uniqueness
of full degenerate fronts. Uniform approximation of the shear, together
with the Lipschitz estimate \eqref{Lip_speed}, then extends speed
convergence to every $\alpha\in C^{1,\delta}(\T^{N-1})$, without
requiring uniqueness of distributional fronts for the original shear.
Under \eqref{cdn_alpha}, uniqueness also gives convergence of the whole
normalized family. The strict speed bounds under the assumption that the sets where $\alpha$ attains its minimum
and maximum have zero Lebesgue measure follow from the signs of $f$ near the stable states and a
weighted parabolic uniqueness argument.

\medskip\noindent\textbf{Outline of the paper.}
Section~\ref{sec2} studies the transverse equilibria, proves comparison
of finite-amplitude speeds, and derives local compactness and energy
convergence for the rescaled profiles. In Section~\ref{sec3}, we prove
regularity, strict monotonicity and uniqueness of full degenerate
fronts under \eqref{cdn_alpha}. Section~\ref{sec4} classifies translated
weak limits and excludes the two-front terrace alternative for
nonconstant shears. Section~\ref{sec5} proves Theorem~\ref{thm:main},
including global profile and energy convergence and the strict speed
bounds. Finally, in Section~\ref{sec6}, we use explicit comparison profiles and a
sliding argument to prove Proposition~\ref{prop:opposite-signs}.

\section{Transverse equilibria and strong-shear compactness}
\label{sec2}

In this section, we first establish the order properties of the transverse periodic
equilibria and the instability of those other than $0$ and $1$.
These properties will be used in Section~\ref{sec4} to classify
the limiting profiles and exclude the terrace alternative for
nonconstant shears. We then prove the ordering of wave speeds with respect to the
shear profile at each fixed amplitude.  After
rescaling in $x$, we derive local estimates uniform under
longitudinal translations and use them to prove  compactness of
the rescaled fronts and convergence of their local transverse energy.

\subsection{Transverse equilibria}

For $p,q\in \mathcal{E}_f$, we write $p\succ q$ when $p\geq q$ and
$p\not\equiv q$. 

\begin{proposition}
	\label{prop_Ef}
	The set $\mathcal{E}_f$ has the following properties:
	\begin{itemize}
		\item[(i)] 
		
		If $p,q\in\mathcal E_f$ and $p\succ q$, then
		\begin{equation*}
			\min_{\T^{N-1}}p>\max_{\T^{N-1}}q.
		\end{equation*}
		\item[(ii)]
		Every nonconstant $w\in\mathcal E_f$ is unstable in the sense that the principal eigenvalue of $-\Delta_y-f'(w)$ on $\T^{N-1}$ is negative, and satisfies
		\begin{equation*}
			0<\min_{\T^{N-1}}w<\theta<\max_{\T^{N-1}}w<1.
		\end{equation*}

		\item[(iii)]  A strictly ordered chain in $\mathcal E_f$ connecting $1$ to
		$0$ is one of
		$1\succ0$, $1\succ\theta\succ0$,
		$ 1\succ w\succ0,$
		where $w$ is nonconstant in the last case.
	\end{itemize}
\end{proposition}
\begin{proof}
	Let $p,q\in \mathcal{E}_f$ satisfy $p\succ q$. The function
	$v:=p-q\ge0$ satisfies
	$$
	\Delta_yv+c(y)v=0
	\quad\text{in }\T^{N-1},
	$$
	where
	$$
	c(y):=\int_0^1
	f'\bigl(q(y)+t(p(y)-q(y))\bigr)\,dt
	\in L^\infty(\T^{N-1}).
	$$
	The elliptic Harnack inequality for nonnegative solutions of linear
	equations with bounded zeroth-order coefficient implies that either
	$v\equiv0$ or $v>0$ on $\T^{N-1}$. Since
	$p\not\equiv q$, the latter alternative holds, and hence
	$p>q$ on $\T^{N-1}$.
	
	For $a\in\T^{N-1}$, set $q_a(y):=q(y+a)$ and
	$$
	S:=\left\{
	a\in\T^{N-1}:
	q_a(y)<p(y)\ \text{for every }y\in\T^{N-1}
	\right\}.
	$$
	Since $q<p$, one has $0\in S$. The set $S$ is open by
	compactness of $\T^{N-1}$. To prove that it is closed,
	let $a_n\in S$ and $a_n\to a$. Then,  $q_a\le p$, and
	$p-q_a$ satisfies
	$$
	\Delta_y(p-q_a)+c_a(y)(p-q_a)=0,
	$$
	with
	$$
	c_a(y):=\int_0^1
	f'\bigl(q_a(y)+t(p(y)-q_a(y))\bigr)\,dt.
	$$
	The same Harnack alternative implies that either $p-q_a>0$
	everywhere or $p-q_a\equiv0$. The latter is impossible, because
	$$
	\int_{\T^{N-1}}p\,dy
	>
	\int_{\T^{N-1}}q\,dy
	=
	\int_{\T^{N-1}}q_a\,dy.
	$$
	Thus,  $a\in S$, so $S$ is also closed. Since
	$\T^{N-1}$ is connected, $S=\T^{N-1}$.
	
	Let $y_-$ be a minimum point of $p$, and let $z_+$ be a
	maximum point of $q$. Taking $a=z_+-y_-$ in the preceding
	strict inequality gives
	$$
	\max_{\T^{N-1}}q
	=q_a(y_-)
	<p(y_-)
	=\min_{\T^{N-1}}p.
	$$
	This proves~(i).
	 
	Let now $w\in \mathcal{E}_f$ be nonconstant. By standard elliptic regularity, $
	w\in C^{3,\delta}(\T^{N-1})$. 
	There exists $j\in\{1,\ldots,N-1\}$ such that
	$\partial_{y_j}w\not\equiv0$. Differentiating
	$\Delta_yw+f(w)=0$ yields
	$$
	\mathcal A_w(\partial_{y_j}w)
	=
	\bigl(-\Delta_y-f'(w)\bigr)\partial_{y_j}w
	=0.
	$$
	Since $\partial_{y_j}w$ is nonzero and has zero mean, it changes
	sign. On the other hand, the principal eigenvalue of
	$\mathcal A_w$ is simple and has a strictly positive
	eigenfunction. Consequently, the eigenvalue $0$ cannot be the
	principal eigenvalue. Since $0$ is an eigenvalue and
	$\lambda_1(w)$ is the lowest eigenvalue, it follows that
	$
	\lambda_1(w)<0$.

	Since $1\succ w\succ0$, part~(i) gives
	$0<\min_{\T^{N-1}}w\leq\max_{\T^{N-1}}w<1$.
	Integrating the stationary equation yields
	$\int_{\T^{N-1}}f(w(y))\,dy=0$.
	If $w\leq\theta$ or $w\geq\theta$ throughout the torus,
	$f(w)$ has one sign, so this identity forces $f(w)\equiv0$.
	Continuity and connectedness would then make $w$ constant.
	Thus, $\min_{\T^{N-1}}w<\theta<\max_{\T^{N-1}}w$,
	which proves~(ii).
	
	By (i)-(ii), no two elements of $\mathcal E_f\setminus\{0,1\}$
	can be strictly ordered: two nonconstant elements would satisfy
	$\min_{\T^{N-1}}p>\max_{\T^{N-1}}q$ and
	$\min_{\T^{N-1}}p<\theta<\max_{\T^{N-1}}q$, while a
	nonconstant element crosses $\theta$.
	Every chain from $1$ to $0$ therefore contains at most one
	intermediate equilibrium, proving~(iii).
\end{proof}

\subsection{Fronts at fixed amplitudes}


We prove the ordering of speeds with respect to the shear at each fixed amplitude.

\begin{lemma}\label{lem:order-speed}
 If
	$\alpha,\beta\in C^{1,\delta}(\T^{N-1})$ and
	$\alpha\leq\beta$ on $\T^{N-1}$, then for each $A\ge 0$,
	\begin{equation*}
		c_A(\alpha,f)\leq c_A(\beta,f).
	\end{equation*}
\end{lemma}
\begin{proof}
		
  Suppose
	$\alpha\leq\beta$, and set
	$u_\alpha=u_A(\alpha,f)$, $u_\beta=u_A(\beta,f)$ and
	$c_\alpha=c_A(\alpha,f)$, $c_\beta=c_A(\beta,f)$, and assume for contradiction
	that $c_\alpha>c_\beta$.  
	
	Put
	\begin{equation*}
		d(y):=c_\alpha-c_\beta+A(\beta(y)-\alpha(y))>0,
		\qquad \mathcal L:=\Delta+(c_\alpha-A\alpha(y))\partial_x.
	\end{equation*}
	Choose $\rho>0$ so that $f$ is strictly decreasing on $[0,\rho]$ and
	$[1-\rho,1]$.	By the uniform limits in \eqref{eqn_TW_A}, there exists $M>0$ such that, for
	$\sigma\in\{\alpha,\beta\}$,
	\begin{equation*}
		u_\sigma(x,y)\leq\frac{\rho}{2}
		\quad\hbox{for }x\geq M,
		\qquad
		u_\sigma(x,y)\geq1-\frac{\rho}{2}
		\quad\hbox{for }x\leq-M.
	\end{equation*}
	
	For $h\geq4M$, define
	\begin{equation*}
		\eps_h=\inf\{\eps\geq0:u_\alpha(x,y)-\eps
		\leq u_\beta(x-h,y)\hbox{ for all }(x,y)\in\R\times\T^{N-1}\}.
	\end{equation*}
	Assume that $\varepsilon_h>0$.  Since
	$u_\alpha(x,\cdot)-u_\beta(x-h,\cdot)\to0$ uniformly in $\T^{N-1}$ as $x\to\pm\infty$, the function
	$V_h:=u_\alpha-\eps_h-u_\beta(\cdot-h,\cdot)$ attains its zero maximum at some point $(x_h,y_h)\in\R\times\T^{N-1}$. 	If $x_h\geq2M$, we have
	$u_\alpha(x_h,y_h)\le \rho/2$ and $u_\beta(x_h-h,y_h)=u_\alpha(x_h,y_h)-\varepsilon_h<\rho/2$.  If $x_h<2M$, it follows that 
$x_h-h<-2M$ and  $u_\beta(x_h-h,y_h)\ge 1-\rho/2$. Therefore,
	$$
	u_\alpha(x_h,y_h)
	=
	u_\beta(x_h-h,y_h)+\varepsilon_h
	>
	u_\beta(x_h-h,y_h)\ge  1-\rho/2.
	$$
	 Hence,
		$
		f(u_\beta(x_h-h,y_h))-f(u_\alpha(x_h,y_h))>0$. Since $u_{\beta,x}<0$, 
		\begin{equation*}
			\mathcal L V_h(x_h,y_h)
			=f(u_\beta(x_h-h,y_h))-f(u_\alpha(x_h,y_h))
			-d(y_h)u_{\beta,x}(x_h-h,y_h)>0.
		\end{equation*}
	This contradicts 
	$
	\mathcal L V_h(x_h,y_h)\le0$,
	which follows from the fact that $(x_h,y_h)$ is a maximum point of
	$V_h$.  Hence, $\eps_h=0$ and
	$u_\alpha\leq u_\beta(\cdot-h,\cdot)$ for all sufficiently large $h$.

	Set
	$$
h_*
	:=\inf 
	\left\{
	h\in\R:
	u_\alpha(x,y)\leq u_\beta(x-h',y)
	\ \hbox{for all }h'\geq h
	\ \hbox{and for }(x,y)\in\R\times\T^{N-1}
	\right\}.
	$$
	 Then, 
		$h_*>-\infty$.  Hence, 
$u_\alpha\leq u_\beta(\cdot-h_*,\cdot)$ in $\R\times\T^{N-1}$.
		The two wave profiles cannot touch at a finite point.  Suppose that there is a touching point $(x_0,y_0)\in\R\times\T^{N-1}$. Then,  $u_\alpha(x_0,y_0)= u_\beta(x_0-h_*,y_0)$, and $f(u_\beta(x_0-h_*,y_0))-f(u_\alpha(x_0,y_0))=0$. Define $V:= u_\alpha-u_\beta(\cdot-h_*,\cdot)$ on $\R\times\T^{N-1}$. Then, 
			\begin{equation*}
			\mathcal L V(x_0,y_0)
			=f(u_\beta(x_0-h_*,y_0))-f(u_\alpha(x_0,y_0))
			-d(y_0)u_{\beta,x}(x_0-h_*,y_0)>0,
		\end{equation*}
		again contradicting $\mathcal L V(x_0,y_0)\le 0$.  Therefore,   $u_\alpha< u_\beta(\cdot-h_*,\cdot)$ in $\R\times\T^{N-1}$. Choose a sequence $(h_n)_{n\in\mathbb{N}}$ such that  $ h_n<h_*$ and 
		$h_n\to  h_*$.
		For each $n$, 
		$$
		\sup_{\R\times\T^{N-1}}
		\bigl(u_\alpha-u_\beta(\cdot-h_n,\cdot)\bigr)>0.
		$$
		Let $(x_n,y_n)$ be a point of positive maximum.  Since $u_\alpha< u_\beta(\cdot-h_*,\cdot)$ on every compact cylinder,
		one has
		$
		|x_n|\to+\infty$ as $n\to+\infty$. Since
		the sequence $(h_n)_{n\in \mathbb{N}}$ is bounded, we infer that  for large $n$, the functions  $u_\alpha(x_n,y_n)$ and $u_\beta(x_n-h_n,y_n)$ both belong  either  to  $[0,\rho]$ or to $[1-\rho,1]$.
		Therefore,
		$$
		f(u_\beta(x_n-h_n,y_n))
		-
		f(u_\alpha(x_n,y_n))>0.
		$$
		Together with $d>0$ and $u_{\beta,x}<0$, this gives that the function $V_n:= u_\alpha- u_\beta(\cdot-h_n,\cdot)$ satisfies 
		$$
		\mathcal L 
		V_n(x_n,y_n)=f(u_\beta(x_n-h_n,y_n))
		-
		f(u_\alpha(x_n,y_n))-d(y_n) u_{\beta,x}(x_n-h_n,y_n)>0,
		$$
		again contradicting $\mathcal{L}V_n(x_n,y_n)\le 0$.
	  Therefore,  our assumption was false and
$c_A(\alpha,f)\leq c_A(\beta,f)$.
\end{proof}

\subsection{Compactness after rescaling}

For $A\geq1$, write $\gamma_A:=c_A/A$, which is bounded by
\eqref{speed-bounds}, and recall $U_A(x,y)=u_A(Ax,y)$.
Then, the functions $U_A$ satisfy
\begin{equation}
\begin{aligned}
	\begin{cases}
 A^{-2}U_{A,xx}+\Delta_yU_A
 +\big(\gamma_A-\alpha(y)\big)U_{A,x}+f(U_A)=0
 &\hbox{in }\R\times\T^{N-1},\\
0<U_A<1,\qquad U_{A,x}<0
 &\hbox{in }\R\times\T^{N-1},\\
 U_A(+\infty,\cdot)=0,\qquad U_A(-\infty,\cdot)=1
 &\hbox{uniformly in }\T^{N-1}.
 \end{cases}
\end{aligned}
\label{eqn_TW_after rescaling}
\end{equation}
The scaling makes the first-order coefficient bounded, but the longitudinal
ellipticity coefficient tends to zero.  

%
%

For every $A\geq1$, since $U_{A,x}<0$ in $\R\times \T^{N-1}$, $U_A(-\infty,\cdot)=1$, and $U_A(+\infty,\cdot)=0$, it follows immediately that 
\begin{equation}
	\label{norm1}
	\Vert U_{A,x} \Vert_{L^1(\R\times\T^{N-1})}=1.
\end{equation}

Because $f$ changes sign, its signed integral does not directly
give a uniform global energy bound. We first establish local
estimates; global bounds will follow in Section~\ref{sec5}
once the full transition and uniform tails have been obtained.

\begin{lemma}
\label{lem:local-compactness}
Let $I\subset\R$ be a bounded interval.  There is a constant $C_I$,
independent of $A\geq1$ and of the longitudinal translation $\xi\in\R$,
such that
\begin{equation}
\begin{aligned}
 \int_{I\times\T^{N-1}}
   |\nabla_yU_A(x+\xi,y)|^2\,dx\,dy +A^{-2}\int_{I\times\T^{N-1}}
   |U_{A,x}(x+\xi,y)|^2\,dx\,dy\leq C_I.
\end{aligned}
\label{eq:local-energy}
\end{equation}
Consequently,
\begin{equation*}
 \sup_{A\geq1,\,\xi\in\R}
 \|U_A(\cdot+\xi,\cdot)\|_{W^{1,1}(I\times\T^{N-1})}<+\infty.
\end{equation*}

If $A_n\to\infty$ and $\gamma_{A_n}\to\gamma$ as $n\to+\infty$, and $\xi_n\in\R$ for all $n\in\mathbb{N}$, then a
subsequence of $U_{A_n}(\cdot+\xi_n,\cdot)$ converges in $L^1_{loc}$
and a.e. to a function $U$ satisfying 
\begin{equation}
 L_\gamma U+f(U)=0\quad\hbox{in }\D'(\R\times\T^{N-1}),
 \qquad 0\leq U\leq1,\qquad U_x\leq0\quad\hbox{in }\D'(\R\times\T^{N-1}).
\label{eq:local-limit}
\end{equation}
The convergence holds in $L^p_{loc}$ for every $1\le p<+\infty$.
\end{lemma}

\begin{proof}
Fix $\xi\in\R$ and set
$$
                  V_{A,\xi}(x,y):=U_A(x+\xi,y).
$$
Since the coefficients of \eqref{eqn_TW_after rescaling} are independent of $x$,
the function $V_{A,\xi}$ satisfies the same equation as $U_A$. Test this equation by
$\eta^2V_{A,\xi}$, where
$\eta\in C_c^1(\R)$ such that $\eta=1$ on $I$.  Integration by parts gives
\begin{equation*}
\begin{aligned}
 \int_{\R\times\T^{N-1}}\eta(x)^2
       |\nabla_yV_{A,\xi}&(x,y)|^2\,dx\,dy
 +A^{-2}\int_{\R\times\T^{N-1}}\eta(x)^2
       |\partial_x V_{A,\xi}(x,y)|^2\,dx\,dy\\
 &\quad=-2A^{-2}\int_{\R\times\T^{N-1}}
       \eta(x)\eta'(x)V_{A,\xi}(x,y)\partial_x V_{A,\xi}(x,y)\,dx\,dy\\
 &\qquad+\int_{\R\times\T^{N-1}}
       \eta(x)^2(\gamma_A-\alpha(y))
       V_{A,\xi}(x,y)\partial_x V_{A,\xi}(x,y)\,dx\,dy\\
 &\qquad+\int_{\R\times\T^{N-1}}
       \eta(x)^2f(V_{A,\xi}(x,y))V_{A,\xi}(x,y)\,dx\,dy.
\end{aligned}
\end{equation*}
Using Young's inequality 
\begin{align*}
	2A^{-2}\big| \eta\eta'V_{A,\xi}\partial_x V_{A,\xi}\big|
	&\le\frac12A^{-2}\eta^2|\partial_x V_{A,\xi}|^2
	+2A^{-2}|\eta'|^2 V_{A,\xi}^2,
\end{align*}
the first term on the right-hand side is absorbed into the left-hand side.   Regarding the second
 term, integration in $x$ yields
\begin{equation*}
\begin{aligned}
 \int_{\R\times\T^{N-1}}
 \eta(x)^2&(\gamma_A-\alpha(y))
 V_{A,\xi}(x,y)\partial_x V_{A,\xi}(x,y)\,dx\,dy\\
 &\quad=-\int_{\R\times\T^{N-1}}
 \eta(x)\eta'(x)(\gamma_A-\alpha(y))
 V_{A,\xi}(x,y)^2\,dx\,dy.
\end{aligned}
\end{equation*}
Finally, since $0\leq V_{A,\xi}\leq1$, the family $(\gamma_A)_{A\ge 1}$ is bounded, and
$f$ is bounded on $[0,1]$, the resulting upper bound is a constant,  independent of
$A$ and $\xi$.  This proves \eqref{eq:local-energy}.

Combining \eqref{eq:local-energy} and $\Vert U_{A,x} \Vert_{L^1(\R\times\T^{N-1})}=1$, it follows that
\begin{equation*}
\begin{aligned}
 \|V_{A,\xi}\|_{W^{1,1}(I\times\T^{N-1})}
 &\leq |I|
 +\int_{I\times\T^{N-1}}|\partial_x V_{A,\xi}(x,y)|\,dx\,dy+\int_{I\times\T^{N-1}}
        |\nabla_yV_{A,\xi}(x,y)|\,dx\,dy\\
 &\leq |I|+1+|I|^{1/2}C_I^{1/2}.
\end{aligned}
\end{equation*}

Suppose that $A_n\to\infty$ and $\gamma_{A_n}\to\gamma$ as $n\to+\infty$, and $\xi_n\in\R$ for all $n\in\mathbb{N}$.  For each $n\in\mathbb{N}$, define
$$
                  V_n(x,y):=U_{A_n}(x+\xi_n,y)~~~~\text{for}~(x,y)\in\R\times\T^{N-1}.
$$
The preceding $W^{1,1}(I\times\T^{N-1})$ bound and the compact
embedding into $L^1(I\times\T^{N-1})$ imply that $(V_n)$ is relatively
compact in $L^1(I\times\T^{N-1})$.  A diagonal extraction on
$K_m:=(-m,m)\times\T^{N-1}$, $m\in\mathbb N$, yields
$V_n\to U$ in $L^1_{\mathrm{loc}}(\R\times\T^{N-1})$
and, after a further extraction, almost everywhere.
For every $m\in\mathbb N$, \eqref{eq:local-energy} bounds
$\nabla_yV_n$ in $L^2(K_m)$. By weak compactness, any
subsequence has a further subsequence converging weakly in
$L^2(K_m)$ to some $G$. For every
$\varphi\in C_c^\infty(K_m)$ and $j=1,\ldots,N-1$,
weak convergence, integration by parts, and the convergence
$V_n\to U$ in $L^1(K_m)$ give, along this subsequence,
\begin{align*}
	\int_{K_m}G_j\varphi\,dx\,dy
	&=\lim_{n\to+\infty}
	\int_{K_m}\partial_{y_j}V_n\,\varphi\,dx\,dy\\
	&=-\lim_{n\to+\infty}
	\int_{K_m}V_n\,\partial_{y_j}\varphi\,dx\,dy
	=-\int_{K_m}U\,\partial_{y_j}\varphi\,dx\,dy.
\end{align*}
Thus, $G=\nabla_yU$ in $\D'(K_m)$.
Since every subsequential weak limit is identified in this way,
$\nabla_yV_n\rightharpoonup\nabla_yU$ in $L^2(K_m)$ along
the whole sequence under consideration. As $m$ is arbitrary,
$\nabla_yU\in L^2_{\mathrm{loc}}(\R\times\T^{N-1})$.

For every nonnegative $\varphi\in C_c^\infty(\R\times\T^{N-1})$, we have
$\int_{\R\times\T^{N-1}}V_n\varphi_x\,dx\,dy\geq0$.
Passing to the local $L^1$ limit gives $U_x\leq0$ in
$\D'(\R\times\T^{N-1})$.

  For every
$\varphi\in C_c^\infty(\R\times\T^{N-1})$, the function $V_n$
satisfies
\begin{equation}
	\begin{aligned}
		0={}&\int_{\R\times\T^{N-1}}V_n(x,y)
		\big[A_n^{-2}\varphi_{xx}(x,y)+\Delta_y\varphi(x,y)
		-(\gamma_{A_n}-\alpha(y))\varphi_x(x,y)\big] \,dx\,dy\\
		&+\int_{\R\times\T^{N-1}}f(V_n(x,y))\varphi(x,y)\,dx\,dy.
	\end{aligned}
	\label{eq:local-limit-weak-form}
\end{equation}
Since $f$ is Lipschitz on $[0,1]$, local $L^1$ convergence
also gives $f(V_n)\to f(U)$ in $L^1_{loc}$.
Letting $n\to+\infty$ in \eqref{eq:local-limit-weak-form},
together with $A_n^{-2}\to0$ and $\gamma_{A_n}\to\gamma$ proves
\eqref{eq:local-limit}.
 
Finally, since $0\leq V_n,U\leq1$, it follows that $|V_n-U|\leq1$, thus $|V_n-U|^p\leq |V_n-U|$ for every $1\leq p<\infty$.
The local $L^1$ convergence then implies local $L^p$
convergence for every finite $p$.
\end{proof}

We shall use the following strengthening of local compactness. It does
not require identification of the limits as $x\to\pm\infty$ or condition~\eqref{cdn_alpha}.
For later reference, monotonicity and the limits $1$ and $0$ of the approximating
profiles imply the translation estimate
\begin{equation}
\int_{\R\times\T^{N-1}}|U_A(x+h,y)-U_A(x,y)|\,dx\,dy=|h|, 
\quad  h\in\R.
\label{eq:revision-xtranslation}
\end{equation}
For $h>0$, this follows by writing
$U_A(x,y)-U_A(x+h,y)=\int_0^h(-U_{A,x})(x+t,y)\,dt$,
using Fubini's theorem and \eqref{norm1}; the case $h<0$ follows by a
change of variable.

\begin{lemma}\label{lem:local-strong-energy}
Let $A_n\to\infty$, $\gamma_{A_n}\to\gamma$, and
$V_n=U_{A_n}(\cdot+\xi_n,\cdot)\to U$ in $L^1_{\mathrm{loc}}(\R\times\T^{N-1})$.
Then, for every compact interval $I\subset\R$,
\begin{equation*}
	\nabla_yV_n\to \nabla_yU\quad\text{in }L^2(I\times\T^{N-1}),
	\qquad A_n^{-1}V_{n,x}\to 0\quad\text{in }L^2(I\times\T^{N-1}).
\end{equation*}
\end{lemma}

\begin{proof}
Choose a nonnegative $\chi\in C_c^\infty(\R)$. For the smooth
approximating profiles, exact integration by parts gives
\begin{equation}
	\label{eq:revision-local-energy}
\begin{aligned}
&\int_{\R\times\T^{N-1}}\chi\bigl(|\nabla_yV_n|^2+A_n^{-2}|V_{n,x}|^2\bigr)
       \,dx\,dy \\
&~~~~~~~~~~~~~~~~~~~=\frac{A_n^{-2}}{2}\int_{\R\times\T^{N-1}}\chi''V_n^2\,dx\,dy
 -\frac12\int_{\R\times\T^{N-1}}(\gamma_{A_n}-\alpha(y))\chi'V_n^2\,dx\,dy\\
 &~~~~~~~~~~~~~~~~~~~
 +\int_{\R\times\T^{N-1}}\chi f(V_n)V_n\,dx\,dy.
\end{aligned}
\end{equation}
We must justify the analogous identity for $U$, which need not be
continuous in $x$. Let $\rho_h$ be a smooth approximate identity
in $x$, and set $U^h=\rho_h*_xU$. Then, 
\begin{equation*}
	\Delta_yU^h+(\gamma-\alpha(y))U_x^h+(f(U))^h=0.
\end{equation*}
Testing by $\chi U^h$, which is legitimate after the usual transverse
Sobolev approximation, gives
\begin{equation*}
	\int_{\R\times\T^{N-1}}\chi|\nabla_yU^h|^2\,dx\,dy
	=-\frac12\int_{\R\times\T^{N-1}}(\gamma-\alpha(y))\chi'(U^h)^2\,dx\,dy
	+\int_{\R\times\T^{N-1}}\chi(f(U))^hU^h\,dx\,dy.
\end{equation*}
The proof of Lemma~\ref{lem:local-compactness} gives
$\nabla_yU\in L^2_{\mathrm{loc}}(\R\times\T^{N-1})$.
Since $\nabla_yU^h=\rho_h*_x\nabla_yU$, the approximation
properties of convolution yield, for every bounded interval
$I\subset\R$,
\begin{equation*}
	\|U^h-U\|_{L^2(I\times\T^{N-1})}
	+\|\nabla_yU^h-\nabla_yU\|_{L^2(I\times\T^{N-1})}
	\to 0
	\quad\hbox{as }h\to0.
\end{equation*}
Consequently,
\begin{equation}\label{eq:revision-limit-energy}
\int_{\R\times\T^{N-1}}\chi|\nabla_yU|^2\,dx\,dy
 =-\frac12\int_{\R\times\T^{N-1}}(\gamma-\alpha(y))\chi'U^2\,dx\,dy
  +\int_{\R\times\T^{N-1}}\chi f(U)U\,dx\,dy.
\end{equation}

The right-hand side of \eqref{eq:revision-local-energy} converges to the
right-hand side of \eqref{eq:revision-limit-energy}. Weak lower
semicontinuity, and nonnegativity of both terms on the left, imply
\begin{equation*}
\sqrt\chi\,\nabla_yV_n\to\sqrt\chi\,\nabla_yU
 \quad\hbox{in }L^2(\R\times\T^{N-1}),\qquad
A_n^{-1}\sqrt\chi\,V_{n,x}\to0\quad\hbox{in }L^2(\R\times\T^{N-1}).
\end{equation*}
Indeed, the sum of their squared norms tends to the squared norm of
the weak limit of the first term. The second term must therefore
vanish and the first must converge in norm. Taking $\chi=1$ on an
arbitrary compact interval proves strong local convergence.

\end{proof}

\section{Uniqueness of full fronts}
\label{sec3}

This section is concerned with 
 shears satisfying \eqref{cdn_alpha}. We prove that a full
degenerate front is regular, strictly decreasing in $x$, and unique
up to translation, with a uniquely determined speed. The proof
combines hypoelliptic regularity and the strong maximum principle
with a sliding argument. The uniform limits at $\pm\infty$ give
the strict speed bounds and allow comparison of translated fronts
on the whole cylinder. These results will identify the common
limit of the normalized profiles in Section~\ref{sec5}.

\subsection{Regularity and maximum principle}

\begin{lemma}
	\label{lem:hypo-regularity}
 Assume \eqref{cdn_alpha}, and let $\gamma\in\R$ be arbitrary.	Let $0\leq U\leq1$ be a bounded distributional solution of
	\begin{equation}
		L_\gamma U+f(U)=0
		\quad\hbox{in }\R\times\T^{N-1}.
		\label{eq:semilinear-hypo}
	\end{equation}
	Then,  $U\in C^{1,\beta}_{loc}(\R\times\T^{N-1})$ for some $\beta\in(0,\delta)$, and there exists
	$C>0$ such that
	\begin{equation}
		\sup_{x_0\in\R}
		\|U\|_{C^{1,\beta}((x_0-1,x_0+1)\times\T^{N-1})}
		\leq C.
		\label{eq:uniform-hypo-interior}
	\end{equation}
	Moreover,
	\begin{equation}
		L_\gamma U_x+f'(U)U_x=0
		\quad\hbox{in }\mathcal D'(\R\times\T^{N-1}).
		\label{eq:derivative-equation}
\end{equation}
	Finally, for every $x\in\R$, $U(x,\cdot)\in C^{2,\beta}(\T^{N-1})$ and
	\begin{equation}
		\sup_{x\in\R}
		\|U(x,\cdot)\|_{C^{2,\beta}(\T^{N-1})}
		\leq C,
		\label{eq:transverse-Schauder}
	\end{equation}
		 and $x\mapsto U(x,\cdot)$ belongs to
		$C(\R;C^{2,\beta'}(\T^{N-1}))$ for every $0<\beta'<\beta$.

\end{lemma}

\begin{proof}
	Under \eqref{cdn_alpha}, the vector fields
	$\partial_{y_i}$  $(i=1,\dots, N-1)$ and
	$(\gamma-\alpha(y))\partial_x$ satisfy the H\"ormander
	bracket condition.
	Since $0\leq U\leq1$ and $|f(U)|\leq\|f\|_{L^\infty([0,1])}$,
	the local estimates in
	\cite[Theorem~18(c)]{RothschildStein1976}
	first give a local H\"older bound for $U$.
	For exponents below one, the Lipschitz continuity of $f$
	gives the corresponding H\"older bound for $f(U)$.
	Thus, \cite[Theorem~18(b)]{RothschildStein1976}
	can be applied successively until the regularity exponent
	exceeds one.
	Applying these estimates on successively smaller cylinders,
	starting from $(-2,2)\times\T^{N-1}$, we obtain
	\begin{equation*}
			\|U\|_{C^{1,\beta}((-1,1)\times\T^{N-1})}\leq C
	\end{equation*}
	for some $\beta\in(0,\delta)$.
	The cylinders used in this finite iteration are fixed.
	Hence, $C$ depends only on $L_\gamma$, $\delta$,
	$\|f\|_{C^{1,\delta}([0,1])}$ and these cylinders,
	with the bound $0\leq U\leq1$ supplying the initial control.
	
	For any $x_0\in\R$, set $U^{x_0}(x,y):=U(x+x_0,y)$.
	Since the coefficients of $L_\gamma$ are independent of $x$,
	the function $U^{x_0}$ satisfies the same equation and still
	takes values in $[0,1]$.
	Therefore, the preceding local estimate applies to
	$U^{x_0}$ with the same constant $C$, giving
	\begin{equation*}
			\|U\|_{C^{1,\beta}((x_0-1,x_0+1)\times\T^{N-1})}
		=
		\|U^{x_0}\|_{C^{1,\beta}((-1,1)\times\T^{N-1})}
		\leq C.
	\end{equation*}
	Since $x_0$ is arbitrary, this proves
	$U\in C^{1,\beta}_{\mathrm{loc}}(\R\times\T^{N-1})$
	and \eqref{eq:uniform-hypo-interior}.
	
	Finally, the coefficients of $L_\gamma$ are independent of $x$,
	so differentiating \eqref{eq:semilinear-hypo} in distributions
	gives $L_\gamma U_x+\partial_x(f(U))=0$.
	Since $U\in C^1_{\mathrm{loc}}(\R\times\T^{N-1})$ and
	$f\in C^1([0,1])$, the chain rule gives
	$\partial_x(f(U))=f'(U)U_x$.
	This proves \eqref{eq:derivative-equation}.
	
	Set $H(x,y):=-(\gamma-\alpha(y))U_x(x,y)-f(U(x,y))$. Then, 
	$H\in C^{0,\beta}_{loc}$ and $\Delta_yU=H$ in $\D'(\R\times\T^{N-1})$. 
	For every $\phi\in C^\infty(\T^{N-1})$, the function
	$x\mapsto\int_{\T^{N-1}}(U\Delta_y\phi-H\phi)\,dy$ is continuous and
	vanishes as a distribution in $x$. Hence, it vanishes identically.
	Thus,  $\Delta_yU(x,\cdot)=H(x,\cdot)$ for every $x$, and 
	Schauder estimates give \eqref{eq:transverse-Schauder}.
	
	Finally, as
	$x'\to x$, one has $U(x',\cdot)\to U(x,\cdot)$ in $C^0$ and
	$H(x',\cdot)\to H(x,\cdot)$ in $C^{0,\beta'}$ for every
	$\beta'<\beta$. Applying the same Schauder estimate to the difference
	of the corresponding slice equations proves the last assertion.
\end{proof}

The following consequence of
\cite[Theorem~1.1]{HamelZlatosHarnack} provides the strong
maximum principle needed for longitudinal comparison.

\begin{proposition}\label{prop_SMP}
	Assume \eqref{cdn_alpha} and $
		\min_{\T^{N-1}}\alpha<\gamma
		<\max_{\T^{N-1}}\alpha$.
	For every $M>0$, 
	let $a\in L^\infty(\R\times\T^{N-1})$ satisfy
	$\|a\|_{L^\infty(\R\times\T^{N-1})}\leq M$, and let $z$ be a bounded 
	distributional solution of
	$$
	L_\gamma z+a(x,y)z=0
	\quad\text{in }\R\times\T^{N-1},
	$$
	satisfying $z\ge 0$ a.e. in $\R\times\T^{N-1}$. 
	Then, $z$ admits a continuous representative, still denoted by $z$, and either $z\equiv0$, or $z>0$ in $\R\times\T^{N-1}$ and there exists
	$C_H=C_H(\alpha,\gamma,M)\geq1$ such that 
	\begin{equation}
		C_H^{-1}z(x,y)
		\leq z(x+h,y')
		\leq C_Hz(x,y)
		\label{eq:harnack-comparison}
	\end{equation}
	for every $x\in\R$, $|h|\leq1$, and
	$y,y'\in\T^{N-1}$.
\end{proposition}

\begin{proof}
We identify functions on $\T^{N-1}$ with their periodic extensions
to $\R^{N-1}$. The strict inequalities for $\gamma$ imply that
$\gamma-\alpha$ changes sign. Moreover,
$D^\zeta(\gamma-\alpha)=-D^\zeta\alpha$ for every multi-index
$\zeta$ with $|\zeta|\geq1$, so \eqref{cdn_alpha} implies that
$\gamma-\alpha$ satisfies the non-degeneracy hypothesis of
\cite[Theorem~1.1]{HamelZlatosHarnack}.
	Moreover, the local regularity
	result \cite[Theorem~18(c)]{RothschildStein1976} implies that $z$ admits
	a continuous representative, which is precisely the regularity argument
	used at the beginning of the proof of
	\cite[Theorem~1.1]{HamelZlatosHarnack}. We use this representative 
	throughout. Since $z\geq0$ almost everywhere, continuity gives
	$z\geq0$ everywhere.

	For each
	$x_0\in\R$, we apply
	\cite[Theorem~1.1]{HamelZlatosHarnack} on
	$(x_0-3,x_0+3)\times (-2,3)^{N-1}$, with inner cylinder
	$[x_0-2,x_0+2]\times (-1,2)^{N-1}$. It gives
	\begin{equation}
		\sup_{[x_0-2,x_0+2]\times (-1,2)^{N-1}}z
		\leq C_H
		\inf_{[x_0-2,x_0+2]\times (-1,2)^{N-1}}z.
		\label{eq:harnack-cylinder}
	\end{equation}
	The constant is independent of $x_0$, since the geometry is fixed,
	the coefficients are independent of $x$, and the zeroth-order
	coefficient enters the estimate only through its $L^\infty$ bound.
	
	Fix $x\in\R$, $|h|\leq1$, and $y,y'\in\T^{N-1}$, and choose
	representatives of $y,y'$ in $[0,1]^{N-1}\subset (-1,2)^{N-1}$. Taking
	$x_0=x$ in \eqref{eq:harnack-cylinder} and comparing the two points
	$(x,y)$ and $(x+h,y')$ gives the upper bound in
	\eqref{eq:harnack-comparison}. Interchanging the two points gives
	the lower bound.
	
	Finally, if $z$ vanishes at one point, then
	\eqref{eq:harnack-cylinder} shows that it vanishes on an inner
	cylinder containing that point. A chain of overlapping longitudinal
	translates of this cylinder then yields $z\equiv0$ on
	$\R\times\T^{N-1}$. 
\end{proof}

\subsection{Uniqueness of solutions of \eqref{eq:main-degenerate-front} under \eqref{cdn_alpha}}

\begin{lemma}
\label{lem_U and gamma}
Assume \eqref{cdn_alpha}.
Let $U$ be a full degenerate front with speed $\gamma$, and suppose
that $\alpha$ is nonconstant.  Then,  $0<U<1$ and $U_x<0$ on $\R\times\T^{N-1}$. Moreover,
\begin{equation*}
          \min_{\T^{N-1}}\alpha<\gamma<\max_{\T^{N-1}}\alpha.
\end{equation*}
\end{lemma}

\begin{proof}
Choose $\rho\in(0,\min\{\theta,1-\theta\})$ so small that $f$ is strictly
decreasing on $[0,\rho]$ and on $[1-\rho,1]$. By Lemma~\ref{lem:hypo-regularity}, $U$ is continuous.  We first
establish the speed bounds.  The continuous function
$x\mapsto\max_yU(x,y)$ tends to $0$ and $1$ as $x\to \pm\infty$.  Its $\rho$-level set is therefore nonempty, closed, and bounded above.  
Let $R^*\in\R$ be
its rightmost point. Then,
\begin{equation}
	0\leq U\leq\rho\quad\hbox{on }[R^*,\infty)\times\T^{N-1},
	\qquad \max_yU(R^*,y)=\rho.
	\label{eq:last-right-section}
\end{equation}
Integrate the equation over $(R^*,S)\times\T^{N-1}$ and let $S\to\infty$.
Since $U(+\infty,\cdot)=0$ uniformly on $\T^{N-1}$,  it follows that
\begin{equation*}
	\int_{\T^{N-1}}(\gamma-\alpha(y))U(R^*,y)\,dy
	=\int_{R^*}^{\infty}\!\int_{\T^{N-1}}f(U(x,y))\,dy\,dx<0,
\end{equation*}
by \eqref{eq:last-right-section} and continuity.
This yields $\gamma<\max_{\T^{N-1}}\alpha$.  The opposite inequality follows
from the analogous argument at the left tail. Indeed, the continuous function $x\longmapsto
\max_{y\in\T^{N-1}}(1-U(x,y))$ tends to 0 as $x\to -\infty$ and has a leftmost $\rho$-level point $R_*\in\R$,  such that $0\le 1-U\le \rho$ on $(-\infty, R_*]\times\T^{N-1}$ and $\max_y(1-U(R_*,y))=\rho$.  Integration there
gives
\begin{equation*}
	\int_{\T^{N-1}}(\gamma-\alpha(y))(1-U(R_*,y))\,dy
	=\int_{-\infty}^{R_*}\!\int_{\T^{N-1}}f(U(x,y))\,dy\,dx>0.
\end{equation*}
Hence,
$\gamma>\min_{\T^{N-1}}\alpha$.  Consequently, 
$$
\min_{\T^{N-1}}\alpha<\gamma<\max_{\T^{N-1}}\alpha.
$$

We now apply Proposition~\ref{prop_SMP} to show that $0<U<1$ and $U_x<0$ on $\R\times\T^{N-1}$.
Since $f\in C^1([0,1])$ and $f(0)=f(1)=0$, the functions 
	\begin{equation*}
	\label{a01}
	a_0(s)=
	\begin{cases}
		\displaystyle	\frac{f(s)}{s},&s>0,\\[1.2ex]
		f'(0),&s=0,
	\end{cases}
	\qquad
	a_1(s)=
	\begin{cases}
		\displaystyle	-\frac{f(s)}{1-s},&s<1,\\[1.2ex]
		f'(1),&s=1.
	\end{cases}
\end{equation*}
are bounded and continuous. The equation for $U$ yields
$$
L_\gamma U+a_0(U)U=0,
\qquad
L_\gamma(1-U)+a_1(U)(1-U)=0.
$$
If $U$ vanished at some point, Proposition~\ref{prop_SMP} applied
to $U\geq0$ would give $U\equiv0$, contradicting
$U(-\infty,\cdot)=1$. If $U$ attained $1$, the same proposition
applied to $1-U\geq0$ would give $U\equiv1$, contradicting
$U(+\infty,\cdot)=0$. Therefore,  $0<U<1$.

Because of  $U(-\infty,\cdot)=1$ and $U(+\infty,\cdot)=0$ uniformly on $\T^{N-1}$, there exists  $M>1$ so that
\begin{equation*}
	U(x,y)\leq\rho/2\quad\hbox{for }x\geq M,\qquad
	U(x,y)\geq1-\rho/2\quad\hbox{for }x\leq-M.
\end{equation*}
For $h\geq4M$, define
\begin{equation*}
	\eps_h:=\inf\{\eps\geq0:
	U(x,y)-\eps\leq U(x-h,y)\hbox{ for all }(x,y)\in\R\times\T^{N-1}\}.
\end{equation*}
Repeating the arguments in the beginning of the proof of Lemma \ref{lem:order-speed},  one gets $\eps_h=0$ and thus
$U(x,y)\leq U(x-h,y)$ for all $h\geq4M$. Let
\begin{equation*}
	h_*:=\inf\{h\in\R:U(x,y)\leq U(x-h',y)
	\hbox{ for }h'\geq h\hbox{ and for all }(x,y)\in\R\times\T^{N-1}\}.
	\label{eq:critical-self-shift}
\end{equation*}
Then, $h_*>-\infty$.  Suppose $h_*>0$. The function $W:=U-U(\cdot-h_*,\cdot)\le 0$ on $\R\times\T^{N-1}$.

Assume that there exists  $(x_0,y_0)\in\R\times\T^{N-1}$ such that $W(x_0,y_0)=0$. Define
\begin{equation*}
	a(x,y)=\int_0^1 f'\big(U(x-h_*,y)+sW(x,y)\big)\,ds.
\end{equation*}
Then,  $L_\gamma W+aW=0$. The strong maximum principle in Proposition~\ref{prop_SMP} forces
$U(x,y)=U(x-h_*,y)$. Iteration would give for any $k\in\mathbb{N}$, $U(x,y)=U(x-kh_*,y)\to 1$ and  $U(x,y)=U(x+kh_*,y)\to 0$ as $k\to+\infty$. This is a contradiction. Therefore, 
\begin{equation}
	\label{no touching}
	U<U(\cdot-h_*,\cdot)~~~\text{on}~\R\times\T^{N-1}.
\end{equation}

According to the
definition of $h_*$, one can choose a sequence $(h_n)_{n\in\mathbb{N}}$  such that $h_*-1/n<h_n<h_*$ for all $n\in\mathbb{N}$ and thus $h_n\to h_*$ as $n\to+\infty$, and such that the functions $W_n:=U-U(\cdot-h_n,\cdot)$ have positive maxima.  It then follows from \eqref{no touching} that the corresponding
maximizing points $(x_n,y_n)$ must satisfy $|x_n|\to \infty$.  Arguing  as at the end of the proof of Lemma \ref{lem:order-speed}, one gets 
again a contradiction.  Hence, $h_*\leq0$, which proves $U_x\leq0$  on $\R\times\T^{N-1}$.
Moreover, \eqref{eq:derivative-equation} gives 
$L_\gamma U_x+f'(U)U_x=0$. The strong maximum principle in Proposition \ref{prop_SMP} together with $U(-\infty,\cdot)=1$ and $U(+\infty,\cdot)=0$ uniformly on $\T^{N-1}$ yields $U_x<0$ on $\R\times\T^{N-1}$.
This completes the proof.
\end{proof}

\begin{proposition} 
\label{prop:full-uniqueness}
Assume \eqref{cdn_alpha}.  If
$(\gamma_i,U_i)$, $i=1,2$, are two full degenerate fronts, then there exists $h\in\R$ such that
\begin{equation*}
 \gamma_1=\gamma_2,
 \qquad U_2(x,y)=U_1(x-h,y).
\label{eq:full-uniqueness}
\end{equation*}
\end{proposition}

\begin{proof}
Assume without loss of generality that $\gamma_1\leq\gamma_2$.  	Choose $\rho>0$ so that $f$ is strictly decreasing on $[0,\rho]$ and
$[1-\rho,1]$. Because of  $U_i(-\infty,\cdot)=1$ and $U_i(+\infty,\cdot)=0$ uniformly on $\T^{N-1}$, there exists $M>1$ so that, for $i=1,2$,
$$
 U_i(x,y)\leq\rho/2\quad\hbox{for }x\geq M,\qquad
 U_i(x,y)\geq1-\rho/2\quad\hbox{for }x\leq-M.
$$
For $h\geq4M$, let
\begin{equation*}
 \eps_h:=\inf\{\eps\geq0:
 U_2(x,y)-\eps\leq U_1(x-h,y) \hbox{ for all }(x,y)\in\R\times\T^{N-1}\}.
\end{equation*}
Suppose, for contradiction, that  $\eps_h>0$.
Since
$U_2(x,\cdot)-U_1(x-h,\cdot)\to0$ uniformly in $\T^{N-1}$ as $x\to\pm\infty$, the function
$V_h:=U_2-\eps_h-U_1(\cdot-h,\cdot)$ attains its zero maximum at some point $(x_h,y_h)\in\R\times\T^{N-1}$. 	If $x_h\geq2M$, we have
$U_2(x_h,y_h)\le \rho/2$ and $U_1(x_h-h,y_h)=U_2(x_h,y_h)-\varepsilon_h<\rho/2$.  If $x_h<2M$, it follows that 
$x_h-h<-2M$ and  $U_1(x_h-h,y_h)\ge 1-\rho/2$. Therefore,
$$
U_2(x_h,y_h)
=
U_1(x_h-h,y_h)+\varepsilon_h
>
U_1(x_h-h,y_h)\ge  1-\rho/2.
$$
Consequently, we have
$
f(U_1(x_h-h,y_h))-f(U_2(x_h,y_h))>0$. Since $U_{i,x}<0$ on $\R\times\T^{N-1}$ by Lemma \ref{lem_U and gamma}, we have 
\begin{equation*}
L_{\gamma_2} V_h(x_h,y_h)
	=f(U_1(x_h-h,y_h))-f(U_2(x_h,y_h))
	-(\gamma_2-\gamma_1)U_{1,x}(x_h-h,y_h)>0.
\end{equation*}
This contradicts
$
L_{\gamma_2} V_h(x_h,y_h)\le0$,
which follows from the fact that $(x_h,y_h)$ is a maximum point of
$V_h$.  Hence, $\eps_h=0$, and for
all sufficiently large $h$, we have
\begin{equation*}\label{a}
                    U_2(x,y)\leq U_1(x-h,y) ~~~~\hbox{    for all }(x,y)\in\R\times\T^{N-1}.
\end{equation*}

Set
$$
 h_*:=\inf\{h\in\R:U_2(x,y)\leq U_1(x-h',y)
                \hbox{ for }h'\geq h\hbox{ and for all }(x,y)\in\R\times\T^{N-1}\}.
$$
Then, $h_*>-\infty$, and the function
$W=U_2-U_1(\cdot-h_*,\cdot)\le 0$ on $\R\times\T^{N-1}$. By defining
\begin{equation*}
 q(x,y)=
 \begin{cases}
 \displaystyle\frac{f(U_2)-f(U_1(\cdot-h_*))}
 {U_2-U_1(\cdot-h_*)},&~~\text{if}~~U_2\ne U_1(\cdot-h_*),\\[1.2ex]
 f'(U_2),&~~\text{if}~~U_2=U_1(\cdot-h_*),
 \end{cases}
\end{equation*}
we observe that the function $W$ satisfies
\begin{equation}\label{L-eqn}
 L_{\gamma_2}W+qW
 =-(\gamma_2-\gamma_1)U_{1,x}(\cdot-h_*,\cdot),
\end{equation}
where the right-hand side is nonnegative.    Suppose that $W$ attains its zero maximum at some
$(x_0,y_0)$. Then,  $W_x(x_0,y_0)=0$ and
$\Delta_yW(x_0,y_0)\leq0$. 
The strict inequality $U_{1,x}<0$ along with \eqref{L-eqn} shows that
$\gamma_2>\gamma_1$ would give a positive right-hand side at
$(x_0,y_0)$, contradicting
$$
\bigl( L_{\gamma_2}W+qW\bigr)(x_0,y_0)
=\Delta_yW(x_0,y_0)\leq0.
$$
Hence, $\gamma_1=\gamma_2$. Consequently, $-W\geq0$ solves
$L_{\gamma_2}(-W)+q(-W)=0$ and vanishes at
$(x_0,y_0)$. Proposition~\ref{prop_SMP} therefore yields
$W\equiv0$.

Suppose now that the function $W<0$ on $\R\times\T^{N-1}$.  One can then take a sequence $(h_n)_{n\in\mathbb{N}}$ such that 
$h_n<h_*$ for all $n\in\mathbb{N}$ and $h_n\to h_*$ as $n\to+\infty$, and such that the functions
$W_n=U_2-U_1(\cdot-h_n,\cdot)$ have positive maxima.  Moreover, $W<0$ on $\R\times\T^{N-1}$ and $W_n\to W$ as $n\to+\infty$ locally uniformly on $(x,y)\in\R\times\T^{N-1}$ implies that the maximizing points $(x_n,y_n)_{n\in\mathbb{N}}$ satisfy $|x_n|\to+\infty$ as $n\to+\infty$. After passing to a further subsequence,  either $x_n\to-\infty$  or $x_n\to+\infty$,
 so that $U_2(x_n,y_n)$ and $U_1(x_n-h_n,y_n)$ for $n$ large enough belong to an interval on which $f$ is
strictly decreasing.  Therefore,
\begin{equation*}
 L_{\gamma_2}W_n(x_n,y_n)
 =f(U_1(x_n-h_n,y_n))-f(U_2(x_n,y_n))
   -(\gamma_2-\gamma_1)U_{1,x}(x_n-h_n,y_n)>0,
\end{equation*}
which is impossible.  Therefore, the proof is complete.
\end{proof}

\section{Limiting profiles and exclusion of terraces}
\label{sec4}

In this section, we prove that the rescaled bistable transition remains a full
front for every nonconstant $\alpha\in C^{1,\delta}(\T^{N-1})$.
Along a sequence for which $\gamma_{A_n}\to\gamma$, local
compactness and the order of the transverse equilibria first
give either a full front or a two-front terrace for the limiting
degenerate equation. The latter consists of two monotone
distributional fronts connecting $1$ to an intermediate transverse
equilibrium $w$ and $w$ to $0$. Both fronts inherit the same speed
$\gamma$ from the original sequence. The instability of $w$ yields
incompatible sign conditions for this common speed and excludes
the terrace alternative. We begin by identifying the uniform
limits at $\pm\infty$ of monotone distributional profiles.

\begin{lemma}
\label{lem:longitudinal-limits}
Let $\gamma\in\R$, and let $U\in L^\infty(\R\times\T^{N-1})$
satisfy $\nabla_yU\in L^2_{\mathrm{loc}}(\R\times\T^{N-1})$ and
\begin{equation*}
	L_\gamma U+f(U)=0\quad\hbox{in }\D'(\R\times\T^{N-1}),
\end{equation*}
with $0\leq U\leq1$  a.e.  in $\R\times\T^{N-1}$ and
$U_x\leq0$ in $\D'(\R\times\T^{N-1})$.
Then,  there exist $p,q\in\mathcal E_f$ and a representative of $U$,
still denoted by $U$, such that $q(y)\leq U(x,y)\leq p(y)$ for all
$(x,y)\in\R\times\T^{N-1}$, with $U(-\infty,\cdot)=p$ and
$U(+\infty,\cdot)=q$ uniformly in $\T^{N-1}$. Moreover,
if $U$ depends on $x$, then $p\succ q$.
\end{lemma}
\begin{proof}
Since $U_x\leq0$ in $\D'(\R\times\T^{N-1})$, we may choose a
representative for which $x\mapsto U(x,y)$ is nonincreasing for
almost every $y\in\T^{N-1}$. For these $y$, the limits
$p(y):=\lim_{x\to-\infty}U(x,y)$ and $q(y):=\lim_{x\to+\infty}U(x,y)$
exist and satisfy $0\leq q(y)\leq p(y)\leq1$. In particular,
$q\leq U\leq p$  a.e. in $\R\times\T^{N-1}$.
For every $1\leq r<+\infty$, dominated convergence gives
$\lim_{x\to-\infty}\|U(x,\cdot)-p\|_{L^r(\T^{N-1})}=0$ and $\lim_{x\to+\infty}\|U(x,\cdot)-q\|_{L^r(\T^{N-1})}=0$
for this representative. 

We  show that
$p,q\in\mathcal E_f$. For $j\in\mathbb N$, set
$U_j^-(x,y):=U(x-j,y)$ and $U_j^+(x,y):=U(x+j,y)$.
By the definitions of $p,q$ and dominated convergence, we get
$U_j^-\to p$ and $U_j^+\to q$ in
$L^1_{\mathrm{loc}}(\R\times\T^{N-1})$ as $j\to+\infty$.
Both translates solve the same distributional equation as $U$,
because its coefficients are independent of $x$. Passing to the
limit in their weak formulations gives
$\Delta_yp+f(p)=0$ and $\Delta_yq+f(q)=0$ in $\D'(\T^{N-1})$.
For $1<r<+\infty$, the elliptic estimate below
\begin{equation}
	\|g\|_{W^{2,r}(\T^{N-1})}
	\leq C_r\bigl(\|\Delta_yg\|_{L^r(\T^{N-1})}
	+\|g\|_{L^r(\T^{N-1})}\bigr)
	\label{eq:section4-elliptic-estimate}
\end{equation}
first yields $p,q\in W^{2,r}(\T^{N-1})$ for every such $r$.
Sobolev embedding and Schauder estimates then give
$p,q\in C^2(\T^{N-1})$. Thus, $p,q\in\mathcal E_f$ and $p(y)\ge q(y)$ for all $y\in\T^{N-1}$.

We next prove that 
\begin{equation}
	\|U-p\|_{L^\infty((-\infty,-R)\times\T^{N-1})}
	+\|U-q\|_{L^\infty((R,+\infty)\times\T^{N-1})}\to0
	\quad\hbox{as }R\to+\infty.
	\label{eq:weak-uniform-limits}
\end{equation}
Choose
$\rho\in C_c^\infty((-1,1))$ with $\rho\geq0$ and
$\int_\R\rho(s)\,ds=1$, and define
$$Z(x,y):=\int_\R\rho(s)U(x-s,y)\,ds.$$
The equation for $U$ gives
\begin{equation*}
	\Delta_yZ+(\gamma-\alpha(y))Z_x+
	\int_\R\rho(s)f(U(x-s,y))\,ds=0
	\quad\hbox{in }\D'(\R\times\T^{N-1}).
\end{equation*}
Moreover, $\|Z\|_{L^\infty(\R\times\T^{N-1})}\leq1$ and
$\|Z_x\|_{L^\infty(\R\times\T^{N-1})}\leq\|\rho'\|_{L^1(\R)}$, so
\begin{equation*}
	\|\Delta_yZ\|_{L^\infty(\R\times\T^{N-1})}
	\leq\|\gamma-\alpha\|_{L^\infty(\T^{N-1})}\|\rho'\|_{L^1(\R)}
	+\|f\|_{L^\infty([0,1])}.
\end{equation*}
Convolution also makes $Z$, $Z_x$, and the averaged reaction term
continuous in $x$ with values in $L^\infty(\T^{N-1})$. Consequently,
the transverse equation holds on every section. Applying
\eqref{eq:section4-elliptic-estimate} with $r>N-1$, we obtain
$\sup_{x\in\R}\|Z(x,\cdot)\|_{W^{2,r}(\T^{N-1})}<+\infty$.
In particular, the functions $Z(x,\cdot)$ are uniformly
equicontinuous on $\T^{N-1}$.

We first verify $\lim_{x\to+\infty}\|Z(x,\cdot)-q\|_{L^1(\T^{N-1})}=0$.
By the  Tonelli's theorem,
\begin{align*}
	\|Z(x,\cdot)-q\|_{L^1(\mathbb T^{N-1})}&=\int_{\mathbb T^{N-1}}\Big|\int_{\mathbb R}
	\rho(s)U(x-s,y)-q(y)\,ds\Big|\,dy\\
	&\leq
	\int_{\mathbb T^{N-1}}\int_{\mathbb R}
	\rho(s)|U(x-s,y)-q(y)|\,ds\,dy\\
	&=
	\int_{\mathbb R}
	\rho(s)\|U(x-s,\cdot)-q\|_{L^1(\mathbb T^{N-1})}\,ds,
\end{align*}
in which we notice that the $L^1(\T^{N-1})$ convergence proved
in the first paragraph gives $\|U(x-s,\cdot)-q\|_{L^1(\T^{N-1})}\to0$
as $x\to+\infty$.
Moreover, since $0\leq U,q\leq1$ almost everywhere,
we have $
0\leq
\rho(s)\|U(x-s,\cdot)-q\|_{L^1(\mathbb T^{N-1})}
\leq
\rho(s)$,
and the function on the right belongs to $L^1(\mathbb R)$.
Applying the dominated convergence theorem in the above inequality gives
\begin{equation*}
	\lim_{x\to+\infty}
	\|Z(x,\cdot)-q\|_{L^1(\mathbb T^{N-1})}=0.
	\label{eq:Z-right-L1}
\end{equation*}

Next, we show that $$ \lim_{x\to+\infty}
\|Z(x,\cdot)-q\|_{L^\infty(\mathbb T^{N-1})}=0.$$
Suppose not. Then,  there exist $\eta>0$, a sequence
$x_j\to+\infty$, and points $y_j\in\mathbb T^{N-1}$ such that
$
|Z(x_j,y_j)-q(y_j)|\geq\eta$ for every $j$.
By the uniform equicontinuity of
$\{Z(x,\cdot):x\in\R\}$ and the uniform continuity of $q$,
there exists $a\in(0,1/4)$, independent of $j$, such that
\begin{equation*}
	|Z(x_j,y)-Z(x_j,y_j)|+|q(y)-q(y_j)|<\frac{\eta}{2}
	\quad\hbox{for }y\in B_a(y_j),
\end{equation*}
where $B_a(y_j)$ denotes the ball of radius $a$ centred at $y_j$
on the flat torus $\T^{N-1}$.
By the triangle inequality, we have 
\begin{align*}
	|Z(x_j,y)-q(y)|
	\geq |Z(x_j,y_j)-q(y_j)|
-|Z(x_j,y)-Z(x_j,y_j)|
	-|q(y)-q(y_j)|
	\geq\frac{\eta}{2},~y\in B_a(y_j).
\end{align*}
 Since translation preserves measure on $\T^{N-1}$,
\begin{equation*}
	\|Z(x_j,\cdot)-q\|_{L^1(\T^{N-1})}
	\geq\frac{\eta}{2}|B_a(y_j)|
	=\frac{\eta}{2}|B_a(0)|=\frac{\eta}{2}\omega_{N-1} a^{N-1}>0,
\end{equation*}
where $\omega_{N-1}$ is the volume of the unit ball in
$\mathbb R^{N-1}$.
This lower bound is independent of $j$, contradicting the
convergence in $L^1(\T^{N-1})$.
Therefore,
$
\lim_{x\to+\infty}
\|Z(x,\cdot)-q\|_{L^\infty(\mathbb T^{N-1})}=0$.
The same argument gives
$\lim_{x\to-\infty}\|Z(x,\cdot)-p\|_{L^\infty(\T^{N-1})}=0$.

Monotonicity and the support of $\rho$ imply
$U(x+1,y)\leq Z(x,y)\leq U(x-1,y) $ for a.e. $(x,y)\in\R\times\T^{N-1}$.
Therefore,
\begin{equation*}
	0\leq U(x,y)-q(y)\leq Z(x-1,y)-q(y),
	\qquad
	0\leq p(y)-U(x,y)\leq p(y)-Z(x+1,y)
\end{equation*}
almost everywhere. Therefore,
\begin{align*}
\|U-q\|_{L^\infty((R,+\infty)\times\T^{N-1})}
&\leq\sup_{x>R-1}\|Z(x,\cdot)-q\|_{L^\infty(\T^{N-1})},\\
\|U-p\|_{L^\infty((-\infty,-R)\times\T^{N-1})}
&\leq\sup_{x<-R+1}\|Z(x,\cdot)-p\|_{L^\infty(\T^{N-1})}.
\end{align*}
Since the right-hand sides above tend to zero as $R\to+\infty$, \eqref{eq:weak-uniform-limits} follows.

We now redefine $U$ on a set of measure zero and prove that
$U(-\infty,\cdot)=p$ and $U(+\infty,\cdot)=q$
uniformly in $\T^{N-1}$.
By Fubini's theorem, there is a null set $E\subset\T^{N-1}$
such that, for every $y\notin E$, the function $U(\cdot,y)$
is nonincreasing with limits $p(y),q(y)$, and
$Z(x+1,y)\leq U(x,y)\leq Z(x-1,y)$ for almost every $x\in\R$.
Fix such a $y$ and any $x\in\R$. For $s<x<t$ at which these
inequalities hold, monotonicity gives
\begin{equation*}
	Z(t+1,y)\leq U(t,y)\leq U(x,y)\leq U(s,y)\leq Z(s-1,y).
\end{equation*}
Letting $s$ and $t$ both tend to $x$, the continuity of $Z$
gives $Z(x+1,y)\leq U(x,y)\leq Z(x-1,y)$ for every $x\in\R$.
For $y\in E$, redefine $U(x,y)=p(y)$ when $x<0$ and
$U(x,y)=q(y)$ when $x\geq0$.
This modification preserves the distributional equation and
weak derivatives. Moreover, $q(y)\leq U(x,y)\leq p(y)$ for all
$(x,y)\in\R\times\T^{N-1}$, by monotonicity outside $E$
and by the chosen values on $E$.

Fix $R>0$. For $x>R+1$, the definition of $Z$ and the support
of $\rho$ give
\begin{equation*}
	|Z(x,y)-q(y)|
	\leq\int_\R\rho(s)|U(x-s,y)-q(y)|\,ds
	\leq\|U-q\|_{L^\infty((R,+\infty)\times\T^{N-1})}
\end{equation*}
for almost every $y$. By continuity, the inequality also holds for every $y$.
The corresponding estimate with $p$ holds when $x<-R-1$.
Combining these estimates with the comparison outside $E$
and the chosen values on $E$, we obtain
\begin{align*}
	\sup_{y\in\T^{N-1}}|U(x,y)-q(y)|
	&\leq\|U-q\|_{L^\infty((R,+\infty)\times\T^{N-1})},~~~~
	x>R+2,\\
	\sup_{y\in\T^{N-1}}|U(x,y)-p(y)|
	&\leq\|U-p\|_{L^\infty((-\infty,-R)\times\T^{N-1})},~~
	x<-R-2.
\end{align*}
By \eqref{eq:weak-uniform-limits}, both right-hand sides tend
to zero as $R\to+\infty$. Thus,  $U(-\infty,\cdot)=p$ and
$U(+\infty,\cdot)=q$ uniformly in $\T^{N-1}$.

Finally, if $p=q$, the inequalities $q\leq U\leq p$ give
$U(x,y)=p(y)$ almost everywhere. Thus,  $p\succ q$ whenever $U$
depends on $x$.
\end{proof}

\begin{lemma}
\label{lem:two sequences}
Let $A_n\to+\infty$, $\gamma_{A_n}\to\gamma$, and let
$(\xi_n)_{n\in\mathbb N}$ and $(\eta_n)_{n\in\mathbb N}$ be
two sequences in $\R$. Assume that
$U_{A_n}(\cdot+\xi_n,\cdot)\to U^1$ and
$U_{A_n}(\cdot+\eta_n,\cdot)\to U^2$ in
$L^1_{\mathrm{loc}}(\R\times\T^{N-1})$ as $n\to+\infty$.
For $i=1,2$, let $p_i,q_i\in\mathcal E_f$ be the limits of $U^i$
as $x\to-\infty$ and $x\to+\infty$, respectively, given by
Lemma~\ref{lem:longitudinal-limits}. Then,  the following assertions hold.
\begin{enumerate}
\item[(i)] If $\eta_n-\xi_n\to\ell\in\R$, then
$U^2(x,y)=U^1(x+\ell,y)$ for almost every
$(x,y)\in\R\times\T^{N-1}$.
\item[(ii)] If $\eta_n-\xi_n\to+\infty$, then
$U^2(x,y)\leq q_1(y)$ for almost every
$(x,y)\in\R\times\T^{N-1}$ and, in particular,
$p_2\leq q_1$ on $\T^{N-1}$.
If $\eta_n-\xi_n\to-\infty$, then
$U^2(x,y)\geq p_1(y)$ for almost every
$(x,y)\in\R\times\T^{N-1}$ and, in particular,
$q_2\geq p_1$ on $\T^{N-1}$.
\end{enumerate}
\end{lemma}

\begin{proof}
Suppose first that $\eta_n-\xi_n\to\ell\in\R$.
By \eqref{norm1},
$\int_{\R\times\T^{N-1}}|U_{A_n,x}(x,y)|\,dx\,dy=1$.
Integrating the fundamental theorem of calculus between the two
translation parameters therefore gives
\begin{equation*}
	\|U_{A_n}(\cdot+\eta_n,\cdot)-U_{A_n}(\cdot+\xi_n+\ell,\cdot)
	\|_{L^1(\R\times\T^{N-1})}\leq|\eta_n-\xi_n-\ell|\to0.
\end{equation*}
Fix $R>0$. By the assumed local convergence, the two functions
on the left converge in $L^1((-R,R)\times\T^{N-1})$ to $U^2$
and $U^1(\cdot+\ell,\cdot)$, respectively. Letting $n\to+\infty$
thus gives
$\|U^2-U^1(\cdot+\ell,\cdot)\|_{L^1((-R,R)\times\T^{N-1})}=0$.
Since $R>0$ is arbitrary, this proves (i).

Suppose next that $\eta_n-\xi_n\to+\infty$.
Fix a positive integer $j$. For all sufficiently large $n$,
$\eta_n-\xi_n>j$, and monotonicity gives
\begin{equation*}
	U_{A_n}(x+\eta_n,y)\leq U_{A_n}(x+\xi_n+j,y)
	\quad\hbox{for }(x,y)\in\R\times\T^{N-1}.
\end{equation*}
The shift $j$ is fixed when $n\to+\infty$, so the two sides
converge in $L^1((-R,R)\times\T^{N-1})$ to $U^2(x,y)$ and
$U^1(x+j,y)$ for every $R>0$. Passing to an almost everywhere
convergent subsequence preserves the inequality, yielding
$U^2(x,y)\leq U^1(x+j,y)$ almost everywhere in
$\R\times\T^{N-1}$. To pass next to $j\to+\infty$, fix $R>0$
and use this comparison for $j>R$:
\begin{equation*}
	\|(U^2-q_1)_+\|_{L^1((-R,R)\times\T^{N-1})}
	\leq 2R\|U^1-q_1\|_{L^\infty((j-R,+\infty)\times\T^{N-1})}.
\end{equation*}
Here, $z_+$ denotes the positive part of $z$, and the factor $2R$
is the measure of $(-R,R)\times\T^{N-1}$.
The right-hand side tends to zero by
Lemma~\ref{lem:longitudinal-limits}. Hence, $U^2\leq q_1$
almost everywhere in $\R\times\T^{N-1}$.
Taking the limit of $U^2$ as $x\to-\infty$ gives
$p_2\leq q_1$ almost everywhere on $\T^{N-1}$, and then
everywhere by continuity of $p_2$ and $q_1$.

If $\eta_n-\xi_n\to-\infty$, by applying the preceding case with
the two sequences interchanged, we eventually obtain 
 $U^2\geq p_1$ a.e. on $\R\times\T^{N-1}$ and $q_2\geq p_1$ on $\T^{N-1}$.
This proves (ii).
\end{proof}

\begin{proposition}
\label{prop_alternative result}
Let $A_n\to+\infty$ and $\gamma_{A_n}\to\gamma$ as
$n\to+\infty$. After passing to a subsequence, still denoted by
$(A_n)_{n\in\mathbb N}$, exactly one of the following alternatives holds.
\begin{enumerate}
\item[(i)] There exist a sequence $(t_n)_{n\in\mathbb N}$ in $\R$
and a function $U\in L^\infty(\R\times\T^{N-1})$ solving
\eqref{eq:main-degenerate-front}, with
$U_x\leq0$ in $\D'(\R\times\T^{N-1})$ and
$\nabla_yU\in L^2_{\mathrm{loc}}(\R\times\T^{N-1})$, such that
\begin{equation*}
	U_{A_n}(\cdot+t_n,\cdot)\to U
	\quad\hbox{in }L^p_{\mathrm{loc}}(\R\times\T^{N-1})
\end{equation*}
for every $1\leq p<+\infty$, and the convergence also holds
almost everywhere in $\R\times\T^{N-1}$.

\item[(ii)] There exist $w\in\mathcal E_f\backslash\{0,1\}$,
sequences $(t_n^1)_{n\in\mathbb N},(t_n^2)_{n\in\mathbb N}$ in
$\R$ with $t_n^1<t_n^2$ and $t_n^2-t_n^1\to+\infty$, and
$x$-dependent functions $U^1,U^2\in L^\infty(\R\times\T^{N-1})$
such that, for $i=1,2$, $0\leq U^i\leq1$ almost everywhere in
$\R\times\T^{N-1}$, $\partial_xU^i\leq0$ in
$\D'(\R\times\T^{N-1})$,
$\nabla_yU^i\in L^2_{\mathrm{loc}}(\R\times\T^{N-1})$, and
\begin{equation*}
	L_\gamma U^i+f(U^i)=0\quad\hbox{in }\D'(\R\times\T^{N-1}),
\end{equation*}
with
\begin{equation}
U^1(-\infty,\cdot)=1,\quad U^1(+\infty,\cdot)=w,\quad
U^2(-\infty,\cdot)=w,\quad U^2(+\infty,\cdot)=0,
\label{eq:two}
\end{equation}
uniformly in $\T^{N-1}$, such that
\begin{equation*}
	U_{A_n}(\cdot+t_n^i,\cdot)\to U^i
	\quad\hbox{in }L^p_{\mathrm{loc}}(\R\times\T^{N-1}),\qquad i=1,2,
\end{equation*}
for every $1\leq p<+\infty$, and both convergences also hold
almost everywhere in $\R\times\T^{N-1}$.
For every sequence $(r_n)_{n\in\mathbb N}$ in $\R$ satisfying
\begin{equation}
r_n-t_n^1\to+\infty,\qquad t_n^2-r_n\to+\infty,
\label{eq:middle-center}
\end{equation}
one has $U_{A_n}(\cdot+r_n,\cdot)\to w$ in
$L^p_{\mathrm{loc}}(\R\times\T^{N-1})$ for every
$1\leq p<+\infty$, without extracting a further subsequence.
\end{enumerate}
\end{proposition}

\begin{proof}
	Let $A_n\to+\infty$ and $\gamma_{A_n}\to\gamma$ as $n\to+\infty$.
	For $v\in L^1(\T^{N-1})$, write
	$\langle v\rangle:=\int_{\T^{N-1}}v(y)\,dy$. Set
	$$
	m_n(x):=\int_0^1\big\langle U_{A_n}(x+s,\cdot)\big\rangle\,ds.
	$$
	Since $U_{A_n,x}<0$, while
	$U_{A_n}(-\infty,\cdot)=1$ and
	$U_{A_n}(+\infty,\cdot)=0$ uniformly in $\T^{N-1}$, the function $m_n:\R\to [0,1]$ is continuous and
	strictly decreasing. Fix $0<b<a<1$, and denote by
	$s_n(a)<s_n(b)$ the unique numbers satisfying
	$$
	m_n(s_n(a))=a,
	\qquad
	m_n(s_n(b))=b.
	$$
	
	After passing to a common subsequence,
	Lemma~\ref{lem:local-compactness} gives
	$$
	U_{A_n}(\cdot+s_n(a),\cdot)\to U^a,
	\qquad
	U_{A_n}(\cdot+s_n(b),\cdot)\to U^b
	$$
	in $L^p_{loc}(\R\times\T^{N-1})$ for every
	$1\leq p<+\infty$, and almost everywhere. By convergence in
	$L^1((0,1)\times\T^{N-1})$, we obtain
	$$
	\int_0^1\big\langle U^a(x,\cdot)\big\rangle\,dx=a,
	\qquad
	\int_0^1\big\langle U^b(x,\cdot)\big\rangle\,dx=b.
	$$

If both $U^a$ and $U^b$ were independent of $x$, they would
equal equilibria $w_a,w_b\in\mathcal E_f$ with averages $a,b$.
Since $s_n(a)<s_n(b)$, monotonicity gives $w_a\geq w_b$.
The inequalities $0<b<a<1$ would then imply
$1\succ w_a\succ w_b\succ0$, contrary to
Proposition~\ref{prop_Ef}~(iii).
Thus, at least one of $U^a,U^b$ depends on $x$.

	Choose $\kappa\in\{a,b\}$ such that $U^\kappa$ depends on $x$,
	and set for simplicity,
	$$
	U:=U^\kappa,\qquad t_n:=s_n(\kappa),\qquad
	m_n(t_n)=\int_0^1\langle U(x,\cdot)\rangle\,dx=\kappa.
	$$
	Lemma~\ref{lem:longitudinal-limits} implies that there exist $p,q\in\mathcal E_f$
	such that $p\succ q$, $U(-\infty,\cdot)=p$ and $U(+\infty,\cdot)=q$
	uniformly in $\T^{N-1}$. If $p\neq1$ and $q\neq0$, then
	$1\succ p\succ q\succ0,$
	contradicting Proposition~\ref{prop_Ef} (iii).
	Consequently, the ordered pair can only be
	$$
	(p,q)=(1,0),\qquad (p,q)=(1,w),
	\qquad\hbox{or}\qquad (p,q)=(w,0),
	$$
	for some $w\in\mathcal E_f\setminus\{0,1\}$ (therefore $0<\langle w\rangle<1$).
	
	\medskip
	
	\noindent
	{\it Case 1.}
	Suppose that $(p,q)=(1,0)$. This gives alternative (i). 
	Let $(\sigma_n)_{n\in\mathbb{N}}$ be another sequence in $\R$, and suppose that, after passing to a
	subsequence,
	$$
	U_{A_n}(\cdot+\sigma_n,\cdot)\to\widehat U
	\quad\hbox{in }L^1_{loc}(\R\times\T^{N-1})
	\quad\hbox{and almost everywhere}.
	$$
	After a further extraction, we get
	$
	\sigma_n-t_n\to\ell$ for some $\ell\in\R\cup\{-\infty,+\infty\}$.
	If $\ell\in\R$, Lemma~\ref{lem:two sequences} (i)
	gives
	$
	\widehat U(x,y)=U(x+\ell,y)$.
	Lemma~\ref{lem:two sequences}~(ii) implies that if $\ell=+\infty$,
	$\widehat U\leq q=0$, and  hence, $\widehat U\equiv0$;
	if $\ell=-\infty$, it follows that
	$\widehat U\geq p=1$, and hence, $\widehat U\equiv1$.
	Thus,  every $x$-dependent local limit is a translate of $U$, and
	no additional $x$-dependent transition component can
	occur. Thus, alternative~(ii) cannot occur simultaneously.

	\medskip
	
	\noindent
	{\it Case 2.}
	Suppose that $(p,q)=(1,w)$. Since
	$
	U(x,y)\geq w(y)$ for almost every $(x,y)\in\R\times\T^{N-1}$,
	one has $\kappa=\int_0^1\langle U(x,\cdot)\rangle\,dx\geq\langle w\rangle$.
	Let $\tau_n$ be the unique number satisfying
	$
	m_n(\tau_n)=\langle w\rangle/2.
	$
	Since $m_n$ is strictly decreasing,
	we have $
	\tau_n>t_n$.

	We claim that
	$
	\tau_n-t_n\to+\infty$.
	Otherwise, after extraction of a subsequence, $\tau_n-t_n\to\ell\in[0,+\infty)$.
	Lemma~\ref{lem:local-compactness}, applied to the translated sequence
	$U_{A_n}(\,\cdot+\tau_n,\cdot)$, yields a limit
	$\widehat U$, up to a subsequence. Lemma~\ref{lem:two sequences} (i) then yields
	$
	\widehat U(x,y)=U(x+\ell,y).
	$
	Moreover, convergence in $L^1((0,1)\times\T^{N-1})$ gives
	$
	\int_0^1\langle\widehat U(x,\cdot)\rangle\,dx
	=\langle w\rangle/2$.
	On the other hand, $
	U(x,y)\geq w(y)$ for almost every $(x,y)\in\R\times\T^{N-1}$ implies
	$
	\int_0^1\langle\widehat U(x,\cdot)\rangle\,dx
	=\int_0^1\langle U(x+\ell,\cdot)\rangle\,dx
	\geq\langle w\rangle,
	$
	which is a contradiction. Hence, $\tau_n-t_n\to+\infty$.

	After passing to a further subsequence, let $U^2$ be such that
	$$
	U_{A_n}(\cdot+\tau_n,\cdot)\to U^2
	\quad\hbox{in }L^p_{loc}(\R\times\T^{N-1})
	$$
	for every $1\leq p<+\infty$, and almost everywhere. Then, 
	$
	\int_0^1\langle U^2(x,\cdot)\rangle\,dx
	=\langle w\rangle/2$ and $
	U^2(x,y)\leq w(y)$ for almost every $(x,y)\in\R\times\T^{N-1}$,
	where the inequality follows from
	Lemma~\ref{lem:two sequences} (ii).
	If the function $U^2$ is independent of $x$, then
	$U^2(x,y)=v(y)$ for some $v\in\mathcal E_f$. Since
	$
	v\leq w$ and $
	0<\langle v\rangle
	=\langle w\rangle/2
	<\langle w\rangle$,
	one would have
	$1\succ w\succ v\succ0$,
	contrary to Proposition~\ref{prop_Ef} (iii). Thus, $U^2$ depends on $x$.
	It then follows from
	Lemma~\ref{lem:longitudinal-limits} that there exist
	$\widetilde p,\widetilde q\in\mathcal E_f$ such that $\widetilde p\succ\widetilde q$, $U^2(-\infty,\cdot)=\widetilde p$ and $U^2(+\infty,\cdot)=\widetilde q$,
	uniformly in $\T^{N-1}$. Since $U^2(x,y)\leq w(y)$ for almost every $(x,y)\in\R\times\T^{N-1}$,
	one has $\widetilde p\leq w$ on $\T^{N-1}$. If $\widetilde p\neq w$, then $1\succ w\succ\widetilde p\succ0$,
	contrary to Proposition~\ref{prop_Ef} (iii). Thus, 
	$\widetilde p=w$. If $\widetilde q\neq0$, then $1\succ w\succ\widetilde q\succ0$,
	which gives the same contradiction. Hence, $\widetilde q=0$. That is, $U^2(-\infty,\cdot)=w$ and $U^2(+\infty,\cdot)=0$,
	uniformly in $\T^{N-1}$.
	Setting
	$$
	U^1:=U,\qquad
	t_n^1:=t_n,\qquad
	t_n^2:=\tau_n,
	$$
	gives alternative~(ii), with
	$t_n^2-t_n^1\to+\infty$.

	\medskip
	
	\noindent
	{\it Case 3.}
	It remains to consider $(p,q)=(w,0)$. In this case $
	U(x,y)\leq w(y)$ for almost every $(x,y)\in\R\times\T^{N-1}$,
	and hence, $\kappa=\int_0^1\langle U(x,\cdot)\rangle\,dx\leq\langle w\rangle$.
	Let $\zeta_n$ be the unique number satisfying
	$$
	m_n(\zeta_n)=\frac{1+\langle w\rangle}{2}.
	$$
	Then,  $\zeta_n<t_n$.
	
	We claim that $t_n-\zeta_n\to+\infty.$
	Otherwise, after passing to a subsequence,
	$t_n-\zeta_n\to\ell\in[0,+\infty)$. The sequence
	$
	U_{A_n}(x+\zeta_n,y)$
	converges locally, up to a further subsequence, to a limit
	$\widehat U$ satisfying $\widehat U(x,y)=U(x-\ell,y)$ by
	Lemma~\ref{lem:two sequences} (i).
	Consequently, $\int_0^1\langle\widehat U(x,\cdot)\rangle\,dx
	\leq\langle w\rangle.$
	This contradicts
	$
	\int_0^1\langle\widehat U(x,\cdot)\rangle\,dx
	=(1+\langle w\rangle)/2
	>\langle w\rangle$.
	Therefore,  $t_n-\zeta_n\to+\infty$.
	
	After extraction of a subsequence, let $U^1$ be such that
	$$
	U_{A_n}(\cdot+\zeta_n,\cdot)\to U^1
	\quad\hbox{in }L^p_{loc}(\R\times\T^{N-1})
	$$
	for every $1\leq p<+\infty$, and almost everywhere. Then, 
	$
	\int_0^1\langle U^1(x,\cdot)\rangle\,dx
	=(1+\langle w\rangle)/2$ and $
	U^1(x,y)\geq w(y)$,
	where the inequality follows from
	Lemma~\ref{lem:two sequences} (ii).
	
	If
	the function $U^1$ is independent of $x$, then
	$U^1(x,y)=v(y)$ for some $v\in\mathcal E_f$, and it follows from
	$
	v\geq w$ and $
	\langle w\rangle
	<\langle v\rangle
	=(1+\langle w\rangle)/2<1$ that $1\succ v\succ w\succ0$,
	contrary to Proposition~\ref{prop_Ef} (iii). Therefore, $U^1$ depends on $x$. Lemma \ref{lem:longitudinal-limits} implies that there exist
	$\widetilde p,\widetilde q\in\mathcal E_f$ such that $\widetilde p\succ\widetilde q$, $U^1(-\infty,\cdot)=\widetilde{p}$ and $U^1(+\infty,\cdot)=\widetilde{q}$ uniformly in $\T^{N-1}$. Since $U^1(x,y)\ge w(y)$ for almost every $(x,y)\in\R\times\T^{N-1}$, we have
	$
	\widetilde q\geq w$.
	If $\widetilde q\neq w$, then
	$1\succ\widetilde q\succ w\succ0,$
	contrary to Proposition~\ref{prop_Ef} (iii). Hence, 
	$\widetilde q=w$. If $\widetilde p\neq1$, then
	$
	1\succ\widetilde p\succ w\succ0$,
	which is again impossible. Thus,  $\widetilde p=1$. Consequently,
	$U^1(-\infty,\cdot)=1$ and $U^1(+\infty,\cdot)=w$ uniformly in $\T^{N-1}$. By
	setting
	$$
	U^2:=U,\qquad
	t_n^1:=\zeta_n,\qquad
	t_n^2:=t_n,
	$$
	this again gives alternative~(ii), with
	$t_n^2-t_n^1\to+\infty$.

	\medskip

	Finally, let $(\sigma_n)\subset\R$, and suppose that
	$$
	U_{A_n}(\cdot+\sigma_n,\cdot)\to\widehat U
	\quad\hbox{ in }L^1_{loc}(\R\times\T^{N-1})\hbox{ and almost everywhere.}
	$$
	If $\sigma_n-t_n^i$ has a bounded subsequence for some $i\in\{1,2\}$, then,
	after extraction, Lemma~\ref{lem:two sequences} (i) shows
	that $\widehat U$ is a translate of $U^i$. It therefore remains
	to consider the case
	$$
	\min\{|\sigma_n-t_n^1|,|\sigma_n-t_n^2|\}\to+\infty.
	$$
	By Lemma~\ref{lem:longitudinal-limits}, there exists
	$\widehat q\in\mathcal E_f$ such that
	$
	\widehat U(+\infty,\cdot)=\widehat q$
	uniformly in $\T^{N-1}$. Since $\widehat U$ is nonincreasing in
	$x$, one has
	$
	\widehat U(x,y)\geq\widehat q(y)$
	for almost every $(x,y)\in\R\times\T^{N-1}$.

	After passing to a further subsequence, one of the following three
	cases occurs. If $t_n^1-\sigma_n\to+\infty$,
	Lemma~\ref{lem:two sequences} (ii) gives
	$
	U^1(x,y)\leq\widehat q(y)$.
	Letting $x\to-\infty$ yields $1\leq\widehat q$. Since
	$0\leq\widehat q\leq1$, it follows that $\widehat q\equiv1$, and hence, $\widehat U\equiv1$. Suppose next that
	$$
	\sigma_n-t_n^1\to+\infty,
	\qquad
	t_n^2-\sigma_n\to+\infty.
	$$
	Applying Lemma~\ref{lem:two sequences} (ii) gives
	$
	\widehat U\leq w$, and
	$
	U^2(x,y)\leq\widehat q(y)$.
	Letting $x\to-\infty$ yields $w\leq\widehat q$. Consequently,
	$
	w\ge \widehat U\geq\widehat q\geq w$, which implies $\widehat U\equiv w$.
	Finally, if $\sigma_n-t_n^2\to+\infty$,
	Lemma~\ref{lem:two sequences} (ii) gives
	$
	\widehat U\leq0$,
	and hence, $\widehat U\equiv0$. Thus, every $x$-dependent local limit is a translate of $U^1$ or
	$U^2$, and no additional $x$-dependent limiting profile can occur.

	Let $(r_n)_{n\in\mathbb N}$ satisfy
	\eqref{eq:middle-center}. Consider an arbitrary subsequence of
	$U_{A_n}(\cdot+r_n,\cdot)$. By
	Lemma~\ref{lem:local-compactness}, it has a further subsequence
	converging in $L^1_{loc}(\R\times\T^{N-1})$, and almost
	everywhere, to a function $Z$. Since
	$$
	r_n-t_n^1\to+\infty,
	\qquad
	r_n-t_n^2\to-\infty,
	$$
	Lemma~\ref{lem:two sequences} (ii), applied with
	$\xi_n=t_n^1$ and $\eta_n=r_n$, gives $Z\leq w$, whereas, applied with
	$\xi_n=t_n^2$ and $\eta_n=r_n$, gives $Z\geq w$.
	Hence, $Z=w$. Thus,  every subsequence has a further subsequence converging to $w$
	in $L^1_{loc}(\R\times\T^{N-1})$. Consequently,
	$$
	U_{A_n}(\cdot+r_n,\cdot)\to w
	\quad\hbox{in }L^1_{loc}(\R\times\T^{N-1}).
	$$
	Since $0\leq U_{A_n}\leq 1$ and $0\leq w\leq1$ on $\R\times\T^{N-1}$, for every compact set
	$K\subset\R\times\T^{N-1}$ and every $1\leq p<+\infty$,
	$$
	\|U_{A_n}(\cdot+r_n,\cdot)-w\|_{L^p(K)}^p
	\leq
	\|U_{A_n}(\cdot+r_n,\cdot)-w\|_{L^1(K)}.
	$$
	Thus,  the convergence holds in $L^p_{loc}(\R\times\T^{N-1})$ for every
	$1\leq p<+\infty$.
	
	\medskip
	
	In {\it Cases~2} and {\it 3}, every $x$-dependent translated limit is a translate of
	either $U^1$ or $U^2$; since neither connects $1$ directly to $0$,
	alternative~(i) cannot hold. Thus,  the two alternatives are mutually
	exclusive.
\end{proof}

It remains to exclude two connections $1\to w$ and $w\to0$ with
the same speed. The following lemma concerns nontrivial nonnegative solutions
that are nonincreasing in $x$ and converge to zero as
$x\to+\infty$. It shows that, if the principal eigenvalue of
$-\Delta_y-a_\infty(y)$ is negative, the existence of such a solution
requires the drift to satisfy a sign condition. We also point out that the weak solution in the lemma is allowed to vanish and is not
assumed to be continuous in $x$.

\begin{lemma}
\label{lem:tailsign}
Let $\gamma\in\R$ and $\alpha\in C^{1,\delta}(\T^{N-1})$
satisfy $\gamma<\max_{\T^{N-1}}\alpha$, and let
$a\in L^\infty(\R\times\T^{N-1})$. Suppose that
$v\in L^\infty(\R\times\T^{N-1})$ satisfies
$\nabla_yv\in L^2_{\mathrm{loc}}(\R\times\T^{N-1})$ and
\begin{equation}
\Delta_yv+(\gamma-\alpha(y))v_x+a(x,y)v=0\quad\hbox{in }\D'(\R\times\T^{N-1}),
\label{eq:v}
\end{equation}
with $v\geq0$ almost everywhere in $\R\times\T^{N-1}$ and
$v_x\leq0$ in $\D'(\R\times\T^{N-1})$.
Assume that there exist $v_-\in L^\infty(\T^{N-1})$ and
$a_\infty\in C^{0,\delta}(\T^{N-1})$, with
$\operatorname*{ess\,inf}_{\T^{N-1}}v_->0$, such that as $R\to +\infty$,
\begin{equation}
\begin{aligned}
\|v-v_-\|_{L^\infty((-\infty,-R)\times\T^{N-1})}
 +\|v\|_{L^\infty((R,+\infty)\times\T^{N-1})}\to0,~~~\|a-a_\infty\|_{L^\infty((R,+\infty)\times\T^{N-1})}\to0.
\end{aligned}
\label{eq:vend}
\end{equation}
Let $\lambda_1$ be the principal periodic eigenvalue of
$-\Delta_y-a_\infty(y)$, and let
$\varphi\in C^{2,\delta}(\T^{N-1})$ be the corresponding
positive eigenfunction normalized by $\int_{\T^{N-1}}\varphi(y)^2\,dy=1$.
If $\lambda_1<0$, then
\begin{equation*}
B:=\int_{\T^{N-1}}(\gamma-\alpha(y))\varphi(y)^2\,dy>0.
\label{eq:Bpositive}
\end{equation*}
\end{lemma}

\begin{proof}
We divide into several steps. Steps~1-3 construct a positive longitudinal regularization whose
coefficient converges to $a_\infty$, and establish a uniform
transverse Harnack estimate \eqref{B}. In Steps~4-5, this estimate gives
a lower bound for the weighted logarithmic integral  \eqref{log integral} if $B\leq0$,
whereas the logarithmic identity \eqref{eq:logidentity} and $\lambda_1<0$ force this
integral to tend to $-\infty$ as $x\to+\infty$.
This contradiction proves $B>0$.

\medskip
\noindent
{\it Step 1. A lower bound on a transverse subdomain.}
Since $\gamma-\alpha$ is continuous and has a negative value,
there exist a smooth connected domain $\Omega\subset\T^{N-1}$ and $\beta_0>0$ such that
$\gamma-\alpha\leq-\beta_0$ on $\overline{\Omega}$.
Let $\lambda_\Omega>0$ and $\psi$ satisfy
\begin{align*}
	\begin{cases}
		-\Delta_y\psi=\lambda_\Omega\psi,
		&\text{in }\Omega,\\
		\psi=0,
		&\text{on }\partial\Omega,\\
		0<\psi\leq\max_{\overline{\Omega}}\psi=1,
		&\text{in }\Omega.
	\end{cases}
\end{align*}
Set $K:=\|a\|_{L^\infty(\R\times\T^{N-1})}$ and
$C:=(\lambda_\Omega+K+1)/\beta_0$.
Choose $\eta>0$ and $x_0\in\R$ sufficiently negative that
$v\geq\eta$ almost everywhere in $(-\infty,x_0+1)\times\Omega$.
This is possible by the left-hand limit in \eqref{eq:vend}
and the positive essential lower bound of $v_-$.
On $\overline{\Omega}$, equation \eqref{eq:v}
can be written as
$(\alpha(y)-\gamma)v_x-\Delta_yv-a(x,y)v=0$,
with $\alpha-\gamma$ bounded away from zero.

Define
$z_0(x,y):=\eta e^{-C(x-x_0)}\psi(y)$ for all $x\geq x_0$ and $y\in\overline{\Omega}$.
A direct calculation gives
\begin{equation*}
	(\Delta_y+(\gamma-\alpha(y))\partial_x+a(x,y))z_0
	=(-\lambda_\Omega-C(\gamma-\alpha(y))+a(x,y))z_0
	\geq z_0\geq0
\end{equation*}
almost everywhere in $(x_0,+\infty)\times\Omega$.
Since $z_0\leq\eta$ for all $x\geq x_0$ and $y\in\overline{\Omega}$, we have
$z_0\leq v$ almost everywhere in $(x_0,x_0+1)\times\Omega$.
Moreover, $z_0=0\leq v$ on the lateral boundary in the
Sobolev trace sense.

For $0<h<1$, define the forward average
\begin{equation*}
	w_h(x,y):=\frac1h\int_x^{x+h}(z_0-v)(s,y)\,ds.
\end{equation*}
Averaging the inequality for $z_0-v$ gives
\begin{equation*}
	(\alpha(y)-\gamma)\partial_xw_h-\Delta_yw_h
	\leq\frac1h\int_x^{x+h}a(s,y)(z_0-v)(s,y)\,ds
\end{equation*}
in distributions on $(x_0,+\infty)\times\Omega$. 
The comparison on the initial interval gives
$w_h(x_0,\cdot)\leq0$, while the lateral boundary inequality
gives $(w_h)_+(x,\cdot)\in H_0^1(\Omega)$ for almost every $x$.
Furthermore,
$\partial_xw_h(x,y)=
((z_0-v)(x+h,y)-(z_0-v)(x,y))/h$
belongs to $L^2((x_0,X)\times\Omega)$ for every finite $X>x_0$.
We may therefore test the averaged inequality by $(w_h)_+$
and integrate from $x_0$ to $X$.
The initial energy is zero because $(w_h)_+(x_0,\cdot)=0$.

The averages of $z_0-v$ and $a(z_0-v)$ converge to the
corresponding functions in $L^2((x_0,X)\times\Omega)$ as
$h\to0$. Moreover, the Lebesgue differentiation theorem for
$L^2(\Omega)$-valued functions gives
\begin{equation*}
	\|w_h(X,\cdot)-(z_0-v)(X,\cdot)\|_{L^2(\Omega)}\to0
	\quad\hbox{as }h\to0
\end{equation*}
for almost every $X>x_0$. Since the positive-part map is
Lipschitz continuous and $\alpha-\gamma$ is bounded, the
terminal weighted energies converge at these sections.
The local transverse energy bounds give weak compactness of
$\nabla_y(w_h)_+$, and strong $L^2$ convergence identifies its
weak limit as $\nabla_y(z_0-v)_+$. Passing to the limit in the
integrated inequality and using weak lower semicontinuity of
the gradient term, we obtain, for almost every $X>x_0$,
\begin{align*}
	&\frac12\int_\Omega(\alpha(y)-\gamma)
	(z_0-v)_+^2(X,y)\,dy
	+\int_{x_0}^{X}\int_\Omega
	|\nabla_y(z_0-v)_+|^2\,dy\,dx\\
	&\qquad\leq K\int_{x_0}^{X}\int_\Omega
	(z_0-v)_+^2\,dy\,dx
	\leq\frac{K}{\beta_0}\int_{x_0}^{X}\int_\Omega
	(\alpha(y)-\gamma)(z_0-v)_+^2\,dy\,dx.
\end{align*}
The integral form of Gronwall's inequality, applied on each
finite interval, gives $(z_0-v)_+=0$ almost everywhere.
Consequently,
\begin{equation}
	v(x,y)\geq\eta e^{-C(x-x_0)}\psi(y)
	\quad\hbox{for a.e. }(x,y)\in(x_0,+\infty)\times\Omega.
	\label{eq:patchlower}
\end{equation}

\medskip
\noindent
{\it Step 2. Longitudinal regularization.}
Fix $k>C$ and set $\rho_k(s):=(k/2)e^{-k|s|}$.
Then,  $\rho_k\in W^{1,1}(\R)$, $\rho_k>0$,
$\int_\R\rho_k(s)\,ds=1$, and $|\rho_k'|=k\rho_k$ almost everywhere.
Define
\begin{equation}
	z(x,y):=\int_\R\rho_k(x-s)v(s,y)\,ds,\qquad
	q(x,y):=\int_\R\rho_k(x-s)a(s,y)v(s,y)\,ds.
	\label{eq:expconv}
\end{equation}
These integrals are well defined for every $x\in\R$ and almost
every $y\in\T^{N-1}$, since $\rho_k\in L^1(\R)$ and
$v,av\in L^\infty(\R\times\T^{N-1})$.
Differentiating $z(x,y)$ with respect to $x$ in the distributional sense
gives $z_x(x,y)=\int_\R\rho_k'(x-s)v(s,y)\,ds$.
Since $v\geq0$, $|a|\leq K$, and $|\rho_k'|=k\rho_k$, we obtain
\begin{equation*}
	0\leq z\leq\|v\|_{L^\infty(\R\times\T^{N-1})},
	\qquad |z_x|\leq kz,\qquad |q|\leq Kz,
\end{equation*}
almost everywhere in $\R\times\T^{N-1}$.
The assumption $v_x\leq0$ also gives $z_x\leq0$.

To justify convolution of \eqref{eq:v} with $\rho_k$, choose
$\rho_{k,j}\in C_c^\infty(\R)$ such that
$\|\rho_{k,j}-\rho_k\|_{W^{1,1}(\R)}\to0$ as $j\to+\infty$.
Let $z_j$ and $q_j$ be defined by \eqref{eq:expconv} with
$\rho_k$ replaced by $\rho_{k,j}$.
For any $\chi\in C_c^\infty(\R\times\T^{N-1})$, the function
$\chi_j(s,y):=\int_\R\rho_{k,j}(x-s)\chi(x,y)\,dx$
belongs to $C_c^\infty(\R\times\T^{N-1})$.
Testing \eqref{eq:v} with $\chi_j$ and applying Fubini's theorem
gives
\begin{equation*}
	\int_{\R\times\T^{N-1}}
	\bigl(z_j\Delta_y\chi+(\gamma-\alpha(y))(z_j)_x\chi+q_j\chi\bigr)
	\,dx\,dy=0.
\end{equation*}

Since $v,av$ are bounded, the convolution estimates imply
\begin{align*}
	&\|z_j-z\|_{L^\infty(\R\times\T^{N-1})}
	+\|(z_j)_x-z_x\|_{L^\infty(\R\times\T^{N-1})}
	+\|q_j-q\|_{L^\infty(\R\times\T^{N-1})}\\
	&~~~~~~~~~~~~~~~~~~~~~~~~~\qquad\leq(1+K)\|v\|_{L^\infty(\R\times\T^{N-1})}
	\|\rho_{k,j}-\rho_k\|_{W^{1,1}(\R)}\to0.
\end{align*}
Passing to the limit in the preceding weak formulation gives
\begin{equation}
	\Delta_yz+(\gamma-\alpha(y))z_x+q=0
	\quad\hbox{in }\D'(\R\times\T^{N-1}).
	\label{eq:z}
\end{equation}

We next show that $z$ has a positive lower bound on every finite
longitudinal interval, uniformly in $y$.
The first limit in \eqref{eq:vend} and
$\operatorname*{ess\,inf}_{\T^{N-1}}v_->0$ give constants
$c>0$ and $X_0\in\R$ such that $v\geq c$ almost everywhere in
$(-\infty,X_0)\times\T^{N-1}$.
Consequently,
\begin{equation*}
	z(x,y)\geq c\int_{-\infty}^{X_0}\rho_k(x-s)\,ds
	=c\int_{x-X_0}^{+\infty}\rho_k(s)\,ds
\end{equation*}
for every $x\in\R$ and almost every $y\in\T^{N-1}$.
The last integral is positive and nonincreasing in $x$. Thus,
\begin{equation*}
	\operatorname*{ess\,inf}_{[-R,R]\times\T^{N-1}}z
	\geq c\int_{R-X_0}^{+\infty}\rho_k(s)\,ds>0
	\qquad\hbox{for every }R>0.
\end{equation*}

We now establish the transverse regularity of $z$.
First, continuity of translations in $L^1(\R)$ gives
\begin{equation*}
	\|z(x+h,\cdot)-z(x,\cdot)\|_{L^\infty(\T^{N-1})}
	\leq
	\|\rho_k(\cdot+h)-\rho_k\|_{L^1(\R)}
	\|v\|_{L^\infty(\R\times\T^{N-1})}\to0
	\quad\hbox{as }h\to0.
\end{equation*}
The same argument applied to $\rho_k' *_x v$ and
$\rho_k *_x(av)$ shows that both $z_x$ and $q$ belong to $ C(\R;L^\infty(\T^{N-1}))$.
Moreover, since the difference quotients of $\rho_k$ converge
to $\rho_k'$ in $L^1(\R)$, the difference quotients of $z$
converge to $z_x$ in $L^\infty(\T^{N-1})$.
Hence, $z\in C^1(\R;L^\infty(\T^{N-1}))$.

Testing \eqref{eq:z} against smooth periodic functions of $y$
therefore gives an identity continuous in $x$, so that
\begin{equation*}
	\Delta_yz(x,\cdot)=-(\gamma-\alpha)z_x(x,\cdot)-q(x,\cdot)
	\quad\hbox{in }\D'(\T^{N-1})
\end{equation*}
for every $x\in\R$.
The bounds $|z_x|\leq kz$ and $|q|\leq Kz$ extend to every
section in $L^\infty(\T^{N-1})$ by this continuity.
Thus,  the right-hand side is bounded in $L^\infty(\T^{N-1})$
uniformly in $x$.
Applying \eqref{eq:section4-elliptic-estimate}, we obtain
\begin{equation*}
	\sup_{x\in\R}\|z(x,\cdot)\|_{W^{2,r}(\T^{N-1})}<+\infty
	\qquad\hbox{for every }1<r<+\infty.
\end{equation*}

For $r>N-1$, Sobolev embedding and the continuity in $x$
give a jointly continuous representative. The preceding
lower bound then yields
$\min_{[-R,R]\times\T^{N-1}}z>0$ for every $R>0$.
We use this representative below.
We may consequently define $a_k:=q/z$.
The estimate $|q|\leq Kz$ and \eqref{eq:z} give
\begin{equation}
	\Delta_yz+(\gamma-\alpha(y))z_x+a_k(x,y)z=0~~~\text{in}~\D'(\R\times\T^{N-1}),
	\label{eq:ak}
\end{equation}
with $\|a_k\|_{L^\infty(\R\times\T^{N-1})}\leq K$.

\medskip
\noindent
{\it Step 3. Convergence of the regularized coefficient.}
We first establish a lower bound for $z$, uniform in $y$, with
the exponent $C$ from Step~1.
For $x\geq x_0$, restricting the convolution in
\eqref{eq:expconv} to $s\in[x,x+1]$ and using
\eqref{eq:patchlower}, we obtain
\begin{equation}
	\label{A}
\begin{aligned}
	\int_{\T^{N-1}}z(x,y)\,dy
	&\geq\int_x^{x+1}\rho_k(x-s)
	\int_\Omega v(s,y)\,dy\,ds\\
	&\geq\frac{k\eta}{2}
	\left(\int_\Omega\psi(y)\,dy\right)e^{-C(x-x_0)}
	\int_0^1e^{-(k+C)t}\,dt.
\end{aligned}
\end{equation}

We next use the elliptic Harnack inequality to turn this
integral estimate into a pointwise lower bound.
Set $m:=1+k\|\gamma-\alpha\|_{L^\infty(\T^{N-1})}+K$.
By Step~2, $z(x,\cdot)>0$ and
$z(x,\cdot)\in W^{2,r}(\T^{N-1})$ for every $x\in\R$
and $1<r<+\infty$.
For each fixed $x$, equation \eqref{eq:ak} can be written as
\begin{equation*}
	\Delta_yz(x,y)
	+\left((\gamma-\alpha(y))\frac{z_x(x,y)}{z(x,y)}
	+a_k(x,y)\right)z(x,y)=0
	\quad\hbox{in }\D'(\T^{N-1}).
\end{equation*}
The zero-order coefficient is bounded independently of $x$,
since
\begin{equation*}
	\left\|(\gamma-\alpha)\frac{z_x(x,\cdot)}{z(x,\cdot)}
	+a_k(x,\cdot)\right\|_{L^\infty(\T^{N-1})}
	\leq k\|\gamma-\alpha\|_{L^\infty(\T^{N-1})}+K=m-1.
\end{equation*}
The elliptic Harnack inequality, applied on a fixed finite
cover of the connected torus, therefore gives a constant
$g_m\in(0,1]$, depending only on $m$ and $N$, such that
\begin{equation}
	\label{B}
	\min_{\T^{N-1}}z(x,\cdot)
	\geq g_m\max_{\T^{N-1}}z(x,\cdot)
	\quad\hbox{for every }x\in\R.
\end{equation}
Since the torus has unit volume, it follows that
\begin{equation*}
	z(x,y)\geq g_m\int_{\T^{N-1}}z(x,\zeta)\,d\zeta
	\quad\hbox{for }(x,y)\in\R\times\T^{N-1}.
	\label{eq:resolventlower}
\end{equation*}
Combining this with 
\eqref{A}, we conclude that
\begin{equation}
	z(x,y)\geq c_ke^{-Cx}
	\quad\hbox{for }x\geq x_0,\ y\in\T^{N-1},
	\label{eq:globallower}
\end{equation}
for some constant
$c_k>0$.

We now use this lower bound to control $a_k-a_\infty$.
The last limit in \eqref{eq:vend} and the bound $|a|\leq K$
imply $\|a_\infty\|_{L^\infty(\T^{N-1})}\leq K$.
Given $\eta_1>0$, choose $R\geq x_0$ sufficiently large that
$\|a-a_\infty\|_{L^\infty((R,+\infty)\times\T^{N-1})}
\leq\eta_1$.
For every $x>R$ and almost every $y\in\T^{N-1}$,
the definitions of $q,z$ and a splitting at $s=R$ give
\begin{align*}
	|q(x,y)-a_\infty(y)z(x,y)|
	&=\left|\int_\R\rho_k(x-s)
	(a(s,y)-a_\infty(y))v(s,y)\,ds\right|\\
	&\leq\eta_1z(x,y)
	+2K\|v\|_{L^\infty(\R\times\T^{N-1})}
	\int_{-\infty}^{R}\rho_k(x-s)\,ds\\
	&=\eta_1z(x,y)
	+K\|v\|_{L^\infty(\R\times\T^{N-1})}e^{-k(x-R)},
\end{align*}
where  we have used $|a-a_\infty|\leq\eta_1$ on $s>R$ and
$|a-a_\infty|\leq2K$ on $s\leq R$.
Dividing by $z(x,y)$ and applying \eqref{eq:globallower}
therefore yields
\begin{equation*}
	\|a_k(x,\cdot)-a_\infty\|_{L^\infty(\T^{N-1})}
	\leq\eta_1+
	\frac{K\|v\|_{L^\infty(\R\times\T^{N-1})}}{c_k}
	e^{kR}e^{-(k-C)x}
	\quad\hbox{for }x>R.
\end{equation*}
Since $k>C$, the second term tends to zero as $x\to+\infty$
for this fixed $R$.
Taking the upper limit and then letting $\eta_1\to 0$
proves
\begin{equation}
	\lim_{x\to+\infty}
	\|a_k(x,\cdot)-a_\infty\|_{L^\infty(\T^{N-1})}=0.
	\label{eq:coefficientlimit}
\end{equation}

\medskip
\noindent
{\it Step 4. A weighted identity.}
Since $\min_{\T^{N-1}}\varphi>0$, we have
$
0<z(x,y)\leq M\varphi(y)$
for $(x,y)\in\R\times\T^{N-1}$, with $M:=\|v\|_{L^\infty(\R\times\T^{N-1})}/
\min_{\T^{N-1}}\varphi>0$.
Define 
\begin{equation*}
		\ell(x,y):=\log\frac{z(x,y)}{M\varphi(y)}.
\end{equation*}
Set
\begin{equation*}
	\begin{aligned}
		S(x):=\int_{\T^{N-1}}\varphi(y)^2\ell(x,y)\,dy,~~E(x):=\int_{\T^{N-1}}\varphi(y)^2
		|\nabla_y\ell(x,y)|^2\,dy,
	\end{aligned}
\end{equation*}
and 
\begin{equation}\label{log integral}
	T(x):=\int_{\T^{N-1}}(\gamma-\alpha(y))
	\varphi(y)^2\ell(x,y)\,dy.
\end{equation}
The choice of $M$ ensures that $\ell\leq0$ and hence, $S\leq0$.
The regularity and the positive lower bound established in Step~2
give, for every $R>0$ and $1<r<+\infty$,
\begin{equation*}
	\sup_{|x|\leq R}\|\ell(x,\cdot)\|_{W^{2,r}(\T^{N-1})}
	+\|\ell_x\|_{L^\infty((-R,R)\times\T^{N-1})}<+\infty.
\end{equation*}
In particular, the following transverse integrations by parts are
legitimate, and $S,T$ are locally absolutely continuous on $\R$.

Dividing \eqref{eq:ak} by $z$ and using
$z=M\varphi e^\ell$, we find
\begin{equation*}
	\frac{\Delta_yz}{z}
	=\frac{\Delta_y\varphi}{\varphi}
	+\Delta_y\ell+2\nabla_y\log\varphi\cdot\nabla_y\ell
	+|\nabla_y\ell|^2.
\end{equation*}
Since
\begin{equation*}
	\int_{\T^{N-1}}
	\bigl(\varphi^2\Delta_y\ell
	+2\varphi\nabla_y\varphi\cdot\nabla_y\ell\bigr)\,dy
	=\int_{\T^{N-1}}\operatorname{div}_y(\varphi^2\nabla_y\ell)\,dy=0,
\end{equation*}
it follows from $-\Delta_y\varphi-a_\infty(y)\varphi=\lambda_1\varphi$ with
 $\int_{\T^{N-1}}\varphi(y)^2\,dy=1$ that
\begin{equation}
T'(x)+E(x)-\lambda_1
 +\int_{\T^{N-1}}(a_k(x,y)-a_\infty(y))\varphi(y)^2\,dy=0
\label{eq:logidentity}
\end{equation}
for almost every $x\in\R$.
By \eqref{eq:coefficientlimit} and $\lambda_1<0$, there is
$X\in\R$ such that
$\|a_k(x,\cdot)-a_\infty\|_{L^\infty(\T^{N-1})} \leq-\frac{\lambda_1}{2}$ for $x\geq X$.
Consequently,
\begin{equation}
T'(x)\leq\frac{\lambda_1}{2}-E(x)
\quad\hbox{for a.e. }x\geq X.
\label{eq:Tderivative}
\end{equation}

\medskip
\noindent
{\it Step 5. The sign of $B$.}
Since $\int_{\T^{N-1}}\varphi^2\,dy=1$,
the definition of $S$ gives
$\min_{\T^{N-1}}\ell(x,\cdot)\leq S(x)
\leq\max_{\T^{N-1}}\ell(x,\cdot)$.
Using the definition of $\ell$, we therefore obtain
\begin{align*}
	\|\ell(x,\cdot)-S(x)\|_{L^\infty(\T^{N-1})}
	\leq \max_{\T^{N-1}}\ell(x,\cdot)
	-\min_{\T^{N-1}}\ell(x,\cdot)\leq \log\left(
	\frac{\max_{\T^{N-1}}\varphi}
	{g_m\min_{\T^{N-1}}\varphi}\right)
\end{align*}
for every $x\in\R$.

Since $\int_{\T^{N-1}}(\gamma-\alpha(y)-B)
\varphi(y)^2\,dy=0$, the definitions of $T$ and $S$,
together with the Cauchy-Schwarz inequality, yield
\begin{equation}
	\begin{aligned}
		|T(x)-BS(x)|
		=\left|\int_{\T^{N-1}}(\gamma-\alpha(y)-B)
		\varphi(y)^2(\ell(x,y)-S(x))\,dy\right|
		\leq C_\alpha,
	\end{aligned}
	\label{eq:weightedpoincare}
\end{equation}
where the constant
\begin{equation*}
	C_\alpha:=
	\|(\gamma-\alpha-B)\varphi\|_{L^2(\T^{N-1})}
	\log\left(
	\frac{\max_{\T^{N-1}}\varphi}
	{g_m\min_{\T^{N-1}}\varphi}\right)
\end{equation*}
is finite and independent of $x$.

Suppose, for contradiction, that $B\leq0$.
Since $S(x)\leq0$, we have $BS(x)\geq0$, and
\eqref{eq:weightedpoincare} gives $T(x)\geq-C_\alpha$
for every $x\in\R$.
On the other hand, $E(x)\geq0$ and
\eqref{eq:Tderivative} imply
$T'(x)\leq\lambda_1/2$ for almost every $x\geq X$.
As $T$ is locally absolutely continuous, integration gives
\begin{equation*}
	-C_\alpha\leq T(x)
	\leq T(X)+\frac{\lambda_1}{2}(x-X)
	\quad\hbox{for every }x\geq X.
\end{equation*}
Since $\lambda_1<0$, the right-hand side tends to $-\infty$
as $x\to+\infty$, contradicting the lower bound.
Therefore, $B>0$.
\end{proof}

Suppose $w\in\mathcal{E}_f\backslash\{0,1\}$, then
Lemma \ref{lem:tailsign}, applied to $U^1-w$ and, after reflection in $x$, to $w-U^2$,
gives incompatible conditions on the same weighted drift
integral whenever $\gamma-\alpha\not\equiv0$.
\begin{proposition}
	\label{prop:exclude-two-components}
	Suppose that 
	$\alpha$ is nonconstant on $\T^{N-1}$.
	Then,  alternative~(ii) of Proposition~\ref{prop_alternative result}
	cannot occur. 
\end{proposition}

\begin{proof}
Assume, for contradiction, that such $w\in\mathcal{E}_f\backslash\{0,1\}$,$U^1$ and $U^2$ exist.
Monotonicity and \eqref{eq:two}
give
$0\leq U^2(x,y)\leq w(y)\leq U^1(x,y)\leq1 $ for a.e. $(x,y)\in\R\times\T^{N-1}$.
Define $v_1:=U^1-w$ and $v_2:=w-U^2$, and set
\begin{equation*}
	a_1(x,y):=\int_0^1 f'(w(y)+t v_1(x,y))\,dt,\qquad
	a_2(x,y):=\int_0^1 f'(w(y)-t v_2(x,y))\,dt.
\end{equation*}
For $i=1,2$, we have
$v_i\in L^\infty(\R\times\T^{N-1})$,
$\nabla_yv_i\in L^2_{\mathrm{loc}}(\R\times\T^{N-1})$, and
\begin{equation*}
	\Delta_yv_i+(\gamma-\alpha(y))\partial_xv_i+a_i(x,y)v_i=0
	\quad\hbox{in }\D'(\R\times\T^{N-1}),
\end{equation*}
and
$\|a_i\|_{L^\infty(\R\times\T^{N-1})} \leq\|f'\|_{L^\infty([0,1])}$.
Moreover, $\partial_xv_1\leq0$ and $\partial_xv_2\geq0$ in
$\D'(\R\times\T^{N-1})$.

Recall that $\lambda_1(w)$ is the principal eigenvalue of
$-\Delta_y-f'(w)$. We claim that $\lambda_1(w)<0$. Indeed, if $w=\theta$,
this follows from $f'(\theta)>0$. If $w$ is nonconstant, 
 Proposition~\ref{prop_Ef}~(ii) applies. Also,
$\min_{\T^{N-1}}w>0$ and $\min_{\T^{N-1}}(1-w)>0$.
Let $\varphi>0$ be the corresponding eigenfunction, so that
$-\Delta_y\varphi-f'(w(y))\varphi=\lambda_1(w)\varphi$.
Normalize it by $\int_{\T^{N-1}}\varphi(y)^2\,dy=1$, and put $$B:=\int_{\T^{N-1}}(\gamma-\alpha(y))\varphi(y)^2\,dy.$$

Suppose that $\gamma<\max_{\T^{N-1}}\alpha$. We apply Lemma~\ref{lem:tailsign} to $v_1$. To be precise,  the  limits in \eqref{eq:two} imply
\begin{equation*}
	\|v_1-(1-w)\|_{L^\infty((-\infty,-R)\times\T^{N-1})}
	+\|v_1\|_{L^\infty((R,+\infty)\times\T^{N-1})}\to0
\end{equation*}
as $R\to+\infty$.
Since $f'$ is H\"older continuous on $[0,1]$, there is a
constant $C_f>0$ such that
\begin{equation*}
	\|a_1-f'(w)\|_{L^\infty((R,+\infty)\times\T^{N-1})}
	\leq C_f
	\|v_1\|_{L^\infty((R,+\infty)\times\T^{N-1})}^{\delta}\to0.
\end{equation*}
All assumptions of Lemma~\ref{lem:tailsign} hold with
$v_-=1-w$ and $a_\infty=f'(w)$. Therefore,
$B>0$.

For the second component, set
$\widetilde v_2(x,y):=v_2(-x,y)$ and
$\widetilde a_2(x,y):=a_2(-x,y)$. Then,
\begin{equation*}
	\Delta_y\widetilde v_2+(\alpha(y)-\gamma)\partial_x\widetilde v_2
	+\widetilde a_2\widetilde v_2=0
	\quad\hbox{in }\D'(\R\times\T^{N-1}).
\end{equation*}
Reflection preserves the boundedness and local transverse energy,
and $\partial_x\widetilde v_2\leq0$ in
$\D'(\R\times\T^{N-1})$.
The limits in \eqref{eq:two} give
$\widetilde v_2(-\infty,\cdot)=w$ and
$\widetilde v_2(+\infty,\cdot)=0$ uniformly in $\T^{N-1}$.
The same H\"older estimate as above gives
\begin{equation*}
	\|\widetilde a_2-f'(w)\|_{L^\infty((R,+\infty)\times\T^{N-1})}
	\leq C_f\|\widetilde v_2\|_{L^\infty((R,+\infty)\times\T^{N-1})}^{\delta}
	\to0.
\end{equation*}
Thus, if $\gamma>\min_{\T^{N-1}}\alpha$,
Lemma~\ref{lem:tailsign}, with $(-\gamma,-\alpha)$ in place
of $(\gamma,\alpha)$, gives $-B>0$.

If $\gamma-\alpha$ changes sign, these conclusions contradict one another.
If $\gamma-\alpha\leq0$ and $\gamma-\alpha\not\equiv0$, the first conclusion gives
$B>0$, whereas continuity of $\gamma-\alpha$ and positivity of $\varphi$
give $B<0$. If $\gamma-\alpha\geq0$ and $\gamma-\alpha\not\equiv0$, the second
conclusion gives $B<0$, whereas the same integral gives $B>0$.
Therefore, all these cases are impossible. This proves the proposition.
\end{proof}

\section{Global convergence and proof of Theorem \ref{thm:main}}
\label{sec5}

This section is devoted to the proof of Theorem~\ref{thm:main}.
We first treat constant shears, for which the fronts are explicit.
We then establish a global convergence lemma: uniform tail
estimates and integral identities allow us to strengthen local
convergence to a full front into convergence of the profiles
and their transverse gradients on the whole cylinder.
After proving convergence of $c_A/A$, we fix the normalization
and apply this lemma to the full limits obtained in
Section~\ref{sec4}.
We next prove the strict speed bounds provided that the sets where $\alpha$ attains its minimum
and maximum have zero Lebesgue measure.
Finally, under \eqref{cdn_alpha}, regularity and uniqueness
of the normalized limiting front upgrade subsequential
convergence to global and uniform convergence of the whole family.

\subsection{Constant shears}
\label{subsec:constant}
Suppose $\alpha\equiv\alpha_0$, and let $(c_0,q)$ be the
one-dimensional bistable front with $q'<0$, $q(-\infty)=1$
and $q(+\infty)=0$. For every $A>0$, uniqueness gives,
up to translation,
\begin{equation*}
	c_A=A\alpha_0+c_0,\qquad u_A(x,y)=q(x),\qquad U_A(x,y)=q(Ax).
\end{equation*}
In particular, $\gamma_A=\alpha_0+c_0/A\to\alpha_0$ as
$A\to+\infty$.

Set $H(x):=\mathbf1_{\{x<0\}}$.
The exponential convergence of $q$ at $\pm\infty$ gives
$q-H\in L^p(\R)$ for every $1\leq p<\infty$ and
$q'\in L^2(\R)$. A change of variables yields
\begin{equation*}
	\|q(A\cdot)-H\|_{L^p(\R)}^p
	=A^{-1}\|q-H\|_{L^p(\R)}^p.
\end{equation*}
Let $\tau_A$ be determined by
$\int_0^1q(A(x+\tau_A))\,dx=1/2$, and set
$V_A(x,y):=q(A(x+\tau_A))$.
The preceding estimate with $p=1$ gives
\begin{equation*}
	\left|\int_0^1H(x+\tau_A)\,dx-\frac12\right|
	\leq A^{-1}\|q-H\|_{L^1(\R)}<\frac{1}{2}
\end{equation*}
for all sufficiently large $A$. Thus, $0<\int_0^1H(x+\tau_A)\,dx<1$ for all sufficiently large $A$,
which implies that $\tau_A\in(-1,0)$ for all sufficiently large $A$. Therefore, 
$\int_0^1H(x+\tau_A)\,dx=\int_0^{-\tau_A}1\,dx=-\tau_A$, whence
the preceding estimate gives
\begin{equation*}
	\left|\tau_A+\frac12\right|
	\leq A^{-1}\|q-H\|_{L^1(\R)}.
\end{equation*}
Since $\|q-H\|_{L^1(\R)}$ is finite and independent of $A$,
we conclude that $\tau_A=-\frac12+O(A^{-1})$ as $A\to+\infty$.
Translation invariance of the $L^p(\R)$ norm and the identity
$\|H(\cdot+\tau_A)-H(\cdot-\frac12)\|_{L^p(\R)}^p
=|\tau_A+\frac12|$ give
\begin{equation}\label{eq:constant}
\|V_A-\mathbf1_{\{x<1/2\}}\|_{L^p(\R\times\T^{N-1})}
=O(A^{-1/p}), \qquad 1\leq p<\infty.
\end{equation}
Since $A(x+\tau_A)=A(x-1/2)+O(1)$, we also have
$V_A(x,y)\to\mathbf1_{\{x<1/2\}}$ for every $x\ne1/2$.
Moreover, $\nabla_yV_A=0$ and
\begin{equation*}
	A^{-2}\int_{\R\times\T^{N-1}}|V_{A,x}|^2\,dx\,dy
	=A^{-1}\int_\R|q'(s)|^2\,ds\to0.
\end{equation*}

For $\gamma=\alpha_0$, the limiting equation reduces to
$\Delta_yU+f(U)=0$. Thus, for each $L\geq0$, the
$y$-independent function
\begin{equation*}
	Q_L(x)=
	\begin{cases}
		1,&x<0,\\
		\theta,&0\leq x<L,\\
		0,&x\geq L
	\end{cases}
\end{equation*}
is a nonincreasing distributional solution with
$Q_L(-\infty)=1$ and $Q_L(+\infty)=0$.
Each $Q_L$ admits a translate satisfying the integral normalization.
Translations preserve the length $L$ of the interval on which
$Q_L=\theta$. Hence, profiles with distinct values of $L$ are not
translates, and uniqueness up to translation fails for the limiting
equation when $\alpha$ is constant. In addition, continuity of $V_A$ and the jump of the limit in
\eqref{eq:constant} give
\begin{equation*}
	\|V_A-\mathbf1_{\{x<1/2\}}\|_{L^\infty(\R\times\T^{N-1})}
	\geq\frac12
	\quad\hbox{for every }A>0.
\end{equation*}

\subsection{Global convergence of full profiles}

\begin{lemma}
\label{lem5.1}
Let $\gamma\in\R$, $A_n\to+\infty$, and
$b_n\in C(\T^{N-1})$ satisfy
$\|b_n-(\gamma-\alpha)\|_{L^\infty(\T^{N-1})}\to0$.
Let $V_n\in C^2(\R\times\T^{N-1})$ satisfy
$0<V_n<1$, $V_{n,x}<0$, and
\begin{equation*}
	A_n^{-2}V_{n,xx}+\Delta_yV_n+b_n(y)V_{n,x}+f(V_n)=0
	\quad\hbox{in }\R\times\T^{N-1},
\end{equation*}
with $V_n(-\infty,\cdot)=1$ and $V_n(+\infty,\cdot)=0$
uniformly in $\T^{N-1}$ for each $n$.
Assume that $V_n\to U$ in
$L^1_{\mathrm{loc}}(\R\times\T^{N-1})$, where $U$ is a full
degenerate front with speed $\gamma$.
Then, for every $1\leq p<+\infty$,
\begin{equation}
\|V_n-U\|_{L^p(\R\times\T^{N-1})}
+\|\nabla_yV_n-\nabla_yU\|_{L^2(\R\times\T^{N-1})}\to0,
\qquad
A_n^{-2}\|V_{n,x}\|_{L^2(\R\times\T^{N-1})}^2\to0.
\label{eq:global-convergence}
\end{equation}
Moreover, $f(U)\in L^1(\R\times\T^{N-1})$ and
$\nabla_yU\in L^2(\R\times\T^{N-1})$, with
\begin{equation}
\gamma-\bar\alpha=\int_{\R\times\T^{N-1}}f(U)\,dx\,dy,
\qquad
\int_{\R\times\T^{N-1}}|\nabla_yU|^2\,dx\,dy
=\int_{\R\times\T^{N-1}}f(U)(U-\tfrac12)\,dx\,dy.
\label{eq:global-identities}
\end{equation}
There exist $n_0\in\mathbb N$ and $R,C,\lambda>0$ such that
\begin{equation*}
	V_n(x,y)\leq Ce^{-\lambda x}\quad \text{for}~x\geq R,\qquad
	1-V_n(x,y)\leq Ce^{\lambda x}\quad \text{for}~x\leq-R,
\end{equation*}
for every $n\geq n_0$ and $y\in\T^{N-1}$.
After passing to a subsequence, $V_n\to U$ almost everywhere in
$\R\times\T^{N-1}$. If $U$ is continuous, then $V_n$ converges to $U$ uniformly in $\R\times\T^{N-1}$
as $n\to+\infty$.
\end{lemma}

\begin{proof}
	{\it Step 1. Uniform exponential bounds.}
	Since $f'(0),f'(1)<0$, choose $0<\eta<1/2$ and $\kappa>0$
	such that
	\begin{equation}
		f(s)\leq-\kappa s,\qquad
		-f(1-s)\leq-\kappa s
		\quad\hbox{for }0\leq s\leq\eta.
		\label{eq:stable}
	\end{equation}

	Set $B_0:=1+\|\gamma-\alpha\|_{L^\infty(\T^{N-1})}$.
	Since $b_n\to\gamma-\alpha$ uniformly and $A_n\to+\infty$,
	there exists $n_0\in\mathbb N$ such that
	\begin{equation*}
		\|b_n\|_{L^\infty(\T^{N-1})}\leq B_0,\qquad
		A_n^{-1}\leq1
		\quad\hbox{for }n\geq n_0.
	\end{equation*}
	Let $\rho\in C_c^\infty((-1,1))$ satisfy $\rho\geq0$ and
	$\int_\R\rho(s)\,ds=1$, and define
	\begin{equation*}
			Z_n(x,y):=\int_\R\rho(s)V_n(x-s,y)\,ds,
		\qquad
		Z(x,y):=\int_\R\rho(s)U(x-s,y)\,ds.
	\end{equation*}
	Since $0<V_n<1$, it follows that $0\leq Z_n\leq1$ and
	\begin{equation*}
			\|Z_{n,x}\|_{L^\infty(\R\times\T^{N-1})}
		\leq\|\rho'\|_{L^1(\R)},~~~	\|Z_{n,xx}\|_{L^\infty(\R\times\T^{N-1})}
		\leq\|\rho''\|_{L^1(\R)}.
	\end{equation*}
	
	Integrating the equation for $V_n(x-s,y)$ against $\rho(s)$
	gives
	\begin{equation*}
			\Delta_yZ_n
		=-A_n^{-2}Z_{n,xx}-b_n(y)Z_{n,x}
		-\int_\R\rho(s)f(V_n(x-s,y))\,ds,
	\end{equation*}
	which implies
	\begin{equation*}
			\|\Delta_yZ_n\|_{L^\infty(\R\times\T^{N-1})}
		\leq\|\rho''\|_{L^1(\R)}
		+B_0\|\rho'\|_{L^1(\R)}
		+\|f\|_{L^\infty([0,1])}
		\quad\hbox{for }n\geq n_0.
	\end{equation*}
	The elliptic estimate
	\eqref{eq:section4-elliptic-estimate} therefore yields
	\begin{equation*}
			\sup_{n\geq n_0}\sup_{x\in\R}
		\|Z_n(x,\cdot)\|_{W^{2,r}(\T^{N-1})}<+\infty
		\quad\hbox{for every }1<r<+\infty.
	\end{equation*}
	
	 Fix now $r>N-1$.
	The uniform $W^{2,r}(\T^{N-1})$ estimate  above
	and Sobolev embedding bound the transverse derivatives.
	Together with the bound on $Z_{n,x}$, this yields
	\begin{equation*}
			\|Z_{n,x}\|_{L^\infty(\R\times\T^{N-1})}
		+\|\nabla_yZ_n\|_{L^\infty(\R\times\T^{N-1})}
		\leq C
		\quad\hbox{for }n\geq n_0,
	\end{equation*}
	where $C>0$ is independent of $n$.
	After extending the functions periodically in $y$,
	integration along line segments gives
$|Z_n(x,y)-Z_n(x',y')|
	\leq C\bigl(|x-x'|+|y-y'|\bigr)$
	for all $x,x'\in\R$, $y,y'\in\R^{N-1}$, and $n\geq n_0$.
	Since $C$ is independent of $n$, this proves equicontinuity.
	Moreover, $0\leq Z_n\leq1$ gives uniform boundedness.
	Therefore, the Arzel\`a-Ascoli theorem applies on
	$[-m,m]\times\T^{N-1}$ for each $m=1,2,\ldots$.
	A diagonal extraction gives a subsequence $(Z_{n_k})$
	and a continuous function $W$ on $\R\times\T^{N-1}$ such that
	\begin{equation*}
			\|Z_{n_k}-W\|_{L^\infty((-L,L)\times\T^{N-1})}\to0
		\quad\hbox{for every }L>0.
	\end{equation*}
	
	On the other hand, 	the definitions of $Z_n$ and $Z$, together with the
	support of $\rho$, give
	\begin{equation*}
		\begin{aligned}
			\|Z_n-Z\|_{L^1((-L,L)\times\T^{N-1})}
			&\leq\int_{-1}^{1}\rho(s)
			\|V_n-U\|_{L^1((-L-s,L-s)\times\T^{N-1})}\,ds\\
			&\leq
			\|V_n-U\|_{L^1((-L-1,L+1)\times\T^{N-1})}
		\end{aligned}
	\end{equation*}
	for every $L>0$. Since $V_n\to U$ locally in $L^1$,
	the right-hand side tends to zero. Thus,
	$Z_n\to Z$ in $L^1_{\mathrm{loc}}(\R\times\T^{N-1})$.
	Therefore, 
	\begin{equation*}
		\begin{aligned}
			\|W-Z\|_{L^1((-L,L)\times\T^{N-1})}
			&\leq
			2L\|W-Z_{n_k}\|_{L^\infty((-L,L)\times\T^{N-1})}+
			\|Z_{n_k}-Z\|_{L^1((-L,L)\times\T^{N-1})}\to0.
		\end{aligned}
	\end{equation*}
Hence, $W=Z$ almost everywhere.
We henceforth use this continuous representative of $Z$.

We now prove that
\begin{equation}
	\|Z_n-Z\|_{L^\infty((-L,L)\times\T^{N-1})}\to0
	\quad\hbox{for every }L>0.
	\label{eq:smoothuniform}
\end{equation}
Suppose not. Then, there exist
$L>0$, $\varepsilon>0$, and a subsequence $(Z_{m_j})_{j\in\mathbb{N}}$ such that
\begin{equation*}
	\|Z_{m_j}-Z\|_{L^\infty((-L,L)\times\T^{N-1})}
	\geq\varepsilon
	\quad\hbox{for every }j\in\mathbb N.
\end{equation*}
By the uniform boundedness and equicontinuity proved above,
the Arzel\`a-Ascoli theorem gives a further subsequence
converging uniformly on $[-L,L]\times\T^{N-1}$.
Its limit agrees with $Z$ almost everywhere, because this
subsequence also converges to $Z$ in
$L^1((-L,L)\times\T^{N-1})$.
Since both functions are continuous, they agree everywhere
on $[-L,L]\times\T^{N-1}$.
Thus, along this further subsequence, the preceding
$L^\infty$ norm tends to zero, contradicting its lower bound
by $\varepsilon$. Consequently, \eqref{eq:smoothuniform} follows.

Since $U(+\infty,\cdot)=0$ and
	$U(-\infty,\cdot)=1$ uniformly on $\T^{N-1}$, we can choose $R>2$ so that $U(x,y)\leq\eta/2$ for $x\geq R-2$ and
	$1-U(x,y)\leq\eta/2$ for $x\leq-R+2$, uniformly on $\T^{N-1}$.
	The integral formula for $Z$ then gives
	$Z(R-1,y)=\int_{-1}^{1}\rho(s)U(R-1-s,y)\,ds\leq\eta/2$ and $1-Z(-R+1,y)=\int_{-1}^{1}\rho(s)(1-U(-R+1-s,y))\,ds\leq\eta/2$, where these inequalities first hold almost everywhere and then
	extend to every point by continuity of $Z$.
	By \eqref{eq:smoothuniform}, after increasing $n_0$,
	\begin{equation*}
			Z_n(R-1,y)\leq\eta,\qquad
		1-Z_n(-R+1,y)\leq\eta
		\quad\hbox{for }n\geq n_0,\ y\in\T^{N-1}.
	\end{equation*}
	Monotonicity and $\operatorname{supp}\rho\subset(-1,1)$ give
	$V_n(x+1,y)\leq Z_n(x,y)\leq V_n(x-1,y)$.
	Consequently, $V_n(R,y)\leq\eta$ and
	$1-V_n(-R,y)\leq\eta$, and monotonicity yields
	\begin{equation}
		V_n(x,y)\leq\eta\quad\hbox{for }x\geq R,\qquad
		1-V_n(x,y)\leq\eta\quad\hbox{for }x\leq-R,
		\label{eq:tailstrip}
	\end{equation}
	for every $n\geq n_0$ and $y\in\T^{N-1}$.
	
Choose $\lambda>0$ such that
	$\lambda^2+B_0\lambda<\kappa$, and set
	\begin{equation*}
		P(x):=\eta e^{-\lambda(x-R)}
		\quad\hbox{for }x\geq R.
	\end{equation*}
	Because $A_n^{-2}\leq1$ and $|b_n|\leq B_0$, we have
	\begin{equation*}
			\begin{aligned}
			&(A_n^{-2}\partial_{xx}+\Delta_y+b_n(y)\partial_x-\kappa)P=(A_n^{-2}\lambda^2-b_n(y)\lambda-\kappa)P
			\leq(\lambda^2+B_0\lambda-\kappa)P<0.
		\end{aligned}
	\end{equation*}
	Moreover, we also notice that 
	\begin{equation*}
		(A_n^{-2}\partial_{xx}+\Delta_y+b_n(y)\partial_x-\kappa)V_n
		=-f(V_n)-\kappa V_n\geq0
		\quad\hbox{for }x\geq R.
	\end{equation*}
	At $x=R$, \eqref{eq:tailstrip} gives $V_n(R,y)\leq P(R)=\eta$ for all $y\in\T^{N-1}$.
	For each fixed $n$, the prescribed limit of $V_n$ also gives
	$V_n-P\to0$ uniformly in $y$ as $x\to+\infty$.
	Thus, any positive value of $V_n-P$ would lead to a positive
	interior maximum. At that point its first derivatives vanish
	and its second derivatives in each coordinate direction are
	nonpositive, so the operator applied to $V_n-P$ is at most
	$-\kappa(V_n-P)<0$. This contradicts the preceding
	differential inequalities. Thus, $V_n\leq P$ on $[R,+\infty)\times\T^{N-1}$.
	
Using the second inequality in \eqref{eq:stable}, the same
	comparison applies to $1-V_n$ on
	$(-\infty,-R]\times\T^{N-1}$ with
	$\eta e^{\lambda(x+R)}$.
	Consequently,
	\begin{equation}
		\begin{aligned}
			V_n(x,y)\leq\eta e^{-\lambda(x-R)}
			\quad\hbox{for }x\geq R,~~~~
			1-V_n(x,y)\leq\eta e^{\lambda(x+R)}
			\quad\hbox{for }x\leq-R
		\end{aligned}
		\label{eq:exptails}
	\end{equation}
	for every $n\geq n_0$ and $y\in\T^{N-1}$.
	
	\medskip\noindent
	{\it Step 2. Global $L^p$ convergence.}
	Since $V_n\to U$ in
	$L^1_{\mathrm{loc}}(\R\times\T^{N-1})$, there exists a
	subsequence $(V_{n_k})$ converging to $U$ almost everywhere
	in $\R\times\T^{N-1}$.
	The bounds in \eqref{eq:exptails} are independent of $n$.
	Passing to the limit along this subsequence therefore gives
	the same bounds for $U$ almost everywhere.
	Therefore, for every $n\geq n_0$, we have
	\begin{equation*}
			\begin{aligned}
			|V_n-U|
			&\leq V_n+U
			\leq2\eta e^{-\lambda(x-R)}
			&&\quad\hbox{for a.e. }x\geq R,\\
			|V_n-U|
			&\leq(1-V_n)+(1-U)
			\leq2\eta e^{\lambda(x+R)}
			&&\quad\hbox{for a.e. }x\leq-R.
		\end{aligned}
	\end{equation*}
	Moreover, $|V_n-U|\leq1$, so
	$|V_n-U|^p\leq|V_n-U|$ for every $1\leq p<+\infty$.
	Splitting the integral into three pieces, we obtain, for every $L>R$ and $n\geq n_0$,
	\begin{align*}
		\|V_n-U\|_{L^p(\R\times\T^{N-1})}^p
		&\leq
		\|V_n-U\|_{L^1((-L,L)\times\T^{N-1})}+2\eta\int_L^{+\infty}e^{-\lambda(x-R)}\,dx
		+2\eta\int_{-\infty}^{-L}e^{\lambda(x+R)}\,dx\\
		&=
		\|V_n-U\|_{L^1((-L,L)\times\T^{N-1})}
		+\frac{4\eta}{\lambda}e^{-\lambda(L-R)}.
	\end{align*}
	Since, for each fixed $L$, the first term on the right-hand side above tends to zero by assumption, it follows from passing to the upper limit as
	$n\to+\infty$ and then letting $L\to+\infty$ that
	\begin{equation*}
			\|V_n-U\|_{L^p(\R\times\T^{N-1})}\to0
		\quad\hbox{for every }1\leq p<+\infty.
	\end{equation*}

\medskip\noindent
{\it Step 3. Integral identities and energy convergence.}
Since $f$ is Lipschitz continuous on $[0,1]$ and
$f(0)=f(1)=0$, there exists $C>0$ such that
\begin{equation*}
	|f(s)|\leq C\min\{s,1-s\}
	\quad\hbox{for }0\leq s\leq1.
\end{equation*}
The exponential bounds \eqref{eq:exptails}, which also hold
for $U$ almost everywhere, therefore control the reaction
terms on both tails. On $(-R,R)\times\T^{N-1}$, these terms
are bounded by $\|f\|_{L^\infty([0,1])}$. Consequently,
\begin{equation*}
	\|f(U)\|_{L^1(\R\times\T^{N-1})}
	+\sup_{n\geq n_0}
	\|f(V_n)\|_{L^1(\R\times\T^{N-1})}<+\infty.
\end{equation*}
Moreover, both $f$ and $s\mapsto f(s)(s-\tfrac12)$ are
Lipschitz continuous on $[0,1]$. The global $L^1$ convergence
proved in Step~2 thus gives
\begin{equation}
	\label{C}
\begin{aligned}
	\|f(V_n)-f(U)\|_{L^1(\R\times\T^{N-1})}+&
	\|f(V_n)(V_n-\tfrac12)-f(U)(U-\tfrac12)\|
	_{L^1(\R\times\T^{N-1})}\\
	&~~~~~~~~~~~~~~~~~~~~~~~~~~\leq C\|V_n-U\|_{L^1(\R\times\T^{N-1})}\to0.
\end{aligned}
\end{equation}

Choose $\chi\in C_c^\infty(\R)$ with $0\leq\chi\leq1$,
$\chi=1$ on $[-1,1]$ and $\chi=0$ outside $[-2,2]$,
and set $\chi_L(x):=\chi(x/L)$ for $L>0$.
Multiplying the equation for $V_n$ by
$(V_n-\tfrac12)\chi_L$ and integrating by parts yields
\begin{eqnarray}	\label{eq:cutoffglobal}
	\begin{aligned}
		&\int_{\R\times\T^{N-1}}\chi_L
		\bigl(|\nabla_yV_n|^2+A_n^{-2}|V_{n,x}|^2\bigr)\,dx\,dy\\
		=&
		\frac{A_n^{-2}}{2}
		\int_{\R\times\T^{N-1}}\chi_L''(V_n-\tfrac12)^2\,dx\,dy
		-\frac12\int_{\R\times\T^{N-1}}
		b_n(y)\chi_L'(V_n-\tfrac12)^2\,dx\,dy\\
		&\qquad+
		\int_{\R\times\T^{N-1}}
		\chi_L f(V_n)(V_n-\tfrac12)\,dx\,dy.
	\end{aligned}
\end{eqnarray}
Since $b_n$ is independent of $x$ and
$\int_\R\chi_L'\,dx=\int_\R\chi_L''\,dx=0$, we may replace
$(V_n-\tfrac12)^2$ in the first two terms on the right-hand side by
$(V_n-\tfrac12)^2-\tfrac14=-V_n(1-V_n)$.
Moreover, the exponential bounds imply
\begin{equation*}
	\sup_{n\geq n_0}
	\|V_n(1-V_n)\|_{L^1(\R\times\T^{N-1})}
	+\|U(1-U)\|_{L^1(\R\times\T^{N-1})}<+\infty.
\end{equation*}
Together with
\begin{equation*}
	\|\chi_L'\|_{L^\infty(\R)}
	=L^{-1}\|\chi'\|_{L^\infty(\R)},
	\qquad
	\|\chi_L''\|_{L^\infty(\R)}
	=L^{-2}\|\chi''\|_{L^\infty(\R)},
\end{equation*}
and the bounds $A_n^{-1}\leq1$ and
$\|b_n\|_{L^\infty(\T^{N-1})}\leq B_0$, this shows that
the sum of the absolute values of the first two terms in
\eqref{eq:cutoffglobal} is at most $C(L^{-2}+L^{-1})$,
uniformly for $n\geq n_0$.

For each fixed $n\geq n_0$, the reaction term converges
to its integral over $\R\times\T^{N-1}$ by dominated
convergence. Since $\chi_L\geq0$ and $\chi_L\to1$ pointwise,
Fatou's lemma first gives finite global energy.
We may then apply dominated convergence to the left-hand
side of \eqref{eq:cutoffglobal}, obtaining
\begin{equation}
	\int_{\R\times\T^{N-1}}
	\bigl(|\nabla_yV_n|^2+A_n^{-2}|V_{n,x}|^2\bigr)\,dx\,dy
	=
	\int_{\R\times\T^{N-1}}
	f(V_n)(V_n-\tfrac12)\,dx\,dy.
	\label{eq:exactglobal}
\end{equation}
The right-hand side is bounded uniformly for $n\geq n_0$.
In particular, $\nabla_yV_n$ is bounded in
$L^2(\R\times\T^{N-1})$.
Integration by parts against smooth compactly supported
test functions and the local $L^1$ convergence identify
every subsequential weak limit as $\nabla_yU$.
Thus, $\nabla_yU\in L^2(\R\times\T^{N-1})$, and
\begin{equation}
	\label{weakconver}
	\nabla_yV_n\rightharpoonup\nabla_yU
	\quad\hbox{in }L^2(\R\times\T^{N-1})
\end{equation}
along the whole given sequence.

Since $0\leq U\leq1$ almost everywhere,
$\nabla_yU\in L^2(\R\times\T^{N-1})$, and
$L_\gamma U+f(U)=0$ in $\D'(\R\times\T^{N-1})$,
the longitudinal regularization argument in the proof of
Lemma~\ref{lem:local-strong-energy} gives
\eqref{eq:revision-limit-energy} with $\chi=\chi_L$.
Testing the equation for $U$ with $\chi_L$ also gives
\begin{equation}
	\label{U+chi_L}
	\int_{\R\times\T^{N-1}}f(U)\chi_L \,dx\,dy
	=
	\int_{\R\times\T^{N-1}}
	(\gamma-\alpha(y))U(x,y)\chi_L'(x)\,dx\,dy.
\end{equation}
Writing $f(U)U=f(U)(U-\tfrac12)+\tfrac12f(U)$ in
\eqref{eq:revision-limit-energy} and substituting the
preceding equality, we obtain
\begin{align*}
\int_{\R\times\T^{N-1}}|\nabla_yU|^2\chi_L\,dx\,dy
&=-\frac12\int_{\R\times\T^{N-1}}
(\gamma-\alpha(y))(U^2-U)\chi_L'\,dx\,dy\\
&\quad+\int_{\R\times\T^{N-1}}
f(U)(U-\tfrac12)\chi_L\,dx\,dy.
\end{align*}
Since
$\|\chi_L'\|_{L^\infty(\R)}
=L^{-1}\|\chi'\|_{L^\infty(\R)}$
and $U(1-U)\in L^1(\R\times\T^{N-1})$,
which follows from the exponential bounds for $U$,  the first term on the
right-hand side above  has the following estimate
\begin{equation*}
	\frac12\|\gamma-\alpha\|_{L^\infty(\T^{N-1})}
	\|\chi_L'\|_{L^\infty(\R)}
	\|U(1-U)\|_{L^1(\R\times\T^{N-1})}\to0
	\quad\hbox{as }L\to+\infty.
\end{equation*}
Since $|\nabla_yU|^2$ and $f(U)(U-\tfrac12)$ are integrable,
dominated convergence now yields
\begin{equation*}
	\int_{\R\times\T^{N-1}}|\nabla_yU|^2\,dx\,dy
	=
	\int_{\R\times\T^{N-1}}f(U)(U-\tfrac12)\,dx\,dy.
\end{equation*}

To obtain the speed identity, we return to the equality \eqref{U+chi_L}.
Let $H(x):=\mathbf1_{\{x<0\}}$.
The exponential bounds imply
$U-H\in L^1(\R\times\T^{N-1})$, while
\begin{equation*}
	\int_\R H(x)\chi_L'(x)\,dx
	=\int_{-\infty}^{0}\chi_L'(x)\,dx
	=\chi_L(0)=1.
\end{equation*}
Consequently,
\begin{align*}
	\left|
	\int_{\R\times\T^{N-1}}f(U)\chi_L \,dx\,dy
	-(\gamma-\bar\alpha)
	\right|
	&=
	\left|
	\int_{\R\times\T^{N-1}}
	(\gamma-\alpha(y))(U(x,y)-H(x))\chi_L'(x)\,dx\,dy
	\right|\\
	&\leq
	\|\gamma-\alpha\|_{L^\infty(\T^{N-1})}
	\|\chi_L'\|_{L^\infty(\R)}
	\|U-H\|_{L^1(\R\times\T^{N-1})}\to0.
\end{align*}
Since $f(U)\in L^1(\R\times\T^{N-1})$, 
 it gives $\int_{\R\times\T^{N-1}}f(U)\,dx\,dy
=\gamma-\bar\alpha$ as $L\to\infty$.
This proves \eqref{eq:global-identities}.

Finally, \eqref{eq:exactglobal}, the second identity in  \eqref{eq:global-identities},
and \eqref{C} give
\begin{equation*}
	\|\nabla_yV_n\|_{L^2(\R\times\T^{N-1})}^2
	+A_n^{-2}\|V_{n,x}\|_{L^2(\R\times\T^{N-1})}^2
	\to\|\nabla_yU\|_{L^2(\R\times\T^{N-1})}^2.
\end{equation*}
The weak convergence \eqref{weakconver} of $\nabla_yV_n$ also gives
\begin{equation*}
	\int_{\R\times\T^{N-1}}
	\nabla_yV_n\cdot\nabla_yU\,dx\,dy
	\to\|\nabla_yU\|_{L^2(\R\times\T^{N-1})}^2.
\end{equation*}
Combining these two limits yields
\begin{equation*}
	\|\nabla_yV_n-\nabla_yU\|_{L^2(\R\times\T^{N-1})}^2
	+A_n^{-2}\|V_{n,x}\|_{L^2(\R\times\T^{N-1})}^2\to0.
\end{equation*}
Both terms are nonnegative, so each tends to zero.
Together with Step~2, this proves \eqref{eq:global-convergence}.

	\medskip
	\noindent
	{\it Step 4. Uniform convergence for continuous limits.}
	Assume that $U$ is continuous.
	For $0<h<1/2$, set $\rho_h(s):=h^{-1}\rho(s/h)$ and define
\begin{equation*}
	Z_{n,h}(x,y):=\int_\R\rho_h(s)V_n(x-s,y)\,ds,
	\qquad
	Z_h(x,y):=\int_\R\rho_h(s)U(x-s,y)\,ds.
\end{equation*}
For each fixed $h$, the argument leading to
	\eqref{eq:smoothuniform} gives
	\begin{equation*}
			\|Z_{n,h}-Z_h\|_{L^\infty((-L-1,L+1)\times\T^{N-1})}
		\to0
		\quad\hbox{for every }L>0.
	\end{equation*}
	Monotonicity yields
	$Z_{n,h}(x+h,y)\leq V_n(x,y)\leq Z_{n,h}(x-h,y)$.
	Consequently,
	\begin{equation*}
			\begin{aligned}
			\|V_n-U\|_{L^\infty((-L,L)\times\T^{N-1})}
			\!\leq\!
			\|Z_{n,h}-Z_h\|_{L^\infty((-L-1,L+1)\times\T^{N-1})}+\!\!
			\sup_{\substack{|x|\leq L,\ |s|\leq2h\\y\in\T^{N-1}}}
			|U(x+s,y)-U(x,y)|.
		\end{aligned}
	\end{equation*}
	For fixed $L$, the last term tends to zero as $h\to 0$
	by uniform continuity of $U$ on
	$[-L-1,L+1]\times\T^{N-1}$.
	Letting first $n\to+\infty$ and then $h\to 0$
	proves local uniform convergence.
	Combining this with \eqref{eq:exptails}, we obtain, for every
	$L>R$,
\begin{equation*}
	\limsup_{n\to+\infty}
	\|V_n-U\|_{L^\infty(\R\times\T^{N-1})}
	\leq2\eta e^{-\lambda(L-R)}.
\end{equation*}
	Letting $L\to+\infty$ completes the proof.
\end{proof}

\subsection{Proof of Theorem \ref{thm:main}}

We divide into several steps. First, we prove convergence of the speeds and identify full limits
of the normalized profiles. Lemma~\ref{lem5.1}
then gives their global convergence and integral identities.
We next establish the strict speed bounds and conclude with the
uniqueness and convergence results under the H\"ormander condition \eqref{cdn_alpha}.

{\it Step 1. Convergence of the speeds.}
Suppose first that $\alpha$ satisfies \eqref{cdn_alpha}.
By \eqref{speed-bounds}, the family $\gamma_A=c_A/A$ is bounded
for $A\geq1$. Given a sequence $A_n\to+\infty$ along which
$\gamma_{A_n}\to\gamma$, Propositions~\ref{prop_alternative result}
and \ref{prop:exclude-two-components} give, after extraction and
longitudinal translation, a full degenerate front with speed
$\gamma$. Proposition~\ref{prop:full-uniqueness} shows that every
such limit has the same speed. Hence, $\gamma_A$ converges to a
number denoted by $\gamma^*(\alpha,f)$.

Let now $\alpha\in C^{1,\delta}(\T^{N-1})$ be arbitrary and fix
$\varepsilon>0$. By the Stone-Weierstrass theorem, there exists a real-valued
trigonometric polynomial $P$ such that
$\|\alpha-P\|_\infty<\varepsilon/2$.  If $P$ is constant, replace it
by $P+\eta\cos(2\pi y_1)$, where
$0<|\eta|<\varepsilon/2$.  Thus we obtain a nonconstant real analytic
trigonometric polynomial $\beta$ satisfying
$\|\alpha-\beta\|_{L^\infty(\T^{N-1})}<\varepsilon$.
Since $\beta$ is nonconstant on $\T^{N-1}$, for every $y\in\T^{N-1}$, some derivative of $\beta$ of positive
order is nonzero.   By compactness, there exists an integer $r\geq1$ such
that
$\sum_{1\leq|\zeta|\leq r}|D^\zeta\beta|>0$ everywhere.
Hence, $\beta$ satisfies \eqref{cdn_alpha}, and the preceding
argument gives convergence of $c_A(\beta,f)/A$.
By \eqref{Lip_speed},
\begin{equation*}
	\left|\frac{c_A(\alpha,f)}{A}-\frac{c_A(\beta,f)}{A}\right|
	<\varepsilon.
\end{equation*}
Consequently,
\begin{equation*}
	0\leq\limsup_{A\to+\infty}\frac{c_A(\alpha,f)}{A}
	-\liminf_{A\to+\infty}\frac{c_A(\alpha,f)}{A}\leq2\varepsilon.
\end{equation*}
Letting $\varepsilon\to 0$ proves \eqref{eq:speed-limit}, and
\eqref{speed-bounds} gives
$\min_{\T^{N-1}}\alpha\leq\gamma^*(\alpha,f)
\leq\max_{\T^{N-1}}\alpha$.

\medskip\noindent
{\it Step 2. Normalization and global profile convergence.}
For every $A>0$, the function
$s\mapsto\mathcal N(U_A(\cdot+s,\cdot))$ is continuous and
strictly decreasing, with limits $1$ and $0$ as $s\to-\infty$
and $s\to+\infty$, respectively. Hence, there is a unique $\tau_A\in\R$ such that
$\mathcal N(U_A(\cdot+\tau_A,\cdot))=1/2$.
Set $V_A(x,y):=U_A(x+\tau_A,y)$.

If $\alpha$ is constant, Section~\ref{subsec:constant} gives
\eqref{eq:general-profile-convergence} and almost-everywhere
convergence to $U(x,y)=\mathbf1_{\{x<1/2\}}$.
For this limit, $f(U)=0$ and $\nabla_yU=0$ almost everywhere,
while $\gamma^*(\alpha,f)=\bar\alpha$; thus,
\eqref{eq:general-identities} also holds.

Suppose that $\alpha$ is nonconstant and fix an arbitrary
sequence $A_n\to+\infty$. Step~1 gives
$\gamma_{A_n}\to\gamma^*(\alpha,f)$.
Propositions~\ref{prop_alternative result} and
\ref{prop:exclude-two-components} yield a subsequence, shifts
$t_n$, and a full degenerate front $U_0$ such that
\begin{equation*}
	U_{A_n}(\cdot+t_n,\cdot)\to U_0
	\quad\hbox{in }L^1_{\mathrm{loc}}(\R\times\T^{N-1}).
\end{equation*}
Moreover, $(U_0)_x\leq0$ in $\D'(\R\times\T^{N-1})$ and
$\nabla_yU_0\in L^2_{\mathrm{loc}}(\R\times\T^{N-1})$.
The uniform limits of $U_0$ allow us to choose $M>0$ with
\begin{equation*}
	\mathcal N(U_0(\cdot-M,\cdot))>\tfrac12,
	\qquad
	\mathcal N(U_0(\cdot+M,\cdot))<\tfrac12.
\end{equation*}
Local $L^1$ convergence transfers both strict inequalities to
$U_{A_n}(\cdot+t_n,\cdot)$ for all sufficiently large $n$.
By the definition of $\tau_{A_n}$ and strict monotonicity,
$t_n-M<\tau_{A_n}<t_n+M$.
After extraction, $\tau_{A_n}-t_n\to\ell\in\R$.
The translation estimate \eqref{eq:revision-xtranslation} gives
\begin{equation*}
	\|V_{A_n}-U_{A_n}(\cdot+t_n+\ell,\cdot)\|
	_{L^1(\R\times\T^{N-1})}
	\leq|\tau_{A_n}-t_n-\ell|\to0.
\end{equation*}
Therefore, $V_{A_n}\to U:=U_0(\cdot+\ell,\cdot)$ in
$L^1_{\mathrm{loc}}(\R\times\T^{N-1})$.
The limit is a full degenerate front with speed
$\gamma^*(\alpha,f)$, satisfies $U_x\leq0$ in $\D'(\R\times\T^{N-1})$,
and $\mathcal N(U)=1/2$.

Apply Lemma~\ref{lem5.1} with
$V_n=V_{A_n}$ and $b_n=\gamma_{A_n}-\alpha$.
It gives \eqref{eq:general-profile-convergence},
$f(U)\in L^1(\R\times\T^{N-1})$,
$\nabla_yU\in L^2(\R\times\T^{N-1})$, and
\eqref{eq:general-identities}.
A further extraction gives almost-everywhere convergence
along the same subsequence, preserving all these conclusions.

Section~\ref{subsec:constant} and \eqref{eq:global-convergence}
show that every sequence $A_n\to+\infty$ has a subsequence
along which the longitudinal energy in
\eqref{eq:main-longitudinal-energy} tends to zero.
If \eqref{eq:main-longitudinal-energy} failed for the whole family, a sequence of these
nonnegative energies bounded below by a positive constant
would contradict this property. This proves
\eqref{eq:main-longitudinal-energy}.

\medskip\noindent
{\it Step 3. Strict speed bounds.}
We first establish a weighted interpolation inequality
for the uniqueness argument below.
Let $a\in C(\T^{N-1})$ satisfy $a\geq0$ and $a>0$
almost everywhere. We claim that, for every $\eta>0$,
there exists $C_\eta>0$ such that
\begin{equation}
	\|w\|_{L^2(\T^{N-1})}^2
	\leq\eta\|\nabla_yw\|_{L^2(\T^{N-1})}^2
	+C_\eta\int_{\T^{N-1}}a(y)w(y)^2\,dy
	\quad\hbox{for }w\in H^1(\T^{N-1}).
	\label{eq:weighted-interpolation}
\end{equation}
Suppose that \eqref{eq:weighted-interpolation} fails.
Then, there exist a fixed $\eta>0$ and
a sequence $(w_j)_{j\in\mathbb{N}}$ in $H^1(\T^{N-1})$ such that
$\|w_j\|_{L^2(\T^{N-1})}=1$, and
\begin{equation*}
	\eta\|\nabla_yw_j\|_{L^2(\T^{N-1})}^2
	+j\int_{\T^{N-1}}a(y)w_j(y)^2\,dy<\|w_j\|_{L^2(\T^{N-1})}^2=1 \quad\hbox{for every }j\geq1.
\end{equation*}
Both terms on the left-hand side are nonnegative, whence
\begin{equation*}
	\|\nabla_yw_j\|_{L^2(\T^{N-1})}^2<\eta^{-1},
	\qquad
	\int_{\T^{N-1}}a(y)w_j(y)^2\,dy<j^{-1}.
\end{equation*}
The first bound together with $\|w_j\|_{L^2(\T^{N-1})}=1$ shows that $(w_j)_{j\in\mathbb{N}}$
is bounded in $H^1(\T^{N-1})$.
By the compact embedding into $L^2(\T^{N-1})$,
a subsequence converges strongly to some $w$ in
$L^2(\T^{N-1})$, and $\|w\|_{L^2(\T^{N-1})}=1$.
On the other hand,
\begin{equation*}
	0\leq\int_{\T^{N-1}}a(y)w(y)^2\,dy
	\leq
	2\|a\|_{L^\infty(\T^{N-1})}
	\|w-w_j\|_{L^2(\T^{N-1})}^2+\frac2j\to0.
\end{equation*}
Thus, $\int_{\T^{N-1}}a(y)w(y)^2\,dy=0$.
Since $a>0$ almost everywhere, this implies $w=0$
almost everywhere, contradicting
$\|w\|_{L^2(\T^{N-1})}=1$.
This proves \eqref{eq:weighted-interpolation}.

Consider a full front $U$ with speed $\gamma$. If 
$\gamma=\max_{\T^{N-1}}\alpha$, the function
$v(x,y):=U(-x,y)$ satisfies
\begin{equation*}
	(\gamma-\alpha(y))v_x-\Delta_yv=f(v).
\end{equation*}
If $\gamma=\min_{\T^{N-1}}\alpha$, the function
$v(x,y):=1-U(x,y)$ instead satisfies
\begin{equation*}
	(\alpha(y)-\gamma)v_x-\Delta_yv=-f(1-v).
\end{equation*}
In both cases, the coefficient of $v_x$ is nonnegative.
The assumption $|\{\alpha=\max_{\T^{N-1}}\alpha\}|=0$
gives $\gamma-\alpha>0$ almost everywhere in the upper-bound
case, while $|\{\alpha=\min_{\T^{N-1}}\alpha\}|=0$
gives $\alpha-\gamma>0$ almost everywhere in the lower-bound
case. For either equation, we now show that if $v=0$
almost everywhere in $(-\infty,x_0)\times\T^{N-1}$ for some
$x_0\in\R$, then $v=0$ almost everywhere in
$\R\times\T^{N-1}$. To treat both cases together, we consider the following
equation with coefficient $a$ and reaction term $g$.

Let $a$ satisfy the assumptions of
\eqref{eq:weighted-interpolation}, and let
$v\in L^\infty(\R\times\T^{N-1})$ satisfy
\begin{equation}
	a(y)v_x-\Delta_yv=g(v)
	\quad\hbox{in }\D'(\R\times\T^{N-1}),
	\label{eq:weighted-forward}
\end{equation}
with $0\leq v\leq1$ almost everywhere and
$\nabla_yv\in L^2_{\mathrm{loc}}(\R\times\T^{N-1})$.
Assume that $g$ is Lipschitz continuous on $[0,1]$ and
$g(0)=0$. We claim that, if there exists $x_0\in\R$ such that
$v=0$ a.e. in $(-\infty,x_0)\times\T^{N-1}$,
then $v=0$ almost everywhere in $\R\times\T^{N-1}$.

Choose
$\rho\in C_c^\infty((-1,1))$ with $\rho\geq0$ and
$\int_\R\rho(s)\,ds=1$. For $0<h<1$, set
\begin{equation*}
	\rho_h(s):=h^{-1}\rho(s/h),
	\qquad
	v_h(x,y):=\int_\R\rho_h(s)v(x-s,y)\,ds.
\end{equation*}
Since $a$ is independent of $x$, convolving
\eqref{eq:weighted-forward} gives
\begin{equation*}
	a(y)(v_h)_x-\Delta_yv_h
	=\int_\R\rho_h(s)g(v(x-s,y))\,ds
	\quad\hbox{in }\D'(\R\times\T^{N-1}).
\end{equation*}
The function $v_h$ is smooth in $x$ and has
$\nabla_yv_h\in L^2_{\mathrm{loc}}(\R\times\T^{N-1})$.
Thus, smooth approximation in $y$ allows us to test this
equation by $\chi v_h$ for $\chi\in C_c^\infty(\R)$.
Integrating by parts yields
\begin{equation}
	\label{D}
\begin{aligned}
	&-\frac12\int_{\R\times\T^{N-1}}
	a(y)\chi'(x)v_h^2\,dx\,dy
	+\int_{\R\times\T^{N-1}}
	\chi(x)|\nabla_yv_h|^2\,dx\,dy\\
	&~~~~~~~~~~~~~~~~~\qquad=
	\int_{\R\times\T^{N-1}}\chi(x)v_h(x,y)
	\left(\int_\R\rho_h(s)g(v(x-s,y))\,ds\right)\,dx\,dy.
\end{aligned}
\end{equation}
On the other hand,
for every bounded interval $I\subset\R$, convolution gives
\begin{equation*}
	\|v_h-v\|_{L^2(I\times\T^{N-1})}
	+\|\nabla_yv_h-\nabla_yv\|_{L^2(I\times\T^{N-1})}\to0
	\quad\hbox{as }h\to  0.
\end{equation*}
The averaged reaction term also converges to $g(v)$ in
$L^2(I\times\T^{N-1})$.
Since $a$ is bounded and $\chi$ has compact support,
we may pass to the limit in all three terms in \eqref{D}, and get
\begin{equation}
	\label{M}
	\frac12M'(x)
	+\|\nabla_yv(x,\cdot)\|_{L^2(\T^{N-1})}^2
	=\int_{\T^{N-1}}g(v(x,y))v(x,y)\,dy
	\quad\hbox{in }\D'(\R),
\end{equation}
where we have defined, 
for almost every $x\in\R$,
\begin{equation*}
	M(x):=\int_{\T^{N-1}}a(y)v(x,y)^2\,dy.
\end{equation*}
The last two terms in \eqref{M} are locally integrable in $x$.
Therefore, $M$ has a nonnegative locally absolutely
continuous representative, which we use below.

Choose $L>0$ such that $|g(s)|\leq Ls$ for $0\leq s\leq1$,
and take $\eta=(2L)^{-1}$ in
\eqref{eq:weighted-interpolation}.
For almost every $x$, we have
$v(x,\cdot)\in H^1(\T^{N-1})$, so that inequality gives
\begin{align*}
	\int_{\T^{N-1}}g(v(x,y))v(x,y)\,dy
	&\leq L\|v(x,\cdot)\|_{L^2(\T^{N-1})}^2\leq
	\frac12\|\nabla_yv(x,\cdot)\|_{L^2(\T^{N-1})}^2
	+LC_\eta M(x).
\end{align*}
Substituting this estimate into the energy identity yields
\begin{equation*}
	\frac12M'(x)
	+\frac12\|\nabla_yv(x,\cdot)\|_{L^2(\T^{N-1})}^2
	\leq LC_\eta M(x)
	\quad\hbox{for a.e. }x\in\R.
\end{equation*}
In particular, $M'(x)\leq2LC_\eta M(x)$ almost everywhere.

The assumed vanishing of $v$ gives $M=0$ almost everywhere
on $(-\infty,x_0)$. Since the chosen representative of $M$
is continuous, it follows that $M=0$ on $(-\infty,x_0]$.
Applying Gronwall's inequality on each finite interval
starting at $x_0$, we obtain
\begin{equation*}
	0\leq M(x)\leq M(x_0)e^{2LC_\eta(x-x_0)}=0
	\quad\hbox{for }x\geq x_0.
\end{equation*}
Thus, $M$ vanishes everywhere.
Since $a>0$ almost everywhere, the definition of $M$
implies $v=0$ almost everywhere in $\R\times\T^{N-1}$,
proving the claim.

Let $U$ be a full front obtained in Step~2 and write
$\gamma=\gamma^*(\alpha,f)$.
Testing its equation against functions independent of $y$ gives
\begin{equation*}
	\frac{d}{dx}\int_{\T^{N-1}}(\gamma-\alpha(y))U(x,y)\,dy
	=-\int_{\T^{N-1}}f(U(x,y))\,dy
	\quad\hbox{in }\D'(\R).
\end{equation*}
The weighted integral on the left has a locally absolutely
continuous representative. Since $f(U)\in L^1(\R\times\T^{N-1})$
and $U(+\infty,\cdot)=0$ uniformly in $\T^{N-1}$, integration
gives, for almost every $R\in\R$,
\begin{equation}
\int_R^{+\infty}\int_{\T^{N-1}}f(U(x,y))\,dy\,dx
=\int_{\T^{N-1}}(\gamma-\alpha(y))U(R,y)\,dy.
\label{eq:right-tail-flux}
\end{equation}

Assume that $|\{\alpha=\max_{\T^{N-1}}\alpha\}|=0$
and suppose, for contradiction, that
$\gamma=\max_{\T^{N-1}}\alpha$.
By \eqref{eq:stable} and the uniform limit at $+\infty$,
we may choose $R>0$ sufficiently large that
$f(U)\leq-\kappa U$ almost everywhere on
$(R,+\infty)\times\T^{N-1}$ and
\eqref{eq:right-tail-flux} holds at $R$. Then,
\begin{align*}
0&\leq\int_{\T^{N-1}}(\gamma-\alpha(y))U(R,y)\,dy
=\int_R^{+\infty}\int_{\T^{N-1}}f(U(x,y))\,dy\,dx\\
&\leq-\kappa\int_R^{+\infty}\int_{\T^{N-1}}U(x,y)\,dy\,dx
\leq0.
\end{align*}
Thus, $U=0$ almost everywhere in
$(R,+\infty)\times\T^{N-1}$.
The function $v(x,y):=U(-x,y)$ satisfies
\eqref{eq:weighted-forward} with
$a(y)=\gamma-\alpha(y)>0$ almost everywhere and $g=f$,
and vanishes almost everywhere in
$(-\infty,-R)\times\T^{N-1}$.
The preceding uniqueness argument gives $v=0$ almost
everywhere in $\R\times\T^{N-1}$, contradicting
$U(-\infty,\cdot)=1$. This proves the strict upper bound.

Assume now that $|\{\alpha=\min_{\T^{N-1}}\alpha\}|=0$
and suppose that $\gamma=\min_{\T^{N-1}}\alpha$.
Integrating the same distributional identity from $-\infty$
to $-R$ and using $U(-\infty,\cdot)=1$ gives
\begin{equation*}
	\int_{-\infty}^{-R}\int_{\T^{N-1}}f(U(x,y))\,dy\,dx
	=\int_{\T^{N-1}}(\gamma-\alpha(y))(1-U(-R,y))\,dy
\end{equation*}
for almost every $R>0$.
Choose such an $R$ sufficiently large that
$f(U)\geq\kappa(1-U)$ almost everywhere in
$(-\infty,-R)\times\T^{N-1}$, as permitted by
\eqref{eq:stable} and the uniform limit at $-\infty$.
Then,
\begin{align*}
0&\geq\int_{\T^{N-1}}(\gamma-\alpha(y))(1-U(-R,y))\,dy\\
&=\int_{-\infty}^{-R}\int_{\T^{N-1}}f(U(x,y))\,dy\,dx
\geq\kappa\int_{-\infty}^{-R}\int_{\T^{N-1}}(1-U(x,y))\,dy\,dx
\geq0.
\end{align*}
Therefore, $1-U=0$ almost everywhere in
$(-\infty,-R)\times\T^{N-1}$.
The function $v(x,y):=1-U(x,y)$ satisfies
\eqref{eq:weighted-forward} with
$a(y)=\alpha(y)-\gamma>0$ almost everywhere and
$g(s)=-f(1-s)$, and vanishes almost everywhere on that
half-cylinder. The same uniqueness argument gives $v=0$
almost everywhere in $\R\times\T^{N-1}$, contradicting
$U(+\infty,\cdot)=0$.
This proves the strict lower bound and completes the proof
of part~\textup{(i)}.

\medskip\noindent
{\it Step 4. Uniqueness and convergence under \eqref{cdn_alpha}.}
Assume \eqref{cdn_alpha}.
Proposition~\ref{prop:full-uniqueness} identifies all full
degenerate fronts up to longitudinal translation, with the
same speed $\gamma=\gamma^*(\alpha,f)$.
Lemmas~\ref{lem:hypo-regularity} and \ref{lem_U and gamma} give
the stated regularity, $0<U<1$, $U_x<0$, and
$\min_{\T^{N-1}}\alpha<\gamma<\max_{\T^{N-1}}\alpha$.
In particular, \eqref{eq:uniform-hypo-interior} yields
$\nabla_yU\in L^\infty(\R\times\T^{N-1})$; its global
$L^2$ integrability was proved in Step~2.
Since $U_x<0$, exactly one translate satisfies
$\mathcal N(U)=1/2$.

For any sequence $A_n\to+\infty$, Step~2 gives a subsequence
converging locally in $L^1$ to a normalized full front.  By uniqueness up to translation, this limit is
$U(\cdot+\ell,\cdot)$ for some $\ell\in\R$.
The equality
$\mathcal N(U(\cdot+\ell,\cdot))=\mathcal N(U)=1/2$
and strict decreasing property of $s\mapsto\mathcal N(U(\cdot+s,\cdot))$
force $\ell=0$.
Since $U$ is continuous, Lemma~\ref{lem5.1} gives
\begin{equation*}
	\|V_{A_{n_k}}-U\|_{L^p(\R\times\T^{N-1})}
	+\|\nabla_yV_{A_{n_k}}-\nabla_yU\|_{L^2(\R\times\T^{N-1})}
	\to0
\end{equation*}
as $k\to+\infty$, for every $1\leq p\leq+\infty$.

If this convergence failed for the whole family for some
fixed $p$, there would exist $\varepsilon>0$ and a sequence
$A_n\to+\infty$ along which the corresponding sum of norms
is at least $\varepsilon$ for every $n$.
The preceding argument would give a subsequence along which
that sum tends to zero, a contradiction.
Therefore, \eqref{eq:regular-profile-convergence} holds as
$A\to+\infty$. Taking $p=+\infty$ gives uniform convergence
of the whole normalized family on $\R\times\T^{N-1}$.
This proves part~\textup{(ii)} and completes the proof of
Theorem~\ref{thm:main}.

\section{Proof of Proposition \ref{prop:opposite-signs}}
\label{sec6}

In this section, we shall prove Proposition~\ref{prop:opposite-signs}, for which we first
establish a comparison estimate for monotone test profiles
by adapting the sliding argument of
Lemma~\ref{lem:order-speed}.
We apply this estimate at finite $A$ and then pass to the
limit $A\to+\infty$.

\begin{lemma}\label{lem:example-speed-comparison}
	Assume \eqref{eq:bistable}, and fix $A\ge1$ and
	$\alpha\in C^{1,\delta}(\T)$.
	Let $V\in C^2(\R\times\T)$ satisfy $0<V<1$ and $V_x<0$,
	with $V(-\infty,\cdot)=1$ and $V(+\infty,\cdot)=0$
	uniformly in $\T$.
	If, for some $\lambda\in\R$ and $m\ge0$,
	\begin{equation}
		\left|A^{-2}V_{xx}+V_{yy}
		+(\lambda-\alpha(y))V_x+f(V)\right|
		\le m(-V_x)
		\quad\text{in }\R\times\T,
		\label{eq:example-relative-residual}
	\end{equation}
	then $\left|c_A(\alpha,f)/A-\lambda\right|\le m$.
\end{lemma}
\begin{proof}
	Let $U_A$ be the front in \eqref{eqn_TW_after rescaling},
	with $\gamma_A=c_A(\alpha,f)/A$.
	We first prove $\gamma_A\le\lambda+m$.
	Suppose otherwise, and set $V_h(x,y)=V(x-h,y)$.
	At any local maximum of $U_A-V_h$, the front equation
	and \eqref{eq:example-relative-residual} give
	\begin{equation}
		\begin{aligned}
			0
			&\ge A^{-2}(U_A-V_h)_{xx}+(U_A-V_h)_{yy}\ge f(V_h)-f(U_A)
			-(\gamma_A-\lambda-m)(V_h)_x.
		\end{aligned}
		\label{eq:example-sliding-inequality}
	\end{equation}
	
	Choose $\rho_0\in(0,1/2)$ such that $f$ is strictly
	decreasing on $[0,\rho_0]$ and $[1-\rho_0,1]$.
	By the uniform limits of $U_A$ and $V$, there exists $M>0$
	such that both profiles are at most $\rho_0$ for $x\ge M$
	and at least $1-\rho_0$ for $x\le-M$.
	For $h\ge2M$, any positive supremum of $U_A-V_h$ is
	attained, since this difference tends uniformly to zero
	at both infinities. At a positive maximum, either $x\ge M$
	and $V_h<U_A\le\rho_0$, or $x<M$ and
	$1-\rho_0\le V_h<U_A$.
	Thus, $f(V_h)>f(U_A)$, contradicting
	\eqref{eq:example-sliding-inequality}.
	This proves $U_A\le V_h$ for $h\ge2M$.
	
	Since $V_h$ increases with $h$ and tends pointwise to zero
	as $h\to-\infty$, the critical shift
	\begin{equation*}
			h_*:=\inf\{h\in\R:U_A\le V_h\text{ in }\R\times\T\}
	\end{equation*}
	is finite. Continuity gives $U_A\le V_{h_*}$.
	Equality at any point is excluded by
	\eqref{eq:example-sliding-inequality}, whose right-hand
	side would then be strictly positive.
	Therefore, $U_A<V_{h_*}$ everywhere.
	
	Enlarge $M$ so that, whenever $|h-h_*|\le1$, both $U_A$
	and $V_h$ belong to $[0,\rho_0]$ for $x\ge M$ and to
	$[1-\rho_0,1]$ for $x\le-M$.
	For $h<h_*$ sufficiently close to $h_*$, the strict
	inequality $U_A<V_h$ persists on $[-M,M]\times\T$.
	If $U_A\le V_h$ failed, $U_A-V_h$ would attain a positive
	maximum outside this cylinder, where the strict decrease
	of $f$ again contradicts
	\eqref{eq:example-sliding-inequality}.
	Thus $U_A\le V_h$ for such $h<h_*$, contrary to the
	definition of $h_*$.
	This proves $\gamma_A\le\lambda+m$.
	
	For the reverse bound, suppose by contradiction that $\lambda-m>\gamma_A$.
	At any local maximum of $V_h-U_A$, the other side of
	\eqref{eq:example-relative-residual} similarly gives
	\begin{equation*}
		0\ge f(U_A)-f(V_h)
		-(\lambda-m-\gamma_A)(U_A)_x.
	\end{equation*}
	The same maximum argument gives $V_h\le U_A$ for all
	sufficiently negative $h$.
	Since $V_h\to1$ pointwise as $h\to+\infty$, there is a
	finite largest shift for which this ordering holds.
	Repeating the preceding argument shows that the ordering
	persists for slightly larger shifts, a contradiction.
	Hence $\lambda-m\le\gamma_A\le\lambda+m$, as required.
\end{proof}

\begin{proof}[Proof of Proposition~\ref{prop:opposite-signs}]
	\medskip\noindent
	{\it Step 1. Transverse functions.} We first have
		\begin{equation}
		\sum_{j=1}^4|\alpha_s^{(j)}(y)|>0,~~~y\in\T,~s\in[-1,1].
		\label{eq:example-finite-type}
	\end{equation}
	Indeed, if $\alpha_s''(y)=\alpha_s^{(4)}(y)=0$, then
	$\cos(2\pi y)=0$ and $s=0$, so $\alpha_s'(y)\ne0$. This proves \eqref{eq:example-finite-type}.
	Fix $s\in[-1,1]$, write $\alpha=\alpha_s$, and let $\chi$
	be the mean-zero periodic solution of $\chi''=\alpha$.
	Then,
	\begin{equation*}
			\chi(y)=-\frac{\cos(2\pi y)}{4\pi^2}
		-\frac{s\cos(4\pi y)}{16\pi^2}.
	\end{equation*}
	A direct calculation further gives
	\begin{equation}
		\int_\T|\chi'|^2\,dy
		=-\int_\T\alpha\chi\,dy=\frac{4+s^2}{32\pi^2},
		\qquad
		\int_\T\alpha\chi^2\,dy=\frac{3s}{128\pi^4}.
		\label{eq:example-coefficients}
	\end{equation}
	Set
	\begin{equation}
		\Gamma:=-\frac{\displaystyle\int_\T\alpha\chi^2\,dy}
		{5\displaystyle\int_\T|\chi'|^2\,dy}
		=-\frac{3s}{20\pi^2(4+s^2)}.
		\label{eq:example-explicit-coefficients}
	\end{equation}
	Define the mean-zero periodic functions $\psi,\rho,\sigma$ by
	\begin{equation}\label{cell eqn}
		\psi''=\alpha\chi+\frac{4+s^2}{32\pi^2},\qquad
		\rho''=\alpha\psi-\frac{3s}{128\pi^4},\qquad
		\sigma''=-\chi,
	\end{equation}
	where the right-hand sides all have zero mean thanks to
	\eqref{eq:example-coefficients}, the zero mean of $\chi$, as well as
	\begin{equation*}
		\int_\T\alpha\psi\,dy
		=\int_\T\chi\psi''\,dy
		=\int_\T\alpha\chi^2\,dy=\frac{3s}{128\pi^4}.
	\end{equation*}
	The functions $\chi,\psi,\rho,\sigma$ are trigonometric
	polynomials with coefficients polynomial in $s$, so their
	derivatives of each fixed order are uniformly bounded.
	All constants below are independent of $s\in[-1,1]$,
	$A\ge1$, and sufficiently small $\eps>0$.
	
	\medskip\noindent
	{\it Step 2. Construction of a comparison profile.}
	Write $h(u)=u(1-u)$ and $g(u)=-h(u)h'(u)$, so that
	$f_\eps=\eps g+\eps^2h$, and set
	\begin{equation*}
		q_\eps(u)=\frac{4\pi\sqrt{2\eps}}{\sqrt{4+s^2}}h(u)
		-\frac{72s\eps}{5(4+s^2)^2}h(u)h'(u).
	\end{equation*}
	For $\eps_0$ small enough,
	$c\sqrt\eps\,h\le q_\eps\le C\sqrt\eps\,h$ on $[0,1]$.
	By separation of variables, the solution of
	\begin{equation*}
		\Phi'=-q_\eps(\Phi)~~~~\text{with}~~ \Phi(0)=\frac12,
	\end{equation*}
	is defined on $\R$, takes values in $(0,1)$, and decreases
	from $1$ to $0$. Differentiating $\Phi'=-q_\eps(\Phi)$ repeatedly and
	using the explicit form of $q_\eps$, we obtain by induction
	\begin{equation}
		|\Phi^{(j)}|\le C_j\eps^{j/2}h(\Phi)~~~\text{for}~~1\le j\le5,
		\qquad
		-\Phi'\ge c\sqrt\eps\,h(\Phi).
		\label{eq:example-weighted-derivatives}
	\end{equation}
	Using the explicit formula for $q_\eps$, we compute
	$\Phi''$ and $\Phi'''$ from $\Phi'=-q_\eps(\Phi)$
	and get
	\begin{equation}
		\left|
		\frac{4+s^2}{32\pi^2}\Phi''+\eps g(\Phi)
		-\frac{3s}{128\pi^4}\Phi'''+\eps\Gamma\Phi'
		\right|
		\le C\eps^2h(\Phi)
		\quad\text{on }\R.
		\label{eq:example-scalar-balance}
	\end{equation}
	
	For $(x,y)\in\R\times\T$, define 
	\begin{equation}
		V=\Phi+\chi\Phi'+\psi\Phi''+\rho\Phi'''
		+\eps\sigma g'(\Phi)\Phi'.
		\label{eq:example-test-function}
	\end{equation}
	By \eqref{eq:example-weighted-derivatives}, for all
	sufficiently small $\eps>0$,
	\begin{equation}
		|V-\Phi|\le C\sqrt\eps\,h(\Phi),
		\qquad
		|V_x-\Phi'|+|V_{xx}|\le C\eps h(\Phi), \qquad\text{on }\R\times\T,
		\label{eq:example-test-estimates}
	\end{equation}
	Since $h(\Phi)\le\min\{\Phi,1-\Phi\}$, it follows that, up to decreasing $\eps_0$
	if necessary,
	\begin{equation}
		0<V<1,\qquad -V_x\ge c\sqrt\eps\,h(\Phi)>0, \qquad\text{on }\R\times\T,
		\label{eq:example-admissibility}
	\end{equation}
	and $V(-\infty,\cdot)=1$, $V(+\infty,\cdot)=0$
	uniformly in $\T$.
		
	Set $R_\eps:=V_{yy}+(\eps\Gamma-\alpha)V_x+f_\eps(V)$ on $\R\times\T$. Using $\chi''=\alpha$ and \eqref{cell eqn}, we get
	\begin{align*}
		V_{yy}
		={}&\alpha(\Phi'+\chi\Phi''+\psi\Phi''')
		+\frac{4+s^2}{32\pi^2}\Phi''-\frac{3s}{128\pi^4}\Phi'''
		-\eps\chi g'(\Phi)\Phi',
	\end{align*}
	Moreover, since $f_\eps=\eps g+\eps^2h$, Taylor's formula and
	\eqref{eq:example-weighted-derivatives} yield as $\eps\to0^+$,
	\begin{equation*}
			g(V)=g(\Phi)+\chi g'(\Phi)\Phi'+O(\eps h(\Phi)),
		\qquad h(V)=O(h(\Phi)),
	\end{equation*}
	uniformly in $(x,y)\in \R\times\T$.
	Combining these identities with the derivative of
	\eqref{eq:example-test-function} gives that
	\begin{equation*}
		R_\eps=\frac{4+s^2}{32\pi^2}\Phi''+\eps g(\Phi)
		-\frac{3s}{128\pi^4}\Phi'''+\eps\Gamma\Phi'
		+O(\eps^2h(\Phi))~~~~\text{as}~~\eps\to0^+,
	\end{equation*}
	uniformly in $(x,y)\in \R\times\T$.
	Hence, \eqref{eq:example-scalar-balance} implies that,
for all sufficiently small $\eps>0$,
	\begin{equation}
		|R_\eps(x,y)|\le C\eps^2h(\Phi(x)),
		\qquad (x,y)\in\R\times\T.
		\label{eq:example-full-residual}
	\end{equation}
	
	\medskip\noindent
	{\it Step 3. Speed estimates.}
	By \eqref{eq:example-test-estimates},
	\eqref{eq:example-admissibility}, and
	\eqref{eq:example-full-residual}, it follows that for every $A>0$
	and all sufficiently small $\eps>0$,
	\begin{equation*}
			|A^{-2}V_{xx}+R_\eps|
		\le C(A^{-2}\eps+\eps^2)h(\Phi)
		\le C(A^{-2}\sqrt\eps+\eps^{3/2})(-V_x),  \qquad\text{on }~\R\times\T.
	\end{equation*}
	Lemma~\ref{lem:example-speed-comparison} therefore yields
	\begin{equation*}
		\left|\frac{c_A(\alpha_s,f_\eps)}{A}-\eps\Gamma\right|
		\le C\big(\eps^{3/2}+A^{-2}\sqrt\eps\big),
		\qquad A\ge1.
	\end{equation*}
	Letting $A\to+\infty$ with $\eps$ fixed, and using
	\eqref{eq:speed-limit} and
	\eqref{eq:example-explicit-coefficients}, proves
	\eqref{eq:example-speed-expansion}.
	In particular,
	\begin{align*}
		\gamma^*(\alpha_1,f_\eps)
		=-\frac{3\eps}{100\pi^2}+O(\eps^{3/2}),~~~~~
		\gamma^*(\alpha_{-1},f_\eps)
		=\frac{3\eps}{100\pi^2}+O(\eps^{3/2}).
	\end{align*}
	Choosing $\eps_0$ so that $C\sqrt{\eps_0}<3/(200\pi^2)$
	gives the asserted signs.
	Finally, \eqref{eq:speed-Lipschitz} gives
	\begin{equation*}
			|\gamma^*(\alpha_s,f_\eps)-\gamma^*(\alpha_t,f_\eps)|
		\le\|\alpha_s-\alpha_t\|_{L^\infty(\T)}=|s-t|.
	\end{equation*}
	The intermediate value theorem yields
	$s_\eps\in(-1,1)$ with
	$\gamma^*(\alpha_{s_\eps},f_\eps)=0$. This completes the proof.
\end{proof}

\bigskip 
\noindent
{\bf Declaration on the Use of AI.} The human authors initiated and led the research program, formulated the mathematical problem, developed the main results, and independently verified all arguments. For the example constructed in Proposition 1.5, the human authors provided the main mathematical idea and strategy, while discussions with ChatGPT assisted in exploring and refining suitable ansatzes for the comparison profiles and the associated correctors. ChatGPT also assisted in improving notation, grammar, and presentation. The final manuscript was carefully revised and approved by the human authors, who take full responsibility for its content.

\bigskip 
\noindent
{\bf Conflict of interest statement.} On behalf of all authors, the corresponding author states that there is no conflict of interest.

\bigskip 
\noindent
{\bf Data availability statement.} No data were generated or analyzed in this study.

\end{document}